%% file: cellular_stratified_space.tex
\documentclass{article}

\usepackage[a4paper]{geometry}

\usepackage{bm}

\usepackage{amsmath}
\usepackage{amssymb}
\usepackage{amsthm}

\usepackage{tikz}
\usetikzlibrary{matrix,arrows}

\usepackage[all]{xy}

\usepackage{pb-diagram}
\usepackage{pb-xy}

\input{my_environments}
\input{my_theorems}

\input{my_categories}

\input{my_symbols}

\newcommand{\tot}{\mathrm{tot}}

\title{\textbf{Cellular Stratified Spaces}}
\author{Dai Tamaki}

\begin{document}

\maketitle

\begin{center}
 Dedicated to Professor Fred Cohen on the occasion of his 70th
 birthday. 
\end{center}

\begin{abstract}
 The notion of cellular stratified spaces was introduced in
 a joint work of the author with Basabe, Gonz{\'a}lez, and Rudyak
 with the aim of constructing a cellular model of the  
 configuration space of a sphere. Although the original aim was not
 achieved in the project, the notion of cellular stratified spaces turns
 out to be useful, at least, in the study of configuration spaces of graphs.
 In particular, the notion of totally normal cellular stratified spaces
 was used successfully in a joint work with the former students of the
 author \cite{1312.7368} to study the homotopy type of
 configuration spaces of graphs with a small number of vertices.

 Roughly speaking, totally normal cellular stratified spaces correspond
 to acyclic categories in the same way regular cell complexes correspond
 to posets. 

 In this paper, we extend this correspondence by replacing cells by
 stellar cells and acyclic categories by topological acyclic categories.
\end{abstract}

\tableofcontents

\input{cellular_stratified_space_intro}

\input{stratification_and_cell}

\input{discrete_face_category}

\input{topological_face_category}

\input{BC}

\input{basic_operations}

\input{duality}

\appendix

\input{quotient_map}

\input{PL}

\input{topological_category}

\bibliographystyle{halpha}
\input{biblist}

\end{document}

%% file: my_environments.tex
   {\begin{proof}[Sketch of Proof of #1]}%
   {\end{proof}}


%% file: my_theorems.tex
\theoremstyle{plain}
  \newtheorem{theorem}{Theorem}[section]
  \newtheorem{corollary}[theorem]{Corollary}
  
  \newtheorem{lemma}[theorem]{Lemma}
  
  \newtheorem{proposition}[theorem]{Proposition}

\theoremstyle{definition}
  \newtheorem{definition}[theorem]{Definition}

  \newtheorem{ex}[theorem]{Example}
  
  \newtheorem{problem}[theorem]{Problem}

  \newtheorem{remark}[theorem]{Remark}

  \newenvironment{example}{\begin{ex}}{\qed\end{ex}}

%% file: my_categories.tex
\newcommand{\category}[1]{\mathbf{#1}}

  \newcommand{\Sets}{\operatorname{\mathbf{Sets}}}

  \newcommand{\Spaces}{\category{Spaces}}

  \newcommand{\Map}{\operatorname{Map}}

  \newcommand{\Funct}{\operatorname{Funct}}

\newcommand{\coequalizer}[5]{\xymatrix{
 {\displaystyle #1} 
 \ar@<1ex>[r]^-{#2} \ar@<-1ex>[r]_-{#3}
 & {\displaystyle #4} \ar[r] & {\displaystyle #5}
 }%
}
\newcommand{\equalizer}[5]{
\xymatrix{
 {\displaystyle #1} \ar[r] & {\displaystyle #2}
 \ar@<1ex>[r]^-{#3} \ar@<-1ex>[r]_-{#4}
 & {\displaystyle #5} 
 }%
}
  
  \newcommand{\Ima}{\operatorname{Im}}

%% file: my_symbols.tex
  \DeclareMathOperator*{\colim}{\mathrm{colim}}

  \newcommand{\quotient}[2]{%
  \left(#1\right)
  \hspace{-4pt}\raisebox{-5pt}{$\bigg/$}\hspace{-2pt}\raisebox{-12pt}{$#2$}%
  }
\newcommand{\rset}[2]{%
\left\{#1 \ \left| \ #2\right\}\right.
}
\newcommand{\lset}[2]{%
\left.\left\{#1 \ \right| \ #2\right\}
}
\newcommand{\set}[2]{%
\left\{#1 \,\middle|\, #2\right\}
}

  \newcommand{\rarrow}[1]{\buildrel #1 \over \longrightarrow}

  \newcommand{\smallfrac}[2]{{\textstyle \frac{#1}{#2}}}

  \newcommand{\ev}{{\operatorname{\mathrm{ev}}}}
  
  \newcommand{\pr}{\operatorname{\mathrm{pr}}}

  \newcommand{\transpose}[1]{\raisebox{1ex}{$\scriptstyle t$}\kern-0.2ex #1}

  \newcommand{\Int}{{\operatorname{\mathrm{Int}}}}

  \newcommand{\R}{\mathbb{R}}
  \newcommand{\bbC}{\mathbb{C}}

  \newcommand{\Q}{\mathbb{Q}}
  
  \newcommand{\Z}{\mathbb{Z}}

  \newcommand{\Sal}{\operatorname{\mathrm{Sal}}} 
  \newcommand{\Lk}{\mathrm{Lk}} 

  \newcommand{\cyl}{\mathrm{cyl}} 

  \newcommand{\op}{{\operatorname{\mathrm{op}}}} 
  \newcommand{\ad}{\operatorname{\mathrm{ad}}} 

  \newcommand{\sign}{\operatorname{\mathrm{sign}}}

  \newcommand{\inj}{\operatorname{\mathrm{inj}}}

  \newcommand{\GL}{\mathrm{GL}}

  \newcommand{\String}{\mathrm{\String}}

  \newcommand{\St}{\mathrm{St}}
  \newcommand{\Gr}{\mathrm{Gr}}
  \newcommand{\RP}{\R\mathrm{P}}
  \newcommand{\CP}{\bbC\mathrm{P}}

  \newcommand{\Conf}{\mathrm{Conf}}

\newcommand{\sk}{\operatorname{sk}}

\newcommand{\Sd}{\operatorname{\mathrm{Sd}}}

\newcommand{\fixhyperref}{%
\ifnum 42146=\euc"A4A2 \AtBeginDvi{}\else
\AtBeginDvi{}\fi}

%% file: cellular_stratified_space_intro.tex
\section{Introduction}
\label{cellular_stratified_space_intro}

The author has been interested in configuration spaces since he was a
graduate student at the University of Rochester under the guidance of
Professor Fred Cohen. The interest arose when the author tried to
understand the global structure of homotopy groups of spheres in his
Ph.D.~thesis \cite{TamakiThesis,Tamaki94,Tamaki1994-2}.
One of key aspects is the combinatorial structures underlying in the
configuration spaces of Euclidean spaces.

Recently the author renewed his interest in configuration spaces during
a joint project with Basabe, Gonz{\'a}lez, and Rudyak \cite{1009.1851}
on higher symmetric topological complexities.
The notion of cellular stratified spaces was discovered during the
discussion with them. 

It turns out that cellular stratified spaces have already appeared in
many areas in topology. For example, the stratification on the
complement of a complexified hyperplane arrangement used in the
construction of the Salvetti complex
\cite{Salvetti87,Bjorner-Ziegler92,DeConcini-Salvetti00} 
is one of motivating examples. 



\input{cellular_stratified_space_examples}

\subsection{Statements of Results}

The aim of this paper is to extend these results and develop the
theory of cellular stratified spaces. For this purpose, we first
introduce the notion of cylindrical structures on cellular stratified 
spaces. 

\begin{definition}[Definition \ref{cylindrical_structure_definition}]
 A \emph{cylindrical structure} on a normal cellular stratified space
 $X$ consists of
 \begin{itemize}
  \item a normal stratification on $S^{n-1}$
	containing $\partial D_{\lambda}$ as a stratified subspace for
	each $n$-cell
	$\varphi_{\lambda} : D_{\lambda}\to\overline{e_{\lambda}}$ in
	$X$,  
  \item a stratified space $P_{\mu,\lambda}$ and a morphism of
	stratified spaces 
	\[
	 b_{\mu,\lambda} : P_{\mu,\lambda}\times D_{\mu} \longrightarrow
	\partial D_{\lambda}
	\]
	for each pair of cells
	$e_{\mu} \subset \partial e_{\lambda}$, and
  \item  a morphism of stratified spaces
	\[
	 c_{\lambda_0,\lambda_1,\lambda_2} :
	 P_{\lambda_1,\lambda_2}\times P_{\lambda_0,\lambda_1} 
	 \longrightarrow P_{\lambda_0,\lambda_2}
	\]
	 for each sequence $\overline{e_{\lambda_0}} \subset
	\overline{e_{\lambda_1}} \subset \overline{e_{\lambda_2}}$,
 \end{itemize}
 satisfying certain compatibility and associativity conditions.
 A cellular stratified space equipped with a cylindrical structure is
 called a \emph{cylindrically normal} cellular stratified space.
\end{definition}

Examples of cylindrically normal cellular stratified spaces 
include
\begin{itemize}
 \item totally normal cellular stratified spaces,
 \item PLCW complexes,
 \item the minimal cell decomposition of $\CP^n$, and
 \item the geometric realization of simplicial sets.
\end{itemize}
The cylindrical structure on $\CP^n$ can be defined by using the
``moduli spaces of flows'' of a Morse-Smale function, as is done in
the preprint \cite{Cohen-Jones-SegalMorse} by R.~Cohen, J.D.S~Jones, and
G.B.~Segal mentioned above.
Alternatively we can construct the same 
cylindrical structure by identifying $\CP^n$ with the
Davis-Januszkiewicz construction $M_{\lambda_n}$\footnote{See Example
\ref{moment_angle_complex} for more details.}. This observation suggests
a large class of quasitoric and torus manifolds have cylindrically
normal cell decompositions. 

Given a cylindrically normal cellular stratified space $X$, we 
define a topological acyclic category $C(X)$, called the
cylindrical face category of $X$. Objects are cells in $X$ and the space 
of morphisms from $e_{\mu}$ to $e_{\lambda}$ is defined to be
$P_{\mu,\lambda}$. 

Our first result says that the classifying space $BC(X)$ of $C(X)$ can
be always embedded in $X$.

\begin{theorem}[Theorem \ref{embedding_Sd}]
 \label{main1}
 For any cylindrically normal cellular stratified space $X$, there exists
 an embedding
 \[
  i_{X} : BC(X) \hookrightarrow X
 \]
 which is natural with respect to morphisms of cellular stratified spaces.
 Furthermore, when all cells in $X$ are closed, $i_{X}$ is a
 homeomorphism. 
\end{theorem}

When $X$ contains non-closed cells, $i_{X}$ is not a homeomorphism. 
Our second result says that, under a reasonable condition, those
non-closed cells can be collapsed into $BC(X)$.

\begin{theorem}[Theorem \ref{deformation_theorem}] 
 \label{main2}
 For a polyhedral\footnote{Definition
 \ref{polyhedral_normality_definition}} cellular stratified space $X$,
 the image of the embedding $i_{X} : BC(X) \hookrightarrow X$ is a
 strong deformation retract of $X$. The deformation retraction can be
 taken to be natural with respect to morphisms of polyhedral cellular 
 stratified spaces. 
\end{theorem}

It turns out that they still hold when we replace cells by
``star-shaped'' cells. We introduce the notion 
of stellar stratified spaces and show that the functor $BC(-)$
transforms cellular stratified spaces to stellar stratified spaces,
giving us a dualizing operation $D(-)$. With this structure, Theorem
\ref{main1} (Theorem \ref{embedding_Sd}) can be rephrased as follows.

\begin{theorem}[Theorem \ref{Salvetti}]
 For a totally normal stellar stratified space $X$, we have an embedding
 of stellar stratified spaces
 \[
  D(D(X)) \hookrightarrow X,
 \]
 which is an isomorphism when all cells in $X$ are closed. 
\end{theorem}

It should be worthwhile noting that, when $X$ is the complement of the
complexification of a real hyperplane arrangement $\mathcal{A}$, $DD(X)$
coincides with the Salvetti complex of $\mathcal{A}$. This fact and
examples of cylindrical structures suggest that we may apply these
theorems to the following problems:
\begin{enumerate}
 \item Construct combinatorial models for configuration spaces and
       apply them to study the homotopy types of configuration spaces. 
 \item Develop a refinement and an extension of the Cohen-Jones-Segal
       Morse theory.  
 \item Reformulate Forman's discrete Morse theory
       \cite{Forman95,Forman98-2} in terms of topological 
       acyclic categories.
\end{enumerate}
We should also be able to apply the results in this paper to other types
of configurations and arrangements. We might be able to apply the
results to toric topology. 

\subsection{Organization}

The article is organized as follows.
\begin{itemize}
 \item \S\ref{stratification_and_cell} is preliminary. We fix notations
       and terminologies for stratified spaces in
       \S\ref{stratified_space}. The definition of cell structures is
       given in \S\ref{cell_structure}.
       Cellular stratified spaces are introduced in
       \S\ref{cellular_stratified_space_definition}. 
       We introduce stellar stratified spaces in
       \S\ref{stellar_stratified_space}. 

 \item We extend the regularity and normality for cell complexes in
       \S\ref{discrete_face_category}. 
       The regularity and normality are extended to cellular stratified
       spaces in \S\ref{regularity_and_normality}. In
       \S\ref{total_normality}, the
       definition and basic properties of totally normal cellular
       stratified spaces are recalled from \cite{1312.7368}.
       \S\ref{total_normality_examples} is devoted to examples of
       totally normal cellular stratifies spaces.

 \item A new structure, called
       cylindrically normal cellular and stellar stratifications, is
       introduced and studied in
       \S\ref{topological_face_category}.
       After defininng cylindrically normal cellular stratified spaces
       in \S\ref{cylindrical_structure}, we impose piecewise-linear
       structure on each cell and introduce polyhedral structures in
       \S\ref{locally_polyhedral}. We review examples of
       cylindrically normal cellular stratified spaces in
       \S\ref{cylindrical_examples}. 

 \item Theorem \ref{main1} and \ref{main2} are proved in \S\ref{BC}.
       
 \item For applications to configuration spaces, we need to understand
       basic operations on cellular stratified spaces, which is the
       subject of \S\ref{basic_operations}.
       We study three kinds of operations; stratified subspaces in
       \S\ref{stratified_subspace}, products in \S\ref{product}, and 
       subdivisions in \S\ref{subdivision_of_cell}. 

 \item As an extension of the barycentric sudivision of regular cell
       complexes, our framework is suitable for discussing duality,
       which is the subject of \S\ref{duality}.
\end{itemize}

\noindent The paper contains three appendices for the convenience of the
reader. 

\begin{itemize}
 \item Understanding the behavior of quotient maps is important in this
       paper, since cell structure maps are required to be quotient
       maps. We summerize important properties of quotient maps in 
       Appendix \ref{quotient_map}. 

 \item In Appendix \ref{PL}, we recall definitions and properties of
       simplicial complexes, simplicial sets, and related structures.

 \item Appendix \ref{topological_category} is a summary on basics of
       topological categories, including their classifying spaces. 
\end{itemize}

\subsection{Acknowledgments}

This work originates in fruitful and enjoyable discussion with
Jes{\'u}s Gonz{\'a}lez on the Google Wave system in 2010. Although
Google terminated 
the development of the system, I would like to thank Google for
providing us with such a useful collaboration tool for free. I would
also like to express my appreciation to the Centro 
di Ricerca Matematica Ennio De Giorgi, Scuola Normale Superiore di 
Pisa, for supporting my participation in the research program
``Configuration Spaces: Geometry, Combinatorics and Topology'' in 2010,
during which a part of this work was done. I would also like to thank
the National University of Signapore for inviting me to the
workshop ``Combinatorial and Toric Topology'' in 2015.

The discussion on the Google Wave has been incorporated into a
joint project with Ibai Basabe and Yuli Rudyak with the aim of applying
cellular stratifies spaces to study the homotopy type of configuration
spaces of spheres. It turns out that configuration spaces of spheres are
much more complicated than we imagined and the sections for cellular
stratified spaces were removed from the published paper
\cite{1009.1851}. Nonetheless the discussion during the project with
them has been essential. I would like to thank Basabe, Gonz{\'a}lez, and
Rudyak, as well as Peter Landweber, who read earlier versions of the
paper and made useful comments.

Before I began the discussion with Gonz{\'a}lez on Google Wave, some 
of the ideas have been already developed during the discussion with
my students. Takamitsu Jinno and Mizuki Furuse worked
on Hom complexes in 2009 and configuration spaces of graphs in 2010,
respectively, in their master's theses. The possibility of
finding a better combinatorial model for configuration spaces was
suggested by their work. Furuse's work was developed further by another
studient Takashi Mukouyama. Their work is contained in a paper
\cite{1312.7368}, 
in which theory of totally normal cellular stratified spaces has been
developed. 
The connection with Cohen-Jones-Segal Morse
theory, which resulted in the current definition of cylindrical
structure, was discovered during the discussion with another former
student, Kohei Tanaka. I am grateful to all these former students.

Mikiya Masuda pointed out my misunderstanding in examples concerning
complex projective spaces in an early draft of this paper. Priyavrat
Deshpande pointed out a mistake in \cite{1312.7368} on subdivisions of
totally normal cellular stratified spaces. The mistake is corrected in
this paper as Proposition \ref{subdivision_of_totally_normal_css}. I
would like to thank both of them.

This work is supported by JSPS KAKENHI Grant Number 23540082 and
15K04870.

%% file: cellular_stratified_space_examples.tex
\subsection{Cellular Stratified Spaces Everywhere}
\label{cellular_stratified_space_examples}

Before we state main results, let us take a look at examples of cellular
stratified spaces. We begin with configuration spaces.
The configuration space of $n$ distinct points in a topological space
$X$ is defined by
\begin{eqnarray*}
 \Conf_n(X) & = & \set{(x_1,\ldots,x_n)\in X^n}{x_i\neq x_j \text{ for }
  i\neq j} \\
 & = & X^n\setminus \Delta_{n}(X),
\end{eqnarray*}
where
\[
 \Delta_n(X) = \bigcup_{1\le i<j\le n} \set{(x_1,\ldots,x_n)\in
 X^n}{x_i=x_j}. 
\]

The starting point of this work is the following problem:

\begin{problem}
 \label{main_problem}
 Given a space $X$, construct a combinatorial model for the homotopy
 type of the configuration space $\Conf_k(X)$ of $k$ distinct points in
 $X$. In other words, find a regular cell complex or a simplicial complex
 $C_k(X)$ embedded in $\Conf_k(X)$ as a $\Sigma_k$-equivariant
 deformation retract. 
\end{problem}

Several solutions are known in special cases.

\begin{example}
 \label{Abrams}
 For a finite CW-complex $X$ of dimension $1$, namely a graph, Abrams
 constructed a subspace $C_k^{\mathrm{Abrams}}(X)$ contained in
 $\Conf_k(X)$ in his thesis \cite{AbramsThesis} and proved that there is
 a homotopy equivalence
 \[
  C_k^{\mathrm{Abrams}}(X) \simeq \Conf_k(X)
 \]
 as long as the following two conditions are satisfied:
 \begin{enumerate}
  \item each path connecting vertices in $X$ of valency more than $2$
	has length at least $k+1$, and 
  \item each homotopically essential path connecting a vertex to itself
	has length at least $k+1$.
 \end{enumerate}
 Here a path means a $1$-dimensional subcomplex homeomorphic to a closed
 interval.
\end{example}

\begin{example}
 \label{Salvetti_complex}
 Consider the case $X=\R^n$. For $1\le i<j\le k$, the hyperplane 
 \[
  H_{i,j} = \lset{(x_1,\ldots,x_k)\in \R^k}{x_i=x_j}
 \]
 in $\R^k$ defines a linear subspace $H_{i,j}\otimes\R^n$ in
 $\R^k\otimes\R^n=\underbrace{\R^n\times\cdots\times \R^n}_k=X^k$
 and we have  
 \[
 \Conf_k(\R^n) = \R^k\otimes\R^n \setminus
 \bigcup_{1\le i<j\le k} H_{i,j}\otimes\R^n. 
 \]
 The collection $\lset{H_{i,j}}{1\le i<j\le k}$ is
 called the \emph{braid arrangement} of rank $k-1$ and is denoted by
 $\mathcal{A}_{k-1}$. 

 When $n=2$, the construction due to Salvetti \cite{Salvetti87} gives us
 a regular cell complex $\Sal(\mathcal{A}_{k-1})$ embedded in
 $\Conf_k(\R^2)$ as a $\Sigma_k$-equivariant deformation retract.

 More generally, the construction sketched at the end of
 \cite{Bjorner-Ziegler92} by Bj{\"o}rner and Ziegler and elaborated in
 \cite{DeConcini-Salvetti00} by De Concini and Salvetti gives us a
 regular cell complex $\Sal^{(n)}(\mathcal{A}_{k-1})$ embedded in
 $\Conf_k(\R^n)$ as a $\Sigma_k$-equivariant deformation retract. 

 This construction is a special case of the construction of a
 regular cell complex whose homotopy type represents the complement of
 the subspace arrangement associated with a real hyperplane arrangement.
\end{example}

There are pros and cons in these two constructions. The conditions in
Abrams' theorem require us to subdivide a given
$1$-dimensional CW-complex finely. For example, his construction fails
to give the right homotopy type of the configuration space
$\Conf_2(S^1)$ of two points in $S^1$ when
it is applied to the minimal cell decomposition;
 $S^1 = e^0 \cup e^1$. The minimal regular cell decomposition
 $S^1 = e^0_-\cup e^0_+ \cup e^1_-\cup e^1_+$ is not fine enough,
 either. We need to subdivide $S^1$ into three $1$-cells to use Abrams' model.
\begin{center}
 \begin{tikzpicture}
  \draw (0,0) circle (1cm);
  \draw [fill] (1,0) circle (2pt);
  \draw (-1.3,0) node {$e^1$};
  \draw (1.5,0) node {$e^0$};

  \draw (4,0) circle (1cm);
  \draw [fill] (3,0) circle (2pt);
  \draw [fill] (5,0) circle (2pt);
  \draw (2.7,0) node {$e^0_-$};
  \draw (5.5,0) node {$e^0_+$};
  \draw (4,-1.3) node {$e^1_-$};
  \draw (4,1.3) node {$e^1_+$};

  \draw (8,0) circle (1cm);
  \draw [fill] (9,0) circle (2pt);
  \draw [fill] (7.5,0.86) circle (2pt);
  \draw [fill] (7.5,-0.86) circle (2pt);
  \draw (6.7,0) node {$e^1_1$};
  \draw (8.6,-1.2) node {$e^1_2$};
  \draw (8.6,1.2) node {$e^1_3$};
  \draw (9.5,0) node {$e^0_1$};
  \draw (7.4,1.2) node {$e^0_2$};
  \draw (7.4,-1.2) node {$e^0_3$};
 \end{tikzpicture}
\end{center}
Another problem is that his theorem is restricted to $1$-dimensional
CW-complexes, although the construction of the model itself works for
any cell complex\footnote{Recently higher dimensional cases
appeared in \cite{1009.2935}.}. 

The second construction suggests that we should consider more
general stratifications than cell decompositions. The complex
$\Sal^{(n)}(\mathcal{A}_{k-1})$ is constructed from the combinatorial
structure of the ``cell decomposition'' of $\R^k\otimes\R^n$ defined by
the hyperplanes in the arrangement $\mathcal{A}_{k-1}$ together with the
standard framing in $\R^n$. ``Cells'' in this decomposition are
unbounded regions in $\R^k\otimes\R^n$. Although such a decomposition is
not regarded as a cell decomposition in the usual sense, we may extend
the definition of face posets to such generalized cell
decompositions. And the complex $\Sal^{(n)}(\mathcal{A}_{k-1})$ is
constructed in terms of the combinatorial structure of the face poset of
the ``cell decomposition'' of $\R^k\otimes\R^n$.
The crucial deficiency of the second construction is, however, that it
works only for Euclidean spaces. 

One of the motivations of this paper is to find a common framework for
working with configuration spaces and complements of
arrangements. Although there are many interesting ``parallel theories''
between configuration spaces and arrangements, e.g.\
the Fulton-MacPherson-Axelrod-Singer compactification
\cite{Fulton-MacPherson94, Axelrod-Singer94-2} and the De
Concini-Procesi wonderful model \cite{DeConcini-Procesi95}, there is no
``Salvetti complex'' for configuration spaces in general.
A more concrete motivation is, therefore, to solve Problem
\ref{main_problem} in such a way that it generalizes the Salvetti
complex for the braid arrangement. 

By analyzing the techniques of combinatorial algebraic topology used in
the proof of Salvetti's theorem, the notion of cellular stratified
spaces was introduced in \cite{1009.1851v5}.
It turns out that, in the case of configuration spaces of spheres, which
was the main target of study in the project, it was not easy to use cellular
stratified spaces to construct a combinatorial model. The section for
cellular stratified spaces was removed from the published version
\cite{1009.1851}. 

However, the theory of cellular stratified spaces can be used to study
configuration spaces of graphs, as is done in \cite{1312.7368},
in which the notion of totally normal cellular stratified spaces played
a central role.

\begin{definition}[Definition \ref{total_normality_definition}]
 Let $X$ be a normal cellular stratified space.
 $X$ is called \textit{totally normal} if, for each $n$-cell
 $e_{\lambda}$,  
\begin{enumerate}
 \item there exists a structure of regular cell complex on $S^{n-1}$
       containing $\partial D_{\lambda}$ as a stratified subspace, and 
 \item for any cell $e$ in $\partial D_{\lambda}$, there exists a cell
       $e_{\mu}$ in $\partial e_{\lambda}$ such that
       $e_{\mu}$ and $e$ share the same domain and the characteristic
       map of $e_{\mu}$ factors through $D_{\lambda}$ via the
       characteristic map of $e$:
	\[
	\begin{diagram}
	 \node{\overline{e}} \arrow{e,J} \node{\partial D_{\lambda}}
	 \arrow{e,J} \node{D_{\lambda}} 
	 \arrow{e,t}{\varphi_{\lambda}} \node{X} \\ 	 
	 \node{D} \arrow{n} \arrow{e,=} \node{D_{\mu}.}
	 \arrow{ene,b}{\varphi_{\mu}} 
	\end{diagram}		
	\]
\end{enumerate}
 The category $C(X)$ consisting of cells in $X$ as objects and above maps
 $D_{\mu}\to \overline{e}\to D_{\lambda}$ as morphisms is called the
 \emph{face category} of $X$.
\end{definition}

The following result says that we can always recover the homotopy type
of $X$ from its face poset, if $X$ is totally normal.

\begin{theorem}[Theorem 2.50 in \cite{1312.7368}]
 \label{FMT_theorem}
 For a totally normal cellular stratified space $X$,
 the classifying space $BC(X)$ of the face category
 $C(X)$ can be embedded in $X$ as a strong deformation retract.

 When $X=\R^k\otimes\R^n$ and the stratification is given
 by a real hyperplane arrangement $\mathcal{A}$ in $\R^k$ by the method
 of Bj{\"o}rner-Ziegler and De Concini-Salvetti
 \cite{Bjorner-Ziegler92,DeConcini-Salvetti00}, then 
 $BC(X)$ is homeomorphic to the higher order Salvetti
 complex $\Sal^{(n)}(\mathcal{A})$.
\end{theorem}

Note that, if $X$ is a regular cell complex, it
is a fundamental fact in combinatorial algebraic topology that the
classifying space $BF(X)$ of the face poset $F(X)$ is the barycentric
subdivision of $X$ and is homeomorphic to $X$.
The above result is a generalization of this well-known fact.

Examples of totally normal cellular stratified spaces include
\begin{itemize}
 \item regular cell complexes,
 \item the Bj{\"o}rner-Ziegler stratification \cite{Bjorner-Ziegler92}
       of Euclidean spaces defined by subspace arrangements,   
 \item graphs regarded as $1$-dimensional cell complexes,
 \item the minimal cell decomposition of $\RP^n$,
 \item Kirillov's PLCW-complexes \cite{1009.4227} satisfying a certain
       regularity condition, and
 \item the geometric realization of $\Delta$-sets.
\end{itemize}

Another source of inspirations for this paper is a preprint
\cite{Cohen-Jones-SegalMorse} of R.~Cohen, J.D.S.~Jones, and
G.B.~Segal. Given a Morse-Smale function $f : M\to \R$ on a smooth
closed manifold $M$, they constructed a topological acyclic category
$C(f)$ and proved that the classifying space $BC(f)$ is homeomorphic to
$M$. Under a weaker assumption, they also proved that $BC(f)$ is
homotopy equivalent to $X$. Their results strongly suggest that a
``topological acyclic category version'' of Theorem \ref{FMT_theorem}
should exist.

%% file: stratification_and_cell.tex
\section{Stratifications and Cells}
\label{stratification_and_cell}

This section is preliminary. We introduce the notions of stratified
spaces, cell structures, and cellular stratified spaces.
Although the term ``stratified space'' has been used in singularity
theory since its beginning and is a well-established concept, we 
need our own definition of stratified spaces.

\input{stratified_space}
\input{cell_structure}
\input{cellular_stratified_space_definition}

\input{stellar_stratified_space}

%% file: stratified_space.tex
\subsection{Stratified Spaces}
\label{stratified_space}

Before we introduce \emph{cellular} stratified spaces, let us first
recall the notion of stratified spaces in general, whose theory has been
developed in singularity theory. Unfortunately, however, there seems to
be no standard definition of stratified spaces. There are many
non-equivalent definitions in the literature. For this reason, we decided
to examine several books and extract properties for our needs. As our
prototypes, we use definitions in books by Kirwan
\cite{KirwanIntersectionHomology}, Bridson and Haefliger
\cite{Bridson-Haefliger}, Pflaum \cite{Pflaum01}, and Sch{\"u}rmann
\cite{SchuermannBook}.

Here is our reformulation. 

\begin{definition}
 \label{stratification_definition}
 Let $X$ be a topological space and $\Lambda$ be a poset. A
 \emph{stratification} of $X$ indexed by $\Lambda$ is an open continuous 
 map  
 \[
  \pi : X \longrightarrow \Lambda
 \]
 satisfying the condition that,
 for each $\lambda \in \Ima\pi$, $\pi^{-1}(\lambda)$ is connected and
 locally closed\footnote{A subset $A$ of a topological space $X$
 is said to be \emph{locally closed}, if every point $x\in A$ has
 a neighborhood $U$ in $X$ with $A\cap U$ closed in $U$. This
 condition is known to be equivalent to saying that $A$ is an
 intersection of an open and a closed subset of $X$, or $A$ is
 open in $\overline{A}$.}, where $\Lambda$ is topologized by the
 Alexandroff topology\footnote{$D\subset \Lambda$ is closed if and only
 if, for $\lambda\in D$ and $\mu\in\Lambda$, $\mu\le\lambda$ implies
 $\mu\in D$. Or $U\subset\Lambda$ is open if and only if, for
 $\lambda\in U$ and $\mu\in\Lambda$, $\lambda\le\mu$ implies $\mu\in U$.}.

 For simplicity, we put $e_{\lambda}=\pi^{-1}(\lambda)$ and call it a
 \emph{stratum} with index $\lambda$.
\end{definition}

\begin{remark}
 We may safely assume that $\pi$ is surjective. When we define
 morphisms of stratified spaces and stratified subspaces, however, it is
 more convenient not to assume the surjectivity.
\end{remark}

Given a map $\pi : X\to \Lambda$, we have a decomposition of
$X$, i.e. 
\begin{enumerate}
 \item $X = \bigcup_{\lambda\in \Ima \pi} e_{\lambda}$.
 \item For $\lambda,\lambda'\in\Ima\pi$,
       $e_{\lambda}\cap e_{\lambda'}=\emptyset$ if
       $\lambda\neq\lambda'$. 
\end{enumerate}

Thus the image of $\pi$ in the indexing poset $\Lambda$ can be
identified with the set of strata. The definition of stratification can
be rephrased in terms of closures of strata.

\begin{lemma}
 Let $\pi: X\to \Lambda$ be a continuous map from a topological space 
 $X$ to a poset $\Lambda$. Then it is open if and only if 
 the condition
 $e_{\lambda}\subset \overline{e_{\lambda'}}$ is equivalent
 to $\lambda\le \lambda'$ for $\lambda,\lambda'\in\Ima\pi$.
\end{lemma}

\begin{proof}
 It is well known that a map $\pi$ is an open continuous map if and only
 if $\overline{\pi^{-1}(B)}=\pi^{-1}(\overline{B})$ for any subset
 $B\subset\Lambda$. 
 Thus, when $\pi$ is open continuous,
 $\pi^{-1}(\lambda)\subset\overline{\pi^{-1}(\lambda')}$ if and only if
 $\lambda\in \overline{\{\lambda'\}}$, which is equivalent to saying
 $\lambda\le\lambda'$. 

 Conversely suppose that
 $\pi^{-1}(\lambda)\subset\overline{\pi^{-1}(\lambda')}$ is equivalent
 to $\lambda\le \lambda'$. For a subset
 $B\subset\Lambda$, we have
 $\overline{\pi^{-1}(B)}\subset \pi^{-1}(\overline{B})$ by the
 continuity of $\pi$. For
 $x\in \pi^{-1}(\overline{B})$, $\pi(x)\in \overline{B}$. By the
 definition of the Alexandroff topology, there exists $\lambda\in B$
 such that $\pi(x)\le \lambda$. By assumption, this is equivalent to
 $\pi^{-1}(\pi(x)) \subset \overline{\pi^{-1}(\lambda)}$. Thus
 $x \in \overline{\pi^{-1}(B)}$ and we have shown that
 $\overline{\pi^{-1}(B)}\supset \pi^{-1}(\overline{B})$.
\end{proof}

\begin{definition}
 For a stratification $\pi : X \to \Lambda$, the image $\Ima\pi$  
 is called the \emph{face poset} and is denoted by $P(X,\pi)$
 or simply by $P(X)$.

 When $P(X)$ is finite or countable, $(X,\pi)$ is said to be
 \emph{finite} or \emph{countable}.
\end{definition}

\begin{remark}
 The above structure (without the connectivity of $\pi^{-1}(\lambda)$)
 is called a decomposition in Pflaum's 
 book \cite{Pflaum01}. Pflaum used the notion of set germ to define a
 stratification from local decompositions. Furthermore, Pflaum imposed
 three further conditions:
 \begin{itemize}
  \item If $e_{\mu}\cap\overline{e_{\lambda}}\neq\emptyset$,
	then $e_{\mu} \subset \overline{e_{\lambda}}$.
  \item Each stratum is a smooth manifold. 
  \item The collection $\{e_{\lambda}\}_{\lambda\in\Lambda}$ is locally
	finite in the sense 
	that, for any $x\in X$, there exists a neighborhood $U$ of $x$
	such that $U\cap e_{\lambda}\neq\emptyset$ only for a finite
	number of strata $e_{\lambda}$.
 \end{itemize}
 The first condition corresponds to the normality of cell complexes.
 We would like to separate the
 third condition as one of the conditions for CW stratifications.
 For the second condition, as is remarked in his book, we may 
 replace smooth manifolds by any collection of geometric objects such as
 complex manifolds, real analytic sets, polytopes, and so on. In 
 \S\ref{cellular_stratified_space_definition}, we choose
 the class of spaces equipped with ``cell structures'' and define the
 notion of cellular stratified spaces. As an example of another choice,
 we use ``star-shaped cells'' and define the notion of stellar
 stratified spaces in \S\ref{stellar_stratified_space}.
 We also impose  more structures on cells and introduce notions of
 cubical and polyhedral structures in Definition
 \ref{cubical_structure_definition} 
 and Definition \ref{polyhedral_normality_definition}, respectively. 

 Bridson and Haefliger \cite{Bridson-Haefliger} define stratifications
 by using closed strata. Their strata correspond to closures of strata
 in our definition. Furthermore, they also required
 their stratifications to be normal in the following sense.
\end{remark}

\begin{definition}
 \label{normal_stratification}
 We say a stratum $e_{\lambda}$ in a stratified space $(X,\pi)$
 is \emph{normal} if $e_{\mu}\subset\overline{e_{\lambda}}$
 whenever $e_{\mu}\cap \overline{e_{\lambda}}\neq\emptyset$. 
 When all strata are normal, the stratification $\pi$
 is said to be \emph{normal}.
\end{definition}

It is immediate to verify the following.

\begin{lemma}
 A stratum $e_{\lambda}$ is normal if and only if
 $\partial e_{\lambda} = \overline{e_{\lambda}}\setminus e_{\lambda}$ is
 a union of strata. 
\end{lemma}

Another difference between our definition and the one by Bridson and
Haefliger is that they considered intersections of closed strata.

\begin{lemma}
 Let $(X,\pi)$ be a normal stratified space. Then, for any pair
 of strata $e_{\mu}$, $e_{\lambda}$, the intersection
 $\overline{e_{\mu}}\cap\overline{e_{\lambda}}$ is a union of strata.
\end{lemma}

\begin{proof}
 This is obvious, since closures of different strata can intersect
 only on the boundaries, which are unions of strata by the definition of
 normality. 
\end{proof}

The following is a typical example of stratifications we are interested
in. 

\begin{example}
 \label{stratification_by_arrangement}
 Let $S_1=\{-1,0,1\}$ with poset structure $0< \pm 1$. The sign function 
 \[
  \mathrm{sign} : \R\longrightarrow S_1
 \]
 given by 
 \[
  \mathrm{sign}(x) = \begin{cases}
		      +1, & \text{ if } x>0 \\
		      0, & \text{ if } x=0 \\
		      -1, & \text{ if } x<0
		     \end{cases}
 \]
 defines a stratification on $\R$:
 \[
  \R=(-\infty,0)\cup \{0\}\cup (0,\infty).
 \]
 
 This innocent-looking stratification turns out to be one of the most
 important ingredients in the theory of real hyperplane arrangements.
 Let $\mathcal{A}=\{H_1,\ldots, H_k\}$ be a real affine hyperplane
 arrangement in $\R^n$ defined by affine $1$-forms
 $L=\{\ell_1,\ldots,\ell_k\}$. Hyperplanes cut $\R^n$ into convex
 regions that are homeomorphic to the interior of the $n$-disk. Each
 hyperplane $H_i$ is cut into convex regions of dimension $n-1$ by other
 hyperplanes, and so on. 
 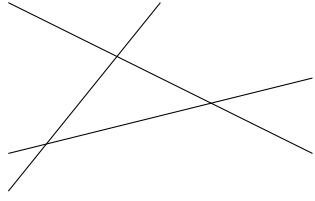
\begin{figure}[ht]
  \begin{center}
   \begin{tikzpicture}
    \draw (0,1) -- (4,-1);
    \draw (0,-1) -- (4,0);
    \draw (0,-1.5) -- (2,1);   
   \end{tikzpicture}
  \end{center}
  \caption{Hyperplanes cut the ambient Euclidean space.}
  \label{three_lines}
 \end{figure}
 These cuttings can be described as a stratification defined by the sign
 function as follows. Define a map
 \[
  \pi_{\mathcal{A}} : \R^n \longrightarrow \Map(L,S_1)
 \]
 by 
 \[
  \pi_{\mathcal{A}}(a)(\ell_i)= \sign(\ell_i(a)).
 \]
 The partial order in $S_1$ induces a partial order on $\Map(L,S_1)$ by
 $\varphi\le \psi$ if and only if $\varphi(\ell_i)\le \psi(\ell_i)$ for
 all $i$. Then $\pi_{\mathcal{A}}$ is the indexing map for the
 stratification of $\R^n$ induced by $\mathcal{A}$.

 There is a standard way to extend the above construction to a
 stratification on $\bbC^n$ by the complexification
 $\mathcal{A}\otimes\bbC=\{H_1\otimes\bbC,\ldots, H_k\otimes\bbC\}$ as is
 studied intensively in the theory of hyperplane arrangements. A good
 reference is the paper \cite{Bjorner-Ziegler92} by Bj{\"o}rner and
 Ziegler. As is sketched at the end of the above paper, the
 stratification of $\R^n$ defined above can be extended to a
 stratification on $\R^n\otimes\R^{\ell}$ as follows:
 Let $\mathcal{A}=\{H_{1},\ldots,H_{k}\}$ and
 $L=\{\ell_1,\ldots,\ell_k\}$ be as above. Then the maps
 \[
  \ell_1\otimes\R^{\ell}, \ldots, \ell_k\otimes\R^{\ell} :
 \R^{n}\otimes\R^{\ell} 
 \longrightarrow \R^{\ell}
 \]
 define a subspace arrangement
 $\mathcal{A}\otimes\R^n=\{H_1\otimes\R^{\ell},\ldots,H_k\otimes\R^{\ell}\}$
 in $\R^{n}\otimes\R^{\ell}$. 

 Let $S_{\ell}=\{0,\pm e_1,\ldots, \pm e_{\ell}\}$ be the poset with
 partial ordering $0<\pm e_1<\cdots< \pm e_{\ell}$.
 Define the $\ell$-th order sign function
 \[
  \sign_{\ell} : \R^{\ell} \longrightarrow S_{\ell}
 \]
 by
 \[
 \sign_{\ell}(\bm{x}) = \begin{cases}
		    \sign(x_{\ell})e_{\ell}, & x_{\ell}\neq 0 \\
		    \sign(x_{\ell-1})e_{\ell-1}, &
			 x_{\ell}=0,x_{\ell-1}\neq 0 \\ 
		    \vdots & \\
		    \sign(x_1)e_1, & x_n=\cdots=x_{2}=0, x_1\neq 0 \\
		    0 & \bm{x}=0.
		   \end{cases}
 \]
 Define a stratification on $\R^n\otimes\R^{\ell}$
 \[
  \pi_{\mathcal{A}\otimes\R^{\ell}} : \R^{n}\otimes\R^{\ell}
 \rarrow{} \Map(L,S_{\ell}) 
 \]
 by
 \[
  \pi_{\mathcal{A}\otimes\R^{\ell}}(a\otimes x)(\ell_i) =
 \sign_{\ell}(\ell_i(a)x). 
 \]
 This is a normal stratification on $\R^n\otimes\R^{\ell}$.
\end{example}

\begin{example}
 \label{stratifications_on_simplex}
 Consider the standard $n$-simplex
 \[
  \Delta^n = \lset{(x_0,\ldots,x_n)\in\R^n}{x_0+\cdots+x_n=1, x_i\ge 0}.
 \]
 Define
 \[
  \pi_n : \Delta^n \longrightarrow 2^{[n]}
 \]
 by
 \[
  \pi_n(x_0,\ldots, x_n) = \lset{i\in [n]}{x_i\neq 0},
 \]
 where $[n]=\{0,\ldots,n\}$ and $2^{[n]}$ is the power set of $[n]$.
 Under the standard poset structure on $2^{[n]}$, $\pi$ is a
 stratification with strata simplices in $\Delta^n$. This is a normal
 stratification. 

 Define another stratification $\pi_{n}^{\max} : \Delta^n\to [n]$ by 
 \[
  \pi_n^{\max}(x_0,\ldots,x_n) = \max\lset{i}{x_i\neq 0},
 \]
 where $[n]$ is equipped with the partial order $0<1<\cdots<n$. The
 resulting decomposition is
 \[
 \Delta^n = (\Delta^n\setminus \Delta^{n-1})\cup
 (\Delta^{n-1}\setminus \Delta^{n-2})\cup 
 \cdots \cup (\Delta^1\setminus \Delta^0) \cup \Delta^0.
 \]
 This is also a normal stratification.
\end{example}

\begin{example}
 \label{stratification_by_group_action}
 Let $G$ be a compact Lie group acting smoothly on a smooth manifold
 $M$. M.~Davis \cite{M.Davis78} defined a stratification on $M$ and on
 the quotient space $M/G$ as follows. The indexing set $I(G)$ is called
 the set of normal orbit types and is defined by
 \[
  I(G) = \set{(\varphi : H\to \GL(V))}{H<G \text{ a closed subgroup},\,
 \varphi  \text{ a representation},\, V^{H}=\{0\}}/_{\sim},
 \]
 where $(\varphi : H\to\GL(V))\sim (\varphi':H'\to\GL(V'))$ if and only
 if there exist an element $g\in G$ and a linear isomorphism
 $f : V\to V'$ such that $H'=gHg^{-1}$ and the diagram
 \[
  \begin{diagram}
   \node{H} \arrow{s,l}{g(-)g^{-1}} \arrow{e,t}{\varphi} \node{\GL(V)}
   \arrow{s,r}{f_*} \\ 
   \node{H'} \arrow{e,b}{\varphi'} \node{\GL(V')}
  \end{diagram}
 \]
 is commutative.

 He defined a map
 \[
  \pi_{M} : M \longrightarrow I(G)
 \]
 by
 \[
  \pi_{M}(x) = \left(G_x, S_x/(S_x)^{G_x}\right),
 \]
 where $G_x$ is the isotropy subgroup at $x$ and
 $S_x = T_xM/T_x(Gx)$. There is a canonical partial order on 
 $I(G)$ and it is proved that $\pi_M$ is a normal stratification
 (Theorem 1.6 in \cite{M.Davis78}). It is also proved that the
 stratification descends to $M/G$.

 In particular, for a torus manifold\footnote{in the sense of
 Hattori-Masuda \cite{Hattori-Masuda03}} or a small cover\footnote{in
 the sense of Davis-Januszkiewicz \cite{Davis-Januszkiewicz91}}
 $M$ of dimension $2n$ or $n$, the quotient $M/T^{n}$ or $M/\Z_2^n$ has
 a canonical normal stratification, respectively. 

 When $M=\CP^n$ or $\RP^n$, the quotient is $\Delta^n$ and the
 stratification by normal orbit types corresponds to the stratification
 $\pi_n$ in Example \ref{stratifications_on_simplex}. We will see the
 other stratification $\pi_n^{\max}$ in Example
 \ref{stratifications_on_simplex} corresponds to  
 the minimal cell decompositions of $\RP^n$ and $\CP^n$ in Example
 \ref{moment_angle_complex}. 

\end{example}

\begin{example}
 The stratification on the quotient $M/G$ has been generalized to the
 notion of manifolds with faces or more
 generally manifolds with corners. See \S6 of Davis' paper
 \cite{M.Davis83}, for example. These are also important examples of
 normal stratifications. 
\end{example}

It is easy to verify that the product of two stratifications is again a
stratification. 

\begin{lemma}
 \label{product_stratification}
 Let $(X,\pi_X)$ and $(Y,\pi_Y)$ be stratified
 spaces. The map 
 \[
  \pi_X\times\pi_Y : X\times Y \longrightarrow P(X)\times P(Y)
 \]
 defines a stratification on $X\times Y$.
\end{lemma}

\begin{proof}
 The product of open maps is again open.
\end{proof}

The following requirements for morphisms of stratified spaces should be
reasonable. 

\begin{definition}
 \label{morphism_of_stratified_spaces}
 Let $(X,\pi_X)$ and $(Y,\pi_Y)$ be stratified spaces. 
 \begin{itemize}
  \item A \emph{morphism of stratified spaces} is a pair
	$\bm{f}=(f,\underline{f})$ of a continuous map $f : X\to Y$  and
	a map of posets $\underline{f} : P(X)\to\ P(Y)$ 
	making the following diagram commutative:
	\[
	\begin{diagram}
	 \node{X} \arrow{e,t}{f} \arrow{s,l}{\pi_X} \node{Y}
	 \arrow{s,r}{\pi_Y} \\ 
	 \node{P(X)} \arrow{e,b}{\underline{f}} \node{P(Y).}
	\end{diagram}
	\]
  \item When $X=Y$ and $f$ is the identity,
	$\bm{f}=(f,\underline{f})$ is 
	called a \emph{subdivision}. We also say that $(X,\pi_X)$ is a
	subdivision of $(Y,\pi_Y)$ or $(Y,\pi_Y)$ is a \emph{coarsening}
	of $(X,\pi_X)$.
  \item When $f(e_{\lambda}) = e_{\underline{f}(\lambda)}$ for each
	$\lambda$, it is called a \emph{strict morphism}.
  \item When $\bm{f}=(f,\underline{f})$ is a strict morphism of
	stratified spaces and $f$ is an embedding of topological spaces,
	$\bm{f}$ is said to be an \emph{embedding} of $X$ into $Y$.
 \end{itemize}
\end{definition}

For a stratified space $\pi : Y \to \Lambda$ and a continuous map
$f : X \to Y$, the composition $f^*(\pi) = \pi\circ f : X \to \Lambda$
may or may not be a stratification.

\begin{definition}
 \label{induced_stratification_definition}
 Let $\bm{f} : (X,\pi_X) \to (Y,\pi_Y)$ be a morphism of
 stratified spaces. When $\underline{f} : P(X)\to \pi_{Y}(f(X))$ is an
 isomorphism of posets
 \[
  \begin{diagram}
   \node{X} \arrow{s,l}{\pi_X} \arrow{e,t}{f} \node{f(X)}
   \arrow{s,r}{\pi_Y} \arrow{e,J}
   \node{Y} \arrow{s,r}{\pi_Y} \\
   \node{P(X)} \arrow{e,b}{\cong} \node{\pi_Y(f(X))} \arrow{e,J}
   \node{P(Y),} 
  \end{diagram}
 \]
 we say $\pi_X$ is \emph{induced} from $\pi_Y$ via $f$. We sometimes
 denote it by $f^*(\pi_Y)$.
\end{definition}

\begin{example}
 \label{stratification_on_Hopf_bundle}
 Consider the double covering map
 \[
  2 : S^1 \longrightarrow S^1.
 \]
 The minimal cell decomposition $\pi_{\mathrm{min}}: S^1=e^0\cup e^1$ on
 $S^1$ in the range does not induce a stratification on $S^1$ in the
 domain, since the inverse images of 
 strata are not connected. But we have a strict morphism of stratified
 spaces 
 \[
  \begin{diagram}
   \node{S^1} \arrow{s} \arrow{e,t}{2} \node{S^1} \arrow{s} \\
   \node{\{0_{+},0_{-},1_{+},1_{-}\}} \arrow{e} \node{\{0,1\}}
  \end{diagram}
 \]
 if $S^1$ in the domain is equipped with the minimal
 $\Sigma_2$-equivariant cell decomposition
 $S^1=e^0_{-}\cup e^0_{+}\cup e^1_{-}\cup e^1_{+}$.

 Consider the complex analogue of the double covering on $S^1$, i.e.\
 the Hopf bundle 
 \[
  \eta : S^3 \longrightarrow S^2.
 \]
 This is a principal fiber bundle with fiber $S^1$. The minimal cell
 decomposition on $S^2$, $S^2=e^0\cup e^2$, is a stratification. This
 stratification induces a stratification on $S^3$
 \[
  S^3 = \eta^{-1}(e^0)\amalg \eta^{-1}(e^2).
 \]
 Note that we have
 \begin{eqnarray*}
  \eta^{-1}(e^0) & \cong & e^0\times S^1 \\
  \eta^{-1}(e^2) & \cong & e^2\times S^1
 \end{eqnarray*}
 The face posets of these stratifications are isomorphic to the poset
 $\{0<2\}$ and we have a commutative diagram
 \[
  \begin{diagram}
   \node{S^3} \arrow{s} \arrow{e,t}{\eta} \node{S^2} \arrow{s} \\
   \node{\{0,2\}} \arrow{e,=} \node{\{0,2\}.}
  \end{diagram}
 \]

 Note that these two examples can be described as
 \begin{eqnarray*}
  S^1 & = & e^0\times S^0 \amalg e^1\times S^0 \\
  S^3 & = & e^0\times S^1 \amalg e^2\times S^1,
 \end{eqnarray*}
 suggesting the existence of a common framework for handling them
 simultaneously. We propose the notion of cylindrical structures as such  
 in \S\ref{cylindrical_structure}. 
\end{example}

As a special case of embeddings of stratified spaces, we have the
notion of stratified subspaces.

\begin{definition}
 \label{stratified_subspace_definition}
 Let $(X,\pi)$ be a stratified space and $A$ be a subspace of $X$. If
 the restriction $\pi|_{A}$ is a stratification, $(A,\pi|_{A})$ is
 called a \emph{stratified subspace} of $(X,\pi)$.

 When the inclusion $i : A \hookrightarrow X$ is a strict morphism, $A$
 is called a \emph{strict stratified subspace} of $X$.
\end{definition}

\begin{remark}
 We study stratified subspaces in detail in
 \S\ref{stratified_subspace}.
\end{remark}

As is the case of cell complexes, the CW condition is useful.

\begin{definition}
 \label{CW_stratification}
 A stratification $\pi$ on $X$ is said to be \emph{CW} if it
 satisfies the following two conditions:
 \begin{enumerate}
  \item (Closure Finite) For each stratum $e_{\lambda}$,
	$\partial e_{\lambda}$ is covered by a finite number of strata. 
  \item (Weak Topology) $X$ has the weak topology determined by the
	covering $\{\overline{e_{\lambda}}\mid \lambda\in P(X) \}$.
 \end{enumerate}
\end{definition}

 For example, weak topologies are useful when we glue quotient maps.

\begin{lemma}
 \label{weak_topology_and_quotient_map}
 Let $\pi_X : X\to \Lambda$ and $\pi_Y : Y \to \Lambda$ be normal CW
 stratified spaces with $P(X)=P(Y)=\Lambda$ and 
 \[
  \bm{f} : X\longrightarrow Y
 \]
 be a strict morphism of stratified spaces with
 $\underline{f}=1_{\Lambda}$. Let us denote the strata in $X$ and $Y$
 indexed by $\lambda$ by $e^{X}_{\lambda}$ and $e^{Y}_{\lambda}$,
 respectively. 
 Suppose
 \[
  f_{\lambda}=f|_{\overline{e^{Y}_{\lambda}}} :
 \overline{e^{Y}_{\lambda}} \to 
 \overline{e^{X}_{\lambda}}
 \]
 is a quotient map for all
 $\lambda\in\Lambda$. Then $f : X\to Y$ is a quotient map. 
\end{lemma}

\begin{proof}
 For a subset $U\subset Y$, suppose $f^{-1}(U)$ is open in $X$. By the
 weak topology condition,
 $f^{-1}(U)\cap\overline{e^{X}_{\lambda}}$ is open 
 in $\overline{e^{X}_{\lambda}}$ for each $\lambda\in\Lambda$. We
 have 
 \[
 f\left(f^{-1}(U)\cap \overline{e^{X}_{\lambda}}\right) =
 U\cap f\left(\overline{e^{X}_{\lambda}}\right) 
 = U\cap \overline{e^{Y}_{\lambda}},
 \]
 since $f$ is a strict morphism and both $X$ and $Y$ are normal. Thus
 \[
  f_{\lambda}^{-1}\left(U\cap \overline{e^{Y}_{\lambda}}\right) =
 f^{-1}\left(f\left(f^{-1}(U)\cap 
 \overline{e^{X}_{\lambda}}\right)\right) = f^{-1}(U)\cap
 \overline{e^{X}_{\lambda}} 
 \]
 is open in $\overline{e^{X}_{\lambda}}$ for each $\lambda$. By
 assumption, $f_{\lambda}$ is a quotient map and $Y$ is CW. And
 $U$ is open in $Y$.
\end{proof}

It is straightforward to verify the
following by using the fact that any topological space has the weak
topology with respect to a locally finite closed covering.

\begin{proposition}
 \label{locally_finite_implies_CW}
 Any locally finite stratified space is CW. 
\end{proposition}

\begin{corollary}
 \label{subdivision_of_CW}
 Let $(X,\pi)$ be a CW stratified space and $(X,\pi')$ be a
 subdivision. If each stratum $e_{\lambda}$  in $(X,\pi)$ is subdivided
 into a finite number of strata in $(X,\pi')$, then $(X,\pi')$ is CW.
\end{corollary}

\begin{proof}
 For each cell $e_{\lambda}$ in $(X,\pi)$, $\overline{e_{\lambda}}$ has
 the weak topology with respect to the covering
 \[
 \overline{e_{\lambda}} = \bigcup_{\lambda'\in P(X,\pi'),
 e_{\lambda'}\subset e_{\lambda}}  \overline{e_{\lambda'}}
 \]
 because of the finiteness assumption. Thus $X$ has the weak topology
 with respect to the covering 
 \[
 X = \bigcup_{\lambda'\in P(X,\pi')} \overline{e_{\lambda'}} =
 \bigcup_{\lambda\in P(X,\pi)} \left(\bigcup_{\lambda'\in P(X,\pi'),
 e_{\lambda'}\subset e_{\lambda}}  \overline{e_{\lambda'}}\right).
 \] 
 The closure finiteness condition also follows from the finiteness of
 the subdivision of each stratum.
\end{proof}

As is the case of CW complexes, metrizability implies local finiteness.   

\begin{lemma}
 \label{metrizable_implies_locally_finite}
 Any metrizable CW stratified space is locally finite.
\end{lemma}

\begin{proof}
 This fact is well known for CW complexes. The same argument can be used
 to prove this generalized statement. We give a proof for the
 convenience of the reader.

 If  $X$ is not locally finite, there exists a point $x\in X$ such that,
 for any open neighborhood $U$ of $x$, $U$ intersects with infinitely
 many strata. For each $n$, let $U_n$ be the $\frac{1}{n}$-neighborhood
 of $x$ and choose a stratum $e_n$ with $U_n\cap e_n\neq\emptyset$ and 
 $x\not\in e_n$. Choose $x_n\in U_n\cap e_n$. Then the set
 $A=\{x_n\}_{n=1,2,\ldots}$ is 
 closed by the CW conditions. This contradicts to the fact
 that $x\not\in A$ and $x\in \overline{A}$. Thus $X$ is
 locally finite.
\end{proof}


%% file: cell_structure.tex
\subsection{Cells}
\label{cell_structure}

We would like to define a cellular stratification on a topological space
$X$ as a stratification on $X$ whose strata are ``cells''. As we have
seen in Example \ref{stratification_by_arrangement}, we would like to
regard chambers and faces of a real hyperplane arrangement as ``cells'',
suggesting the need of non-closed cells.

\begin{definition}
 \label{globular_definition}
 A \emph{globular $n$-cell} is a subset $D$ of $D^n$ containing
 $\Int(D^n)$. We call $D\cap \partial D^n$ the \emph{boundary} of $D$
 and denote it by $\partial D$. The number $n$ is called the
 \emph{globular dimension} of $D$.
\end{definition}

\begin{remark}
 We introduce another dimension, called the stellar dimension, for a
 more general class of subsets of $D^n$ in
 \S\ref{stellar_stratified_space}. 
\end{remark}

We use the following definition of cell structures.

\begin{definition}
 \label{cell_structure_definition}
 Let $X$ be a topological space. For a non-negative integer $n$, an
 \emph{$n$-cell structure} on a subspace $e\subset X$ is a pair
 $(D,\varphi)$ of a globular $n$-cell $D$ and a continuous map 
 \[
  \varphi : D \longrightarrow X
 \]
 satisfying the following conditions:
 \begin{enumerate}
  \item $\varphi(D)=\overline{e}$ and $\varphi : D \to \overline{e}$ is
	a quotient map. 
  \item The restriction $\varphi|_{\Int(D^n)} : \Int(D^n)\to e$ is a
	homeomorphism. 
 \end{enumerate}

 For simplicity, we denote an $n$-cell structure $(D,\varphi)$ on $e$
 by $e$ when there is no risk of confusion.
 The map $\varphi$ is called the \emph{cell structure map} or the
 \emph{characteristic map} of $e$ and 
 $D$ is called the \emph{domain} of $e$. The number $n$ is called the
 (globular) \emph{dimension} of $e$. 
\end{definition}

\begin{example}
 \label{open_disk}
 The open $n$-disk $\Int(D^n)$ is a globular $n$-cell.
 The standard homeomorphism
 \[
  \Int (D^n) \rarrow{\cong} \R^n
 \]
 defines an $n$-cell structure on $\R^n$. The domain is $\Int (D^n)$.
\end{example}

\begin{example}
 \label{nonreduced_cell}
 Let $X= \Int(D^2) \cup \{(1,0)\}$.
 The identity map defines a $2$-cell structure on $X$. 

 There is another choice. Let $D = \Int(D^2) \cup S^1_{+}$. The
 deformation retraction 
 of $S^1_{+}$ onto $(1,0)$ can be extended to a continuous map
 \[
  \varphi : D \longrightarrow X
 \]
 whose restriction to $\Int(D^n)$ is a homeomorphism. For example,
 $\varphi$ is given in polar coordinates by
 \[
  \varphi(re^{i\theta}) = \begin{cases}
			 r e^{i(1-r)\theta}, & |\theta|\le
			 \frac{\pi}{2} \\
			 r e^{i\{\theta-(\pi-\theta)r\}}, &
			 \frac{\pi}{2} \le \theta\le \pi \\
			 r e^{i\{\theta+(\pi+\theta)r\}}, & -\pi\le
			 \theta\le -\frac{\pi}{2}.
			\end{cases}
 \]

 \begin{figure}[ht]
 \begin{center}
  \begin{tikzpicture}
   \draw [dotted] (0,0) circle (1cm);
   \draw [blue] (1,0) arc (0:90:1cm);
   \draw [blue] (1,0) arc (0:-90:1cm);
   \draw [fill,blue] (0,1) circle (2pt);
   \draw [fill,white] (0,1) circle (1pt);
   \draw [fill,blue] (0,-1) circle (2pt);
   \draw [fill,white] (0,-1) circle (1pt);
   \draw (0,0) node {$D$};

   \draw [->,green] (0.55,0.95) arc (60:30:1.125cm);
   \draw [->,green] (0.55,-0.95) arc (-60:-30:1.125cm);

   \draw [->] (1.5,0) -- (2.5,0);
   \draw (2,0.3) node {$\varphi$};

   \draw [dotted] (4,0) circle (1cm);
   \draw [fill,blue] (5,0) circle (2pt);
   \draw (4,0) node {$X$};
  \end{tikzpicture}
 \end{center}
  \caption{An exotic cell structure on $\Int(D^2)\cup\{(1,0)\}$}
  \label{exotic_cell_structure}
 \end{figure}
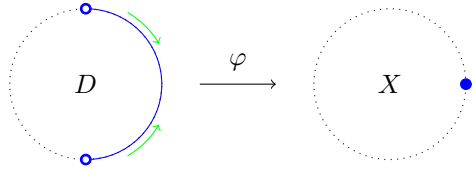
 Note that $\varphi$ is not a quotient map. For example, the image of
 $\{(x,y)\in D \mid x>0\}$ under $\varphi$ is open under the quotient
 topology, but it is not open under the relative topology on $X$. Thus
 this is not a $2$-cell structure on $X$.
\end{example}

In other words, we need to put the quotient topology on $X$ in order for
the map $\varphi$ in the above example to be a cell
structure. Fortunately, the quotient topology is a popular choice in
many practical examples.

\begin{example}
 \label{nonreduced_simplicial_set}
 For any simplicial set $X$, the geometric realization $|X|$ is
 known to be a CW complex, whose cells are in one-to-one correspondence 
 with nondegenerate simplices in $X$. 

 Consider the simplicial set $X=s(\Delta^2)/s(\Delta^1)$, where
 $\Delta^1$  and $\Delta^2$ are regarded as ordered simplicial
 complexes and $s(-)$ is the functor in Example
 \ref{ordered_simplicial_complex_as_simplicial_set} which transforms
 ordered simplicial complexes to simplicial sets. The geometric
 realization $|X|$ is a cell complex consisting of two $0$-cells
 $[0]=[1]$, $[2]$, two $1$-cells $[0,2]$, $[1,2]$, and a
 $2$-cell $[0,1,2]$. The characteristic map for the $2$-cell is given
 by the composition 
 \[
  \psi : D^2 \cong \Delta^2\times\{[0,1,2]\} \hookrightarrow
 \coprod_{i=0}^{\infty}\Delta^n\times X_n \rarrow{} |X|.
 \]
 Thus it is given by collapsing the blue arc in Figure
 \ref{Delta^2/Delta^1}. 
 \begin{figure}[ht]
 \begin{center}
  \begin{tikzpicture}
   \draw (0.5,0.86) arc (60:300:1cm); 
   \draw [blue] (1,0) arc (0:60:1cm);
   \draw [blue] (1,0) arc (0:-60:1cm);

   \draw [fill] (-1,0) circle (2pt);
   \draw [blue,fill] (0.5,0.86) circle (2pt);
   \draw [blue,fill] (0.5,-0.86) circle (2pt);

   \draw [->] (1.75,0) -- (2.25,0);
   \draw (2,0.3) node {$\psi$};
   \draw (4,0) circle (1cm);
   
   \draw [fill] (3,0) circle (2pt);
   \draw [blue,fill] (5,0) circle (2pt);
   \draw (2.7,0) node {$[2]$};
   \draw (5.8,0) node {$[0]=[1]$};
   \draw (4,1.25) node {$[1,2]$};
   \draw (4,-1.25) node {$[0,2]$};
   \draw (4,0) node {$[0,1,2]$};
  \end{tikzpicture}
 \end{center}
  \caption{A $2$-cell structure on $s(\Delta^2)/s(\Delta^1)$.}
  \label{Delta^2/Delta^1}
 \end{figure}
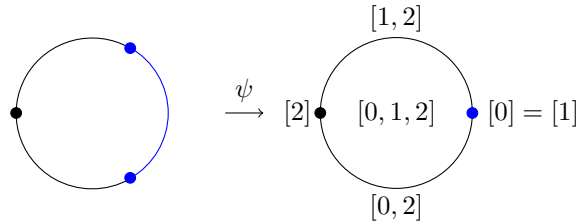
%
%
%
 By definition, $|X|$ is equipped with the quotient topology and $\psi$
 defines a $2$-cell structure.
\end{example}

Cells satisfying one (or both) of the following conditions appear
frequently. 

\begin{definition}
 \label{regular_cell}
 Let $X$ be a topological space and  $e\subset X$ a subspace. An
 $n$-cell structure $(D,\varphi)$ on $e$, or 
 simply an $n$-cell $e$, is said to be
 \begin{itemize}
  \item \emph{closed} if $D=D^n$,
  \item \emph{regular} if $\varphi : D \to \overline{e}$ is a
	homeomorphism.  
 \end{itemize}
\end{definition}

\begin{example}
 Given a cell complex $X$ and its $n$-cell $e$, the characteristic map
 \[
  \varphi : D^n \longrightarrow X
 \]
 defines a closed $n$-cell structure on $e$. When $X$ is a regular cell
 complex, it is a regular $n$-cell structure on $e$.
\end{example}

\begin{example}
 Let $\mathcal{A}$ be an arrangement of a finite number of hyperplanes
 in $\R^n$. Bounded strata in the stratification $\pi_{\mathcal{A}}$ are 
 convex polytopes and they are closed and regular cells.

 We may also define non-closed but regular cell structures on unbounded
 strata as follows. 
 Suppose $\mathcal{A}$ is essential, namely normal vectors
 to the hyperplanes span $\R^n$. Then we may choose a closed ball $B$
 with center at the origin in $\R^n$ which contains all bounded strata.
 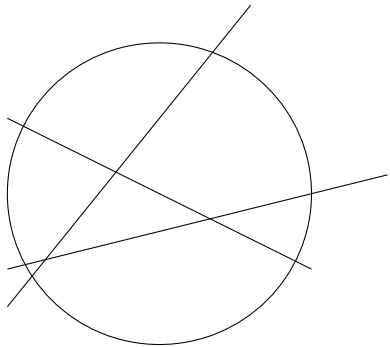
\begin{figure}[ht]
 \begin{center}
  \begin{tikzpicture}
   \draw (0,1) -- (4,-1);
   \draw (0,-1) -- (5,0.25);
   \draw (0,-1.5) -- (3.2,2.5);   
   \draw (2,0) circle (2);
  \end{tikzpicture}
 \end{center}
  \caption{A bounding sphere for bounded strata.}
  \label{bounding_sphere}
 \end{figure}
 Hyperplanes also cut the boundary sphere and define a stratification
 $\pi_{\mathcal{A},B}$ on $B$ whose 
 strata are all closed cells. The inclusion $\Int(B)\hookrightarrow \R^n$
 is a morphism of stratified spaces under which face posets can be
 identified 
 \[
 \begin{diagram}
  \node{\Int (B)} \arrow{s,l}{\pi_{\mathcal{A},B}} \arrow{e,J} \node{\R^n}
  \arrow{s,r}{\pi_{\mathcal{A}}} \\ 
  \node{P(\pi_{\mathcal{A},B}|_{\Int(B)})} \arrow{e,=}
  \node{P(\pi_{\mathcal{A}}).} 
 \end{diagram}
 \]
 For each unbounded stratum $e$ in $\pi_{\mathcal{A}}$, the
 intersection $e\cap \Int(B)$ is a cell in $B$. Let
 $\varphi : D^k\to \overline{e\cap B}$ be a characteristic map for
 $e\cap \Int(B)$ and define
 $D=\varphi^{-1}(\overline{e}\cap \Int(B))$. We can 
 compress the outside of $B$ into $e\cap \Int(B)$ via a homeomorphism
 $\psi : \overline{e} \to \overline{e}\cap \Int(B)$. The composition
 \[
  D \rarrow{\varphi|_{D}} \overline{e}\cap \Int(B) \rarrow{\psi^{-1}}
 \overline{e} 
 \]
 defines a regular cell structure on $e$.
\end{example}

Non-closed cells might have a bad topology.

\begin{example}
 \label{powdered_ball}
 Let 
 \[
  p : D^2\setminus\{(0,1)\} \longrightarrow \{(x,y)\in\R^2\mid y\ge 0\}
 \]
 be the homeomorphism given by extending the stereographic projection
 $S^1\setminus\{(0,1)\} \to \R$. Let
 \[
  X = \set{(x,y)\in\R^2}{y>0} \cup \set{(x,0)}{x\in\Q}
 \]
 and define $D=p^{-1}(X)$.
 \begin{figure}[ht]
 \begin{center}
  \begin{tikzpicture}
   \shade (0,0) circle (1.2cm);
   \draw (0,0) circle (1.2cm);
   \draw (0,1.2) circle (2pt);
   \draw [fill,white] (0,1.2) circle (1pt);
   \draw (0,0) node {$D^2\setminus\{(0,1)\}$};

   \draw [->] (2,0) -- (3,0);
   \draw (2.5,0.3) node {$p$};

   \shade (3.8,-1.2) rectangle (6.2,1.2);
   \draw (5,0) node {$\set{(x,y)}{y\ge 0}$};
   \draw (3.8,-1.2) -- (6.2,-1.2);
  \end{tikzpicture}
 \end{center}
  \caption{A $2$-cell structure for the closed upper half-plane.}
 \end{figure}
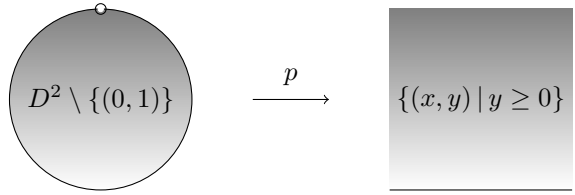
 Then the restriction
 \[
  p|_{D} : D \longrightarrow X
 \]
 defines a $2$-cell structure on $X$.
 \begin{figure}[ht]
 \begin{center}
  \begin{tikzpicture}
   \shade (0,0) circle (1.2cm);
   \draw [dotted] (0,0) circle (1.2cm);
   \draw (0,0) node {$D$};

   \draw [->] (2,0) -- (3,0);
   \draw (2.5,0.3) node {$p|_{D}$};

   \shade (3.8,-1.2) rectangle (6.2,1.2);
   \draw (5,0) node {$X$};
   \draw [dotted] (3.8,-1.2) -- (6.2,-1.2);
  \end{tikzpicture}
 \end{center}
  \caption{A $2$-cell structure for $X$.}
 \end{figure}
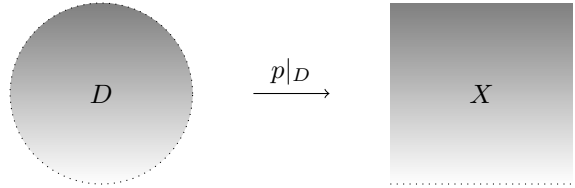
 $X$ and $D$ are not locally
 compact. This example suggests that taking a product of cell structures
 might not be easy because of our requirement for a characteristic map
 to be a quotient map.

 For example, let $Y$ be the quotient of $X$ under the relation
 $(x,0)\sim (x',0)$ if and only if $x-x'\in\Z$. The composition of
 $p|_D$ with the canonical projection
 \[
  \varphi_{Y} : D \rarrow{p|_D} X \longrightarrow Y
 \]
 defines a $2$-cell structure on $Y$. But the product with $p|_D$
 \[
  p|_{D}\times\varphi_Y : D\times D \longrightarrow X\times Y
 \]
 is not a quotient map, since the product
 $\Q\times \Q\to \Q\times(\Q/\Z)$ of the identity map and the 
 quotient map is not a quotient map.

 We need to impose certain conditions to take products of cell
 structures freely. Our solution is to require cellular structures on
 the boundaries of the domains of cells. See 
 \S\ref{cylindrical_structure} for more details. 
\end{example}

%% file: cellular_stratified_space_definition.tex
\subsection{Cellular Stratifications}
\label{cellular_stratified_space_definition}

So far we have defined the notions of stratified spaces and cell
structures. Now we are ready to define cellular stratified spaces by
combining these two structures.

\begin{definition}
 \label{definition_of_cellular_stratified_space}
 Let $X$ be a Hausdorff space. A \emph{cellular stratification} on
 $X$ is a pair $(\pi,\Phi)$ of a stratification
 $\pi : X \to P(X)$ on $X$ and a collection of cell structures
 $\Phi = \{\varphi_{\lambda} : D_{\lambda}\to \overline{e_{\lambda}}\}_{\lambda\in P(X)}$
 satisfying the condition that, for each $n$-cell $e_{\lambda}$,
 $\partial e_{\lambda}$ is covered by cells of
 dimension less than or equal to $n-1$. 

 A \emph{cellular stratified space} is a triple $(X,\pi,\Phi)$
 where $(\pi,\Phi)$ is a cellular stratification on $X$. As
 usual, we abbreviate it by $(X,\pi)$ or $X$, if there is no
 danger of confusion. 
\end{definition}

\begin{remark}
 The term ``cellular stratified space'' has been already used in the
 study of singularities. See, for example, Sch{\"u}rmann's book
 \cite{SchuermannBook}. We found, however, that his definition is too
 restrictive for our purposes.
\end{remark}

\begin{example}
 Consider the ``topologist's sine curve''
 \[
  S = \left.\left\{(x,\sin\smallfrac{1}{x}) \ \right| \ 0<x\le 1 \right\}.
 \]
 Its closure in $\R^2$ is given by
 \[
  \overline{S} = S\cup \{(0,t) \mid -1\le t\le 1\}.
 \]
 Since $S$ is homeomorphic to the half interval $(0,1]$ via the function
 $\sin\frac{1}{x}$, the decomposition
 \[
  \overline{S} = \{(1,\sin 1)\}\cup \{(0,1)\}\cup \{(0,-1)\}\cup
 \left.\left\{(x,\sin\smallfrac{1}{x}) \ \right| \ 0<x< 1 \right\}
 \cup \{(0,t) \mid -1<t<1\}
 \]
 is a stratification of $\overline{S}$ consisting of five
 strata. Although the stratum $\lset{(x,\sin\frac{1}{x})}{0< x< 1}$ is
 homeomorphic 
 to $\Int(D^1)$, there is no $1$-cell structure on this stratum,
 since there is no way to extend a homeomorphism
 $(0,1] \cong S$ to a continuous map $[0,1] \to \overline{S}$.

 Furthermore this stratification does not satisfy the dimension
 condition in Definition \ref{definition_of_cellular_stratified_space},
 since $\partial S = \{(0,t) \mid -1\le t\le 1\}$.
\end{example}

\begin{remark}
 The above example is borrowed from Pflaum's book \cite{Pflaum01}. He
 describes an even more pathological example. See 1.1.12 on page 18 of
 his book. 
\end{remark}

Here is another non-example.

\begin{example}
 \label{powdered_ball2}
 Consider the $2$-cell structure on 
 $X = \{(x,y)\in\R^2 \mid y>0\} \cup \{(x,0)\mid x\in\Q\}$ defined in
 Example \ref{powdered_ball}. Let $e^2=\{(x,y)\in\R^2 \mid y>0\}$ and
 $e^0_{x}= \{(x,0)\}$ for $x\in \Q$. We have a decomposition
 \[
 X = \bigcup_{x\in\Q} e_x^0\cup e^2
 \]
 and, for each $x\in\Q$, the identification 
 \[
  \psi_x : D^0 \cong \{(x,0)\} \hookrightarrow X
 \]
 defines a $0$-cell structure on $\{(x,0)\}$. But this is not a cellular
 stratified space, since $e_x^0$ is not locally closed.  


 On the other hand, let $Y=\{(x,y)\in \R^2 \mid y>0\}\cup \Z\times\{0\}$
 and consider the decomposition
 \[
  Y = \bigcup_{n\in\Z} e_n^0\cup e^2.
 \]
 Each $e_n^0$ is locally closed in $Y$ and this is a cellular
 stratification on $Y$. Note, however, this is not a CW stratification,
 since we need infinitely many $0$-cells to cover the boundary of
 the $2$-cell $\partial e^2$.
\end{example}

\begin{definition}
 We say a cellular stratification is \emph{CW} if its underlying
 stratification is CW (Definition \ref{CW_stratification}).
\end{definition}


\begin{lemma}
 A CW cellular stratification $(\pi,\Phi)$  on a space $X$ defines and
 is defined by the following structure
 \begin{itemize}
  \item a filtration $X_0\subset X_1 \subset \cdots \subset
	X_{n-1}\subset X_n\subset \cdots$ on $X$,
  \item an $n$-cell structure on each connected component of
	$X_n\setminus X_{n-1}$ 
\end{itemize}
 satisfying the following conditions:
 \begin{enumerate}
  \item $X=\bigcup_{n=0}^{\infty} X_n$.
  \item For each stratum $e$ in $X_n$, $\partial e$ is covered by a
	finite number of strata in $X_{n-1}$.
  \item $X$ has the weak topology determined by the covering
	$\{\overline{e}\mid e\in P(X)\}$.
 \end{enumerate}
\end{lemma}

\begin{proof}
 Suppose $(\pi,\Phi)$ is a CW cellular stratification. Define
 \[
  X_n = \bigcup_{e\in P(X),\dim e\le n} e
 \]
 then we obtain a filtration on $X$ with $X=\bigcup_{n=0}^{\infty} X_n$.
 By definition, the boundary $\partial e$ of each $n$-cell $e$ is
 covered by a finite number of cells of dimension less than or equal to
 $n-1$. Also by definition, $X$ has the weak topology by the covering
 $\set{\overline{e}}{e\in P(X)}$.

 It remains to show that the $n$-cells are connected components of the
 difference $X_n\setminus X_{n-1}$, i.e.\ each $n$-cell $e$ is open and
 closed in 
 $X_n\setminus X_{n-1}$. By the local closedness, $e$ is open in
 $\overline{e}$. Since $\partial e \subset X_{n-1}$, $e$ is open in
 $X_n\setminus X_{n-1}$. For any $n$-cell $e'$, the intersection
 $e\cap \overline{e'}$ in $X_n\setminus X_{n-1}$ is $e\cap e'$ and is
 $\emptyset$ or $e=e'$. And thus it is closed in $\overline{e'}$.

 It is left to the reader to check the converse.
\end{proof}

\begin{example}
 \label{cellular_stratification_by_arrangement}
 The stratification $\pi_{\mathcal{A}}$ of $\R^n$ defined by a
 hyperplane arrangement $\mathcal{A}$ (Example
 \ref{stratification_by_arrangement}) is a CW cellular stratification. 

 The Bj{\"o}rner-Ziegler stratification
 $\pi_{\mathcal{A}\otimes\R^{\ell}}$ on $\R^n\otimes\R^{\ell}$
 is also a CW cellular stratification.
\end{example}

In Definition \ref{morphism_of_stratified_spaces}, we defined morphisms
of stratified spaces as ``stratification-preserving maps''. 
Note that, in our definition of cellular stratified spaces, we include
characteristic maps as defining data.
When we consider maps between cellular stratified spaces, we require
them to be compatible with characteristic maps.

\begin{definition}
 Let $(X,\pi_X,\Phi_X)$ and $(Y,\pi_Y,\Phi_Y)$ be cellular stratified
 spaces. A \emph{morphism of cellular stratified spaces}
 from $(X,\pi_X,\Phi_X)$ to $(Y,\pi_Y,\Phi_Y)$ consists of 
 \begin{itemize}
  \item a morphism $\bm{f} : (X,\pi_X) \to (Y,\pi_Y)$ of stratified
	spaces, and
  \item a family of maps
	\[
	 f_{\lambda} : D_{\lambda} \longrightarrow
	D_{\underline{f}(\lambda)} 
	\]
	indexed by cells
	$\varphi_{\lambda}: D_{\lambda}\to \overline{e_{\lambda}}$ in
	$X$ making the diagrams 
	\[
	 \begin{diagram}
	  \node{X} \arrow{e,t}{f} \node{Y} \\
	  \node{D_{\lambda}} \arrow{n,l}{\varphi_{\lambda}}
	  \arrow{e,b}{f_{\lambda}} \node{D_{\underline{f}(\lambda)}}
	  \arrow{n,r}{\psi_{f(\lambda)}} 
	 \end{diagram}
	\]
	commutative, where
	$\psi_{\underline{f}(\lambda)} : D_{\underline{f}(\lambda)} \to
	\overline{e_{\underline{f}(\lambda)}}$ is the characteristic map
	for $e_{\underline{f}(\lambda)}$. 
 \end{itemize}
 The category of cellular stratified spaces is denoted by
 $\category{CSSpaces}$. 
\end{definition}

\begin{remark}
 When $\bm{f} : (X,\pi_X,\Phi_X)\to (Y,\pi_Y,\Phi_Y)$ is a morphism of
 cellular stratified spaces, the compatibility of $f_{\lambda}$ with
 characteristic maps implies
 $f_{\lambda}(\Int(D_{\lambda}))\subset\Int(D_{\underline{f}(\lambda)})$.
\end{remark}

\begin{remark}
 In algebraic topology, the requirement for maps between cell complexes
 is much weaker. A map $f : X \to Y$ between cell complexes
 is said to be \emph{cellular}, if $f(X_n)\subset Y_n$ for each
 $n$. The author thinks, however, this terminology is misleading.
\end{remark}

\begin{definition}
 A morphism
 $(\bm{f},\{f_{\lambda}\}) : (X,\pi_X,\Phi_X) \to (Y,\pi_Y,\Phi_Y)$ of cellular
 stratified spaces is said to be \emph{strict} if
 $\bm{f} : (X,\pi_X) \to (Y,\pi_Y)$ is a strict morphisms of stratified
 spaces and $f_{\lambda}(0)=0$ for each $\lambda\in P(X,\pi_X)$.
\end{definition}

Once we have morphisms between cellular stratified spaces, we have a
notion of equivariant stratification.

\begin{definition}
 Let $G$ be a group. A cellular stratified space $X$ equipped with a
 monoid morphism\footnote{Or a functor $G \to \category{CSSpaces}$.}
 $G \longrightarrow \category{CSSpaces}(X,X)$
 is called a \emph{$G$-cellular stratified space}.
\end{definition}

In order to study configuration spaces, products and subdivisions
are important. However, the product of two cell structures may not be a
cell structure as we have seen in Example
\ref{powdered_ball}. Subdivisions of cell structures are not easy to
handle, either.
We discuss products and subdivisions of cellular stratified spaces in
\S\ref{product} and \S\ref{subdivision_of_cell}, respectively.

%% file: stellar_stratified_space.tex
\subsection{Stellar Stratified Spaces}
\label{stellar_stratified_space}

In Definition \ref{cell_structure_definition}, we required the domain
$D$ of an $n$-cell to contain $\Int(D^n)$. As we will see in the proof 
of Theorem \ref{embedding_Sd}, this condition is requiring too
much. Furthermore, the definition of the globular dimension of a cell is
not appropriate when the cell is not closed. In this
section, we introduce stellar cells 
and study stratified spaces whose strata are stellar cells.
Stellar structures also play an essential role in our description of the
classifying space of the face category of cellular stratified spaces. 

Let us first define ``star-shaped cells''.

\begin{definition}
 \label{aster_definition}
 A subset $S$ of $D^N$ is said to be an \emph{aster} if
 $\{0\}\ast\{x\}\subset S$ for any $x\in S$, where $\ast$ is the join
 operation defined by connecting points by line 
 segments\footnote{Definition \ref{join_definition}.}. The subset
 $S\cap \partial D^N$ is called the \emph{boundary} of $S$ and is
 denoted by $\partial S$. The complement $S\setminus\partial S$ of the
 boundary is called the \emph{interior} of $S$ and is denoted by $\Int(S)$.

 We say $S$ is \emph{thin} if $S=\{0\}\ast \partial S$.
\end{definition}

We require the existence of a cellular stratification on the
boundary in order to define the dimension. 

\begin{definition}
 A \emph{stellar cell} is an aster $S$ in $D^N$ for some $N$ 
 such that there exists a cellular stratification on
 $\partial D^N$ containing $\partial S$ as a cellular stratified
 subspace.

 When the (globular) dimension of $\partial S$ is $n-1$, we define the
 \emph{stellar dimension} of $S$ to be $n$ and call $S$ a \emph{stellar
 $n$-cell}.  
\end{definition}

An $n$-cell in the sense of \S\ref{cell_structure} is stellar if its
boundary has a structure of cellular stratified space. However,
the dimension as a stellar cell might be smaller than $n$.

\begin{example}
 Consider the globular $n$-cell $\Int(D^n)$ in Example
 \ref{open_disk}. It is a stellar cell with empty boundary. Thus its
 stellar dimension is $0$.

 By adding three points to the boundary, for example,
 we obtain a globular $n$-cell $D$
 whose stellar dimension is $1$. 

 \begin{figure}[ht]
 \begin{center}
  \begin{tikzpicture}
   \draw [dotted] (0,0) circle (1cm);
   \draw (0,0) node {$\Int(D^n)$};

   \draw [dotted] (3,0) circle (1cm);
   \draw [fill] (2.5,0.86) circle (2pt);
   \draw [fill] (3.5,0.86) circle (2pt);
   \draw [fill] (3,-1) circle (2pt);
   \draw (3,0) node {$D$};
  \end{tikzpicture}
 \end{center}
  \caption{An open disk and an open disk with three points added.}
 \end{figure}
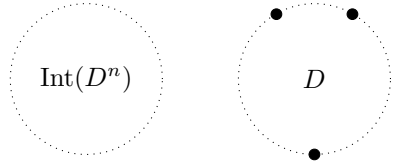

 These two stellar cells are not thin. The first example contains
 $\{0\}$ as a thin stellar cell. The second example contains a graph of
 the shape of $Y$ as a thin stellar cell. See Figure
 \ref{thin_stellar_cells}. 

\begin{figure}[ht]
 \begin{center}
  \begin{tikzpicture}
   \draw [dotted] (0,0) circle (1cm);
   \draw [fill,blue] (0,0) circle (2pt);
   
   \draw [dotted] (3,0) circle (1cm);
   \draw [blue] (3,0) -- (2.5,0.86);
   \draw [blue] (3,0) -- (3.5,0.86);
   \draw [blue] (3,0) -- (3,-1);
   \draw [fill,blue] (3,0) circle (2pt);
   \draw [fill,blue] (2.5,0.86) circle (2pt);
   \draw [fill,blue] (3.5,0.86) circle (2pt);
   \draw [fill,blue] (3,-1) circle (2pt);
  \end{tikzpicture}
 \end{center}
 \caption{Thin stellar cells.}
 \label{thin_stellar_cells}
\end{figure}
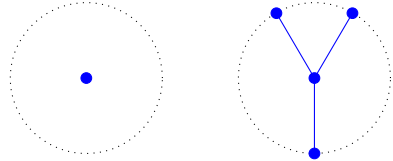
\end{example}

\begin{definition}
 \label{stellar_structure_definition}
 An $n$-\emph{stellar structure} on a subset $e$ of a
 topological space $X$ 
 is a pair $(S,\varphi)$ of a stellar $n$-cell $S$ and a quotient map
 \[
  \varphi : S \longrightarrow X
 \]
 satisfying the following conditions:
 \begin{enumerate}
  \item $\varphi(S)=\overline{e}$.
  \item The restriction of $\varphi$ to $\Int(S)$ is a homeomorphism
	onto $e$. 
 \end{enumerate} 
\end{definition}

By replacing cell structures by stellar structures in the definition of
cellular stratifications, we obtain the notion of stellar
stratifications. 

\begin{definition}
 A \emph{stellar stratification} on a topological space $X$ consists of
 a stratification $(X,\pi)$, and a stellar structure on each
 $e_{\lambda}=\pi^{-1}(\lambda)$ for $\lambda\in P(X)$
 satisfying the condition that 
 for each stellar $n$-cell $e_{\lambda}$,
 $\partial e_{\lambda}$ is covered by stellar cells
 of stellar dimension less than or equal to $n-1$.

 A space equipped with a stellar stratification is called a
 \emph{stellar stratified space}.
\end{definition}

The following is a typical example.

\begin{example}
 \label{graph_as_stellar_complex}
 Consider a finite graph $X$ regarded as a $1$-dimensional cell
 complex. Suppose $X$ is regular as a cell complex. Then we may define
 the barycentric subdivision $\Sd(X)$ of $X$.
 This cell complex can be expressed as 
 a union of open stars of vertices of $X$ and the barycenters of
 $1$-cells in $X$. 
 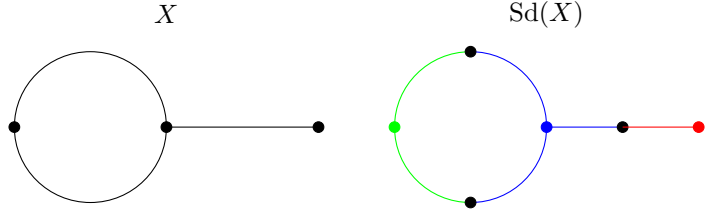
\begin{figure}[ht]
 \begin{center}
  \begin{tikzpicture}
   \draw (0,0) circle (1cm);
   \draw [fill] (1,0) circle (2pt);
   \draw [fill] (-1,0) circle (2pt);
   \draw (1,0) -- (3,0);
   \draw [fill] (3,0) circle (2pt);
   \draw (1,1.5) node {$X$};

   \draw [green] (5,1) arc (90:270:1cm);
   \draw [blue] (5,-1) arc (-90:90:1cm);
   \draw [fill,blue] (6,0) circle (2pt);
   \draw [fill,green] (4,0) circle (2pt);
   \draw [fill] (5,1) circle (2pt);
   \draw [fill] (5,-1) circle (2pt);
   \draw [blue] (6,0) -- (7,0);
   \draw [fill] (7,0) circle (2pt);
   \draw [red] (7,0) -- (8,0);
   \draw [fill,red] (8,0) circle (2pt);
   \draw (6,1.5) node {$\Sd(X)$};
  \end{tikzpicture}
 \end{center}
  \caption{A stellar stratification on $\Sd(X)$.}
 \end{figure}
 And we have a structure of stellar stratified space on $\Sd(X)$.
\end{example}

\begin{remark}
 We will extend the definition of the barycentric subdivision to
 cellular and stellar stratified spaces in
 \S\ref{barycentric_subdivision} and investigate more precise relations
 between stellar stratifications and their barycentric subdivisions.
\end{remark}

The definition of morphisms of stellar stratified spaces should be
obvious.

\begin{definition}
 The category of stellar stratified spaces is denoted by
 $\category{SSSpaces}$. 
\end{definition}

If a stellar stratified space $X$ is CW, we may describe $X$ as a
quotient space of the collection of domains of cell structures.

\begin{lemma}
 \label{quotient_of_D(X)}
 For a CW stellar stratified space $X$ with stellar structure
 $\{\varphi_{\lambda}:D_{\lambda}\to\overline{e_{\lambda}}\}_{\lambda\in\Lambda}$,
 define 
 \[
  D(X) = \coprod_{\lambda\in\Lambda} D_{\lambda}
 \]
 and let $\Phi: D(X) \to X$ be the
 map defined by the stellar structure maps. Then
 $\Phi$ is a quotient map. More explicitly, we have a
 homeomorphism
 \[
  X \cong D(X)/_{\sim_{\Phi}},
 \]
 where the relation $\sim_{\Phi}$ is defined by $x\sim_{\Phi} y$ if and
 only if $\varphi_{\mu}(x) =\varphi_{\lambda}(y)$ for $x\in D_{\mu}$ and
 $y\in D_{\lambda}$.
\end{lemma}

This fact is proved for CW cellular stratified spaces as
Lemma 2.28 in \cite{1312.7368}. The proof can be applied to stellar
stratified spaces without a change and is omitted.

%% file: discrete_face_category.tex
\section{Totally Normal Cellular Stratified Spaces}
\label{discrete_face_category}

In order to study their homotopy types, we would like to impose
appropriate ``niceness conditions'' on cellular and stellar stratified
spaces. 

\input{regularity_and_normality}

\input{total_normality}
\input{total_normality_examples}

%% file: regularity_and_normality.tex
\subsection{Regularity and Normality}
\label{regularity_and_normality}

Regularity and normality are frequently used conditions on
CW complexes. We have already defined normality for stratified spaces
(Definition \ref{normal_stratification}) and regularity for cells
(Definition \ref{regular_cell}).

\begin{definition}
 Let $X$ be a cellular or stellar stratified space. We say $X$ is
 \emph{normal}, if 
 it is normal as a stratified space. We say $X$ is \emph{regular}
 if all cells in $X$ are regular.
\end{definition}

\begin{example}
 The cellular stratification $\pi_{\mathcal{A}\otimes\R^{\ell}}$
 on $\R^{n}\otimes\R^{\ell}$ in Example
 \ref{stratification_by_arrangement} 
 defined by a real arrangement $\mathcal{A}$ in $\R^n$ is regular and
 normal. 
\end{example}

\begin{example}
 Consider the following cellular stratified space
 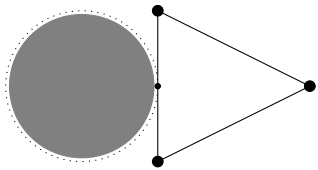
\begin{figure}[ht]
 \begin{center}
  \begin{tikzpicture}
   \draw [dotted] (0,0) circle (1cm);
   \draw [gray,fill] (0,0) circle (0.95cm);
   \draw [fill] (1,0) circle (1pt);
   \draw [fill] (1,1) circle (2pt);
   \draw [fill] (1,-1) circle (2pt);
   \draw [fill] (3,0) circle (2pt);
   \draw (1,1) -- (3,0);
   \draw (1,-1) -- (3,0);
   \draw (1,1) -- (1,-1);
  \end{tikzpicture}
 \end{center}
  \caption{A non-normal regular cellular stratified space.}
 \end{figure}
 obtained by gluing $\Int(D^2)$ to the boundary of a $2$-simplex at the
 middle point of an edge. The domain of the characteristic map of the
 $2$-cell is $\Int(D^2)\cup \{(1,0)\}$, whose boundary is
 mapped into the 
 $1$-skeleton. Thus this is a regular cellular stratified
 space. However, this is not normal.

 If we regard the globular $2$-cell $\Int(D^2)\cup\{(1,0)\}$ as a
 stellar cell, its stellar dimension is $1$ and this stellar structure
 does not satisfy the dimensional requirement. Thus this 
 is not a stellar stratified space.
\end{example}

In the case of CW complexes, regularity always implies normality. See
Theorem 2.1 in Chapter III of
the book \cite{Lundell-Weingram} by Lundell and Weingram. The above
example suggests that the failure of this fact for cellular stratified
spaces is partly due to the ``wrong definition'' of dimensions of
globular cells. The right notion of dimension is the stellar one.
 
Even for stellar stratified spaces, however, regularity does not
necessarily imply normality.

\begin{example}
 By adding an arc to the previous example, we obtain a stratified space
 as follows.
 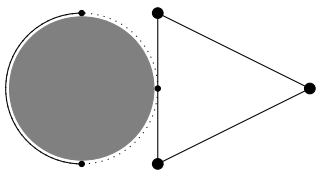
\begin{figure}[ht]
 \begin{center}
  \begin{tikzpicture}
   \draw [dotted] (0,0) circle (1cm);
   \draw [gray,fill] (0,0) circle (0.95cm);
   \draw (0,1) arc (90:270:1cm);
   \draw [fill] (0,1) circle (1pt);
   \draw [fill] (0,-1) circle (1pt);

   \draw [fill] (1,0) circle (1pt);
   \draw [fill] (1,1) circle (2pt);
   \draw [fill] (1,-1) circle (2pt);
   \draw [fill] (3,0) circle (2pt);
   \draw (1,1) -- (3,0);
   \draw (1,-1) -- (3,0);
   \draw (1,1) -- (1,-1);
  \end{tikzpicture}
 \end{center}
  \caption{A non-normal regular stellar stratified space.}
 \end{figure}
 The globular $2$-cell
 $\Int(D^2)\cup\{(1,0)\}\cup \lset{(x,y)\in S^1}{x\le 0}$ is a stellar
 $2$-cell. And this is a regular stellar stratified space. But this is
 not normal.
\end{example}

\begin{remark}
 It seems that, if $X$ is a stellar
 stratified space in which the boundary of each stellar $n$-cell is
 ``pure of dimension $n-1$'' and if $X$ is regular, then $X$ is normal.
\end{remark}

%% file: total_normality.tex
\subsection{Total Normality}
\label{total_normality}

In the case of cellular or stellar stratified spaces, relations among
cells are not 
as rigid as those in cell complexes.

\begin{definition}
 \label{total_normality_definition}
 Let $X$ be a normal cellular stratified space.
 $X$ is called \emph{totally normal} if, for each $n$-cell
 $e_{\lambda}$,  
\begin{enumerate}
 \item there exists a structure of regular cell complex on $S^{n-1}$
       containing $\partial D_{\lambda}$ as a strict stratified subspace
       of $S^{n-1}$ and 
 \item for any cell $e$ in $\partial D_{\lambda}$, there exists a cell
       $e_{\mu}$ in $\partial e_{\lambda}$ such that
       $e_{\mu}$ and $e$ share the same domain and the characteristic
       map $\varphi_{\mu}$ of $e_{\mu}$ factors through $D_{\lambda}$
       via the characteristic map of $e$:
	\[
	\begin{diagram}
	 \node{\overline{e}} \arrow{e,J} \node{\partial D_{\lambda}}
	 \arrow{e,J} \node{D_{\lambda}} 
	 \arrow{e,t}{\varphi_{\lambda}} \node{\overline{e_{\lambda}}}
	 \arrow{e,J} \node{X} \\ 	 
	 \node{D} \arrow{n} \arrow{e,=} \node{D_{\mu}}
	 \arrow{e,b}{\varphi_{\mu}} \node{\overline{e_{\mu}}.}
	 \arrow{ene,J}
	\end{diagram}		
	\]
\end{enumerate}
\end{definition}

The total normality for stellar stratified spaces is defined
analogously. But we need to use cellular subdivisions.

\begin{definition}
 Let $X$ be a normal stellar stratified space. We say $X$ is
 \emph{totally normal} if, for each stellar $n$-cell
 $e_{\lambda}$ with domain $D_{\lambda}\subset D^{N}$,
 \begin{enumerate}
  \item there exist a structure of regular cell complex on $S^{N-1}$
	and a structure of stellar stratified space on
	$\partial D_{\lambda}$ such that each stellar cell in
	$\partial D_{\lambda}$ is a strict stratified subspace of
	$S^{N-1}$, and
  \item for any stellar cell $e$ in $\partial D_{\lambda}$, there exists
	a stellar cell $e_{\mu}$ in $\partial e_{\lambda}$ such that
       $e_{\mu}$ and $e$ share the same domain and the characteristic
       map of $e_{\mu}$ factors through $D_{\lambda}$ via the
       characteristic map of $e$.
 \end{enumerate}

\end{definition}


\begin{lemma}
 Any totally normal cellular stratified space has a structure of a
 totally normal stellar stratified space. 
\end{lemma}

The characteristic maps of totally normal stellar stratified spaces
preserve cells.

\begin{lemma}
 \label{characteristic_map_in_totally_normal_space}
 For a cell $e_{\lambda}$ in a totally normal stellar
 stratified space $X$, let
 $\varphi_{\lambda} : D_{\lambda}\to \overline{e_{\lambda}}\subset X$ be
 the characteristic map. Then there exists a structure of stellar
 stratified space on $D_{\lambda}$ under which $\varphi_{\lambda}$ is
 a strict morphism of stellar stratified spaces.
\end{lemma}

\begin{proof}
 Let $e_{\lambda}$ be a cell in a totally normal stellar stratified
 space $X$ and
 $\varphi_{\lambda} : D_{\lambda}\to \overline{e_{\lambda}}$ the
 characteristic map. Let  
 \[
  \partial D_{\lambda} = \bigcup_{\nu} e'_{\nu}
 \]
 be the stellar stratification in the definition of total normality. We
 have 
 \[
  \partial e_{\lambda} = \varphi_{\lambda}(\partial D_{\lambda}) =
 \bigcup_{\nu} \varphi_{\lambda}(e'_{\nu}).
 \]
 By the definition of total normality, for each $\nu$, there exists a
 cell $e_{\mu'}$ in $\partial e_{\lambda}$ whose characteristic map
 makes the following diagram commutative
 \[
  \begin{diagram}
   \node{\overline{e'_{\nu}}} \arrow{e,J} \node{\partial D_{\lambda}}
   \arrow{e,t}{\varphi_{\lambda}} \node{\partial e_{\lambda}} \\
   \node{D_{\nu}} \arrow{n,l}{\psi_{\nu}} \arrow{e,=} \node{D_{\mu'}}
   \arrow{e,b}{\varphi_{\mu'}} \node{\overline{e_{\mu'}},} \arrow{n,J}
  \end{diagram}
 \]
 where $\psi_{\nu}$ is the characteristic map for $e'_{\nu}$. This
 implies that each $\varphi_{\lambda}(e'_{\nu})$ is a cell in $X$ and
 thus $\varphi_{\lambda}$ is a strict morphism of stellar stratified
 spaces.
\end{proof}

\begin{corollary}
 Let $(\pi,\Phi)$ be a stellar stratification on $X$ satisfying the
 first condition of total normality.

 Then it is totally normal if and only if the following condition is
 satisfied: For each pair $e_{\mu}< e_{\lambda}$, let
 $F(X)(e_{\mu},e_{\lambda})$ be the set
 of all maps
 \[
 b : D_{\mu} \longrightarrow D_{\lambda}
 \]
 making the diagram
 \[
 \begin{diagram}
  \node{D_{\lambda}} \arrow{e,t}{\varphi_{\lambda}}
  \node{\overline{e_{\lambda}}} \\ 
  \node{D_{\mu}} \arrow{n,l}{b} \arrow{e,b}{\varphi_{\mu}}
  \node{\overline{e_{\mu}}} 
  \arrow{n,J} 
 \end{diagram}
 \]
 commutative, where
 $\varphi_{\lambda}$ and $\varphi_{\mu}$ are the
 characteristic maps of $e_{\lambda}$ and $e_{\mu}$, respectively. Then 
 \begin{equation}
  \partial D_{\lambda} = \bigcup_{e_{\mu}<e_{\lambda}}
   \bigcup_{b\in F(X)(e_{\mu},e_{\lambda})} 
   b(\Int(D_{\mu})). 
   \label{cover_of_domain}
 \end{equation}

\end{corollary}

\begin{proof}
 Suppose $(\pi,\Phi)$ is totally normal.
 For a pair of cells $e_{\mu}<e_{\lambda}$, by Lemma
 \ref{characteristic_map_in_totally_normal_space}, there exists a
 stellar cell
 $e$ in $D_{\lambda}$ such that $\varphi_{\lambda}(e)=e_{\mu}$. By the
 assumption of total normality, the characteristic map
 $\psi : D\to \overline{e}$ of $e$ makes the following diagram
 commutative
 \[
  \begin{diagram}
   \node{\overline{e}} \arrow{e,J} \node{\partial D_{\lambda}}
   \arrow{e,t}{\varphi_{\lambda}}  \node{\partial e_{\lambda}} \\
   \node{D} \arrow{e,=} \arrow{n,l}{\psi} \node{D_{\mu}}
   \arrow{e,b}{\varphi_{\mu}} 
   \node{\overline{e_{\mu}}} \arrow{n,J}
  \end{diagram}
 \]
 The collection of all such characteristic maps $\psi$ is
 $F(X)(e_{\mu},e_{\lambda})$ and we have
 \[
 \partial D_{\lambda} = \bigcup_{e_{\mu}<e_{\lambda}} \bigcup_{b\in
 F(X)(e_{\mu},e_{\lambda})} b(\Int(D_{\mu})). 
 \]

 Conversely, the assumption (\ref{cover_of_domain}) implies that, for
 any stellar cell $e$ in $\partial D_{\lambda}$, there is a
 corresponding cell $e_{\mu}$ in $\partial e_{\lambda}$ whose
 characteristic map makes the required diagram commutative.
\end{proof}

The lifts $b : D_{\mu}\to D_{\lambda}$ of characteristic maps appeared
in the above Corollary play an essential role when we analyze totally
normal stellar stratified spaces. It is easy to see that each $b$ is an
embedding. 

\begin{lemma}
 \label{b_is_an_embedding}
 Let $X$ be a totally normal stellar stratified space. Then
 each 
 $b\in F(X)(e_{\mu},e_{\lambda})$ is an embedding of stellar
 stratified spaces for each pair $e_{\mu}<e_{\lambda}$. 
\end{lemma}

\begin{proof}
 This follows from the assumption that the cellular stratification on 
 $D_{\lambda}$ is regular.
\end{proof}

%% file: total_normality_examples.tex
\subsection{Examples of Totally Normal Cellular Stratified Spaces}
\label{total_normality_examples}

Let us take a look at some examples. The first example is
borrowed from Kirillov's paper \cite{1009.4227}.

\begin{example}
 Consider $D=\Int(D^2)\cup S^1_+ = e^1\cup e^2$. Define $X$ by folding
 the blue part of $e^1$ according to the directions indicated by the
 blue arrows.
 \begin{figure}[ht]
 \begin{center}
  \begin{tikzpicture}
   \draw [dotted] (0,0) circle (2cm);
   \draw [green] (0,2cm) arc (90:45:2cm);
   \draw [blue] (1.41cm,-1.41cm) arc (-45:45:2cm);
   \draw [green] (0,-2cm) arc (-90:-45:2cm);
   \draw [blue,dotted] (0,0) -- (1.9,0.55);
   \draw [blue,dotted] (0,0) -- (1.9,-0.55);
   \draw [blue,dotted] (0,0) -- (1.41,1.41);
   \draw [blue,dotted] (0,0) -- (1.41,-1.41);
   \draw [blue,->] (1.55,1) -- (1.65,0.8);
   \draw [blue,->] (1.8,-0.1) -- (1.8,0.1);
   \draw [blue,->] (1.65,-0.8) -- (1.55,-1);
   \draw [green,fill] (0,2cm) circle (2pt);
   \draw [white,fill] (0,2cm) circle (1pt);
   \draw [green,fill] (0,-2cm) circle (2pt);
   \draw [white,fill] (0,-2cm) circle (1pt);
   
   \draw [->] (3cm,0) -- (5cm,0);
   \draw (4cm,0.5cm) node {$\varphi$};

   \draw [dotted] (8cm,0) circle (2cm);
   \draw [green] (8cm,-2cm) arc (-90:90:2cm);
   \draw [green,fill] (8cm,2cm) circle (2pt);
   \draw [white,fill] (8cm,2cm) circle (1pt);
   \draw [green,fill] (8cm,-2cm) circle (2pt);
   \draw [white,fill] (8cm,-2cm) circle (1pt);
   \draw [blue] (9.92,-0.55) arc (-16:16:2cm);
   \draw [blue,dotted] (8,0) -- (9.9,-0.55);
   \draw [blue] (8,0) -- (9.9,0.55);
   \draw [blue,dotted] (9.5,-0.45) arc (-15:15:1.5cm);
   \draw [blue,dotted] (9.45,0.45) arc (165:195:1.5cm);
  \end{tikzpicture}
 \end{center}
  \caption{An example from Kirillov's paper.}
 \end{figure}
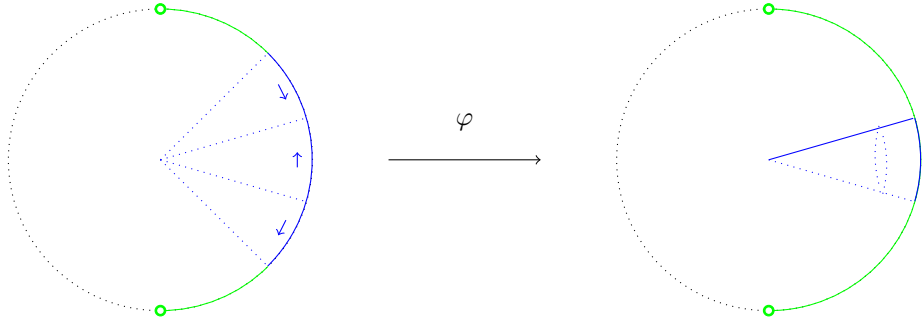
 Note that $\varphi(e^1)$ is homeomorphic to $\Int(D^1)$. Let
 \[
 \psi : \Int(D^1) \longrightarrow e^1 \subset X
 \]
 a homeomorphism. Identifications only occur on $e^1$ and the quotient
 map  
 \[
  \varphi : D \longrightarrow X
 \]
 is a homeomorphism onto its image when restricted to $\Int(D^2)$ and
 thus defines a characteristic map for the $2$-cell. 

 These maps $\psi$ and $\varphi$ define a cellular stratification on
 $X$. However, there is no way to obtain a map $h : \Int(D^1) \to D$
 making the diagram commutative
 \[
 \begin{diagram}
  \node{D} \arrow{e,t}{\varphi} \node{X} \\
  \node{\Int(D^1)} \arrow{n,l,..}{h} \arrow{ne,b}{\psi}
 \end{diagram}
 \]
 and thus this cellular stratification is not totally normal. Note,
 however, that we obtain a totally normal cellular stratification by an
 appropriate  subdivision of $e^1$ and $\varphi(e^1)$.

%
%
\end{example}

\begin{example}
 \label{minimal_cell_decomposition_of_circle}
 Consider the minimal cell decomposition
 \[
  S^1 = e^0\cup e^1 = \{(1,0)\}\cup \left(S^1\setminus\{(1,0)\}\right).
 \]
 The characteristic map for the $1$-cell
 \[
  \varphi_1 : D^1=[-1,1] \longrightarrow S^1
 \]
 is given by $\varphi(t) = (\cos(2\pi t),\sin(2\pi t))$. There are two
 lifts $b_{-1}$ and $b_{1}$ of the characteristic map $\varphi_0$ for the
 $0$-cell. 

 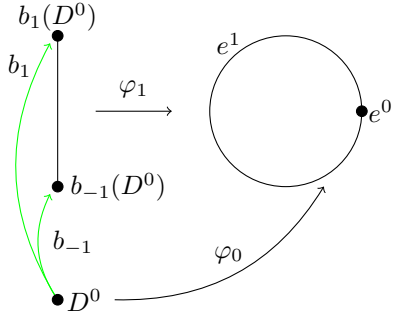
\begin{figure}[ht]
 \begin{center}
  \begin{tikzpicture}
   \draw (0,1) -- (0,-1);
   \draw [fill] (0,1) circle (2pt);
   \draw (0,1.25) node {$b_{1}(D^0)$};
   \draw [fill] (0,-1) circle (2pt);
   \draw (0.8,-1) node {$b_{-1}(D^0)$};

   \draw [->] (0.5,0) -- (1.5,0);
   \draw (1,0.3) node {$\varphi_1$};

   \draw (3,0) circle (1);
   \draw [fill] (4,0) circle (2pt);
   \draw (2.25,0.9) node {$e^1$};
   \draw (4.25,0) node {$e^0$};

   \path [green,->] (0,-2.5) edge [bend left] (-0.1,-1.1);
   \draw (0.2,-1.75) node {$b_{-1}$};
   \path [green,->] (0,-2.5) edge [bend left] (-0.1,0.9);
   \draw (-0.5,0.6) node {$b_{1}$};
   \draw [fill] (0,-2.5) circle (2pt);
   \draw (0.35,-2.5) node {$D^0$};

   \path [->] (0.75,-2.5) edge [bend right] (3.5,-1);
   \draw (2.25,-1.9) node {$\varphi_0$};
  \end{tikzpicture}
 \end{center}
  \caption{The minimal cell decomposition of $S^1$ is totally normal.}
 \end{figure}
%
%
%
%

 Since $\partial D^1= \{-1,1\} = b_{-1}(D^0)\cup b_{1}(D^0)$, this is 
 totally normal.
\end{example}

\begin{example}
 \label{graphs_are_totally_normal}
 More generally, any $1$-dimensional CW-complex is totally normal. In
 fact, $0$-cells are always regular and there are only two types of
 $1$-cells, i.e.\ regular cells and cells whose characteristic maps are
 given by collapsing the boundary of $D^1$ to a point. By the above
 example, all $1$-cells are totally normal.
 This fact allows us to apply results of this paper to configuration
 spaces of graphs \cite{1312.7368}.
\end{example}

\begin{example}
 \label{punctured_torus}
 Consider the punctured torus
 \[
 X = S^1\times S^1 \setminus e^0\times e^0
 = e^1\times e^0\cup e^0\times e^1\cup e^1\times e^1
 \]
 with the stratification induced from the product stratification of the
 minimal cell decomposition of $S^1$.

 A characteristic map for the $2$-cell can be obtained by removing four
 corners from the domain of the product
 \[
 \varphi_1\times\varphi_1 : D^1\times D^1 \setminus
 \{(1,1),(1,-1),(-1,1),(-1,-1)\} \longrightarrow X  
 \]
 of $\varphi_1$ in Example
 \ref{minimal_cell_decomposition_of_circle}. Let us denote the
 characteristic maps for $e^1\times e^0$ and $e^0\times e^1$ by
 $\varphi_{1,0}$ and $\varphi_{0,1}$, respectively.
 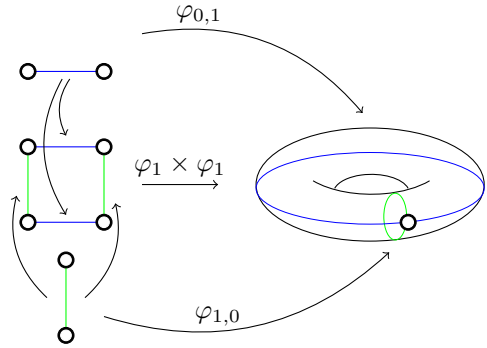
\begin{figure}[ht]
 \begin{center}
  \begin{tikzpicture}
   \draw [green] (-0.5,-0.5) -- (-0.5,0.5);
   \draw [blue] (-0.5,0.5) -- (0.5,0.5);
   \draw [blue] (-0.5,-0.5) -- (0.5,-0.5);
   \draw [green] (0.5,-0.5) -- (0.5,0.5);
   \draw [fill] (-0.5,-0.5) circle (3pt);
   \draw [white,fill] (-0.5,-0.5) circle (2pt);
   \draw [fill] (-0.5,0.5) circle (3pt);
   \draw [white,fill] (-0.5,0.5) circle (2pt);
   \draw [fill] (0.5,-0.5) circle (3pt);
   \draw [white,fill] (0.5,-0.5) circle (2pt);
   \draw [fill] (0.5,0.5) circle (3pt);
   \draw [white,fill] (0.5,0.5) circle (2pt);
   
   \draw [->] (1,0) -- (2,0);
   \draw (1.5,0.25) node{$\varphi_1\times\varphi_1$};

   \draw (4,0) ellipse (1.5cm and 0.75cm);
   \draw (4.5,-0.05) arc (15:165:0.5cm and 0.25cm); 
   \draw (3.25,0.05) arc (220:320:1cm and 0.5cm);
   \draw [green,very thin] (4.33,-0.425) ellipse (0.15cm and 0.31cm);
   \draw [blue,very thin] (4,-0.05) ellipse (1.5cm and 0.475cm);
   \draw [fill] (4.5,-0.5) circle (3pt);
   \draw [white,fill] (4.5,-0.5) circle (2pt);

   \draw [green] (0,-2) -- (0,-1);
   \draw [fill] (0,-2) circle (3pt);
   \draw [white,fill] (0,-2) circle (2pt);
   \draw [fill] (0,-1) circle (3pt);
   \draw [white,fill] (0,-1) circle (2pt);

   \path [->] (0.5,-1.75) edge [bend right] (4.25,-0.9);
   \draw (2,-1.75) node {$\varphi_{1,0}$};

   \path [->] (0.25,-1.5) edge [bend right] (0.65,-0.25);
   \path [->] (-0.25,-1.5) edge [bend left] (-0.65,-0.15);

   \draw [blue] (-0.5,1.5) -- (0.5,1.5);
   \draw [fill] (-0.5,1.5) circle (3pt);
   \draw [white,fill] (-0.5,1.5) circle (2pt);
   \draw [fill] (0.5,1.5) circle (3pt);
   \draw [white,fill] (0.5,1.5) circle (2pt);

   \path [->] (1,2) edge [bend left] (3.9,0.95);
   \draw (1.75,2.25) node {$\varphi_{0,1}$};

   \path [->] (0.05,1.4) edge [bend right] (0,0.65);
   \path [->] (-0.05,1.4) edge [bend right] (0,-0.4);
  \end{tikzpicture}
 \end{center}
  \caption{A totally normal punctured torus.}
 \end{figure}
 The characteristic map of each $1$-cell has two lifts and the images
 cover the boundary of the domain of the $2$-cell. Thus this is totally
 normal. 
\end{example}

\begin{example}
 \label{Delta-set_is_totally_normal}
 Let $X$ be a $\Delta$-set\footnote{Definition
 \ref{Delta-set_definition}.}. Define 
 \[
  \pi_X : \|X\| \longrightarrow \coprod_{n=0}^{\infty} X_n
 \]
 by $\pi_X(x)=\sigma$ if $x$ is represented by
 $(\bm{t},\sigma)\in\Delta^n\times X_n$ with $t_i\neq 0$ for all
 $i\in\{0,\ldots,n\}$. 
 Note that $P(\|X\|)=\coprod_{n=0}^{\infty}X_n$ can be made into a poset
 by defining $\tau\le \sigma$ if and only if there exists a morphism
 $u : [m]\to[n]$ in $\Delta_{\inj}$ with $X(u)(\sigma)=\tau$.

 Then the map $\pi_X$ is a cellular stratification and we have
 \[
 \|X\| = \bigcup_{n=0}^{\infty}\bigcup_{\sigma\in X_n}
 \Int(\Delta^n)\times \{\sigma\}.
 \]
 Let us denote the $n$-cell $\Int(\Delta^n)\times \{\sigma\}$
 corresponding to $\sigma\in X_n$ by $e_{\sigma}$. The characteristic  
 map $\varphi_{\sigma}$ for $e_{\sigma}$ is defined by the composition 
 \[
 \varphi_{\sigma} :  D^n \cong \Delta^n\times\{\sigma\} \hookrightarrow
 \coprod_{n=0}^{\infty} \Delta^n\times X_n \rarrow{} \|X\|.
 \]

 Suppose $\tau<\sigma$ for $\tau \in X_m$ and $\sigma\in X_n$. There
 exists a morphism $u : [m]\to [n]$ with $X(u)(\sigma)=\tau$. With this 
 morphism, we have the following commutative diagram
 \[
  \begin{diagram}
   \node{\Delta^n\times\{\sigma\}} \arrow{e,t}{\varphi_{\sigma}}
   \node{\overline{e_{\sigma}}} \\ 
   \node{\Delta^m\times\{\tau\}} \arrow{e,b}{\varphi_{\tau}}
   \arrow{n,l}{u^{\tau,\sigma}_*} \node{\overline{e_{\tau}},} \arrow{n,J}
  \end{diagram}
 \]
 where
 $u^{\tau,\sigma}_* : \Delta^m\times\{\tau\} \to \Delta^n\times\{\sigma\}$
 is a copy of the affine map $u_* : \Delta^m \to \Delta^n$  induced by
 the inclusion of vertices $u : [m]\to[n]$.
 Hence this cellular stratification on $\|X\|$ is totally normal.

 In particular, for any acyclic category $C$ (with discrete topology),
 the classifying space $BC$ is a totally normal CW complex by Lemma
 \ref{classifying_space_of_acyclic_category}. 
\end{example}

\begin{example}
 \label{PLCW}
 A.~Kirillov, Jr.\ introduced a structure called PLCW-complex in
 \cite{1009.4227}. A PLCW-complex is defined by attaching PL-disks. The
 attaching map of an $n$-cell is required to be a strict morphism of
 cellular stratified spaces under a suitable PLCW decomposition of the
 boundary sphere. 

 Besides the PL requirement in Kirillov's definition, the only
 difference between totally normal CW complexes and PLCW complexes is
 that, in totally normal CW complexes, the boundary sphere of each 
 characteristic map is required to have a regular cell
 decomposition. For example, the cell decomposition of $D^3$ in Figure
 \ref{nontotallynormal_PLCW} 
 is a PLCW complex but is not totally normal.
 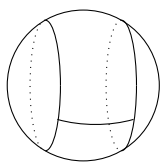
\begin{figure}[ht]
 \begin{center}
  \begin{tikzpicture}
   \draw (0,0) circle (1cm);

   \draw [dotted] (0.5,0.866) arc (90:270:0.2cm and 0.866cm);
   \draw (0.5,-0.866) arc (270:360:0.2cm and 0.866cm);
   \draw (0.7,0) arc (0:90:0.2cm and 0.866cm);

   \draw [dotted] (-0.5,0.866) arc (90:270:0.2cm and 0.866cm);
   \draw (-0.5,-0.866) arc (270:360:0.2cm and 0.866cm);
   \draw (-0.3,0) arc (0:90:0.2cm and 0.866cm);

   \draw (-0.34,-0.45) arc (240:300:1cm and 0.5cm);
  \end{tikzpicture}
 \end{center}
  \caption{A PLCW complex which is not totally normal.}
  \label{nontotallynormal_PLCW}
 \end{figure}
%
%
%
\end{example}

\begin{example}
 Let $X=\R\times\R_{\ge 0}$ with $0$-cells $e^0_n=\{(n,0)\}$ for
 $n\in\Z$, $1$-cells $e^1_n = (n,n+1)\times\{0\}$ for $n\in\Z$, and a
 $2$-cell $e^2= \R\times\R_{>0}$. The characteristic map $\varphi$ of
 the $2$-cell is given by extending the stereographic projection
 $S^1\setminus\{(0,1)\} \to \R$. The domain is
 $D=D^2\setminus\{(0,1)\}$. 
 \begin{figure}[ht]
 \begin{center}
  \begin{tikzpicture}
   \shade (0,0) circle (1cm);
   \draw (0,0) circle (1cm);
   \draw (0,1) circle (2pt);
   \draw [fill,white] (0,1) circle (1pt);
   \draw (0,0) node {$D$};

   \draw [->] (1.5,0) -- (2.5,0);
   \draw (2,0.25) node {$\varphi$};

   \shade (3,-1) rectangle (5,1);
   \draw (4,0) node {$X$};
   \draw (3,-1) -- (5,-1);
   \draw [fill] (3.5,-1) circle (2pt);
   \draw [fill] (4,-1) circle (2pt);
   \draw [fill] (4.5,-1) circle (2pt);
  \end{tikzpicture}
 \end{center}
  \caption{A cellular stratification on the closed upper half-plane with
  $0$-cells $\Z$.}
 \end{figure}
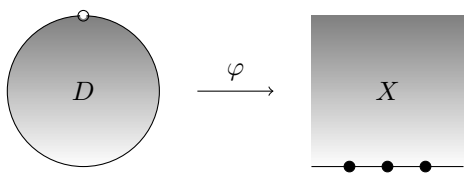
%
%
 This is regular and normal. It even satisfies the second condition in
 Definition \ref{total_normality_definition}. But it is not
 totally normal, since the corresponding stratification on
 $\partial D^2\setminus\{(0,1)\}$ cannot be a strict stratified subspace
 of a regular cell decomposition of $\partial D^2$. 

 Note that this example does not satisfy the closure finiteness
 condition. Hence it is not CW.
\end{example}

\begin{example}
 \label{minimal_cell_decomposition_of_S2}
 Consider the minimal cell decomposition
 \[
  S^2 = e^0\cup e^2.
 \]
 We may choose a regular cell decomposition of $\partial D^2$ and make
 it a stellar stratification. For example, we may use the minimal regular
 cell decomposition 
 $S^1 = e^0_+\cup e^0_{-}\cup e^1_+\cup e^1_{-}$.

 In order to make $S^2$ into totally normal, however, this
 stratification is not fine enough.
 There are infinitely many 
 lifts of the characteristic map of the 
 $0$-cell parametrized by points in $S^1$. See Figure
 \ref{lifts_of_0-cells}. 
 \begin{figure}[ht]
 \begin{center}
  \begin{tikzpicture}
   \draw (0,0) circle (1cm);
   \draw (0,0) node {$D^2$};

   \draw [->] (1.5,0) -- (2.5,0);
   \draw (2,0.25) node {$\varphi_2$};

   \shade [ball color=white] (4,0) circle (1cm);
   \draw (4,0) circle (1cm);
   \draw [dotted] (4,0) ellipse (1cm and 0.5cm);
   \draw (3,0.7) node {$e^2$};
   \draw [fill] (5,0) circle (2pt);
   \draw (5.25,0) node {$e^0$};

   \draw [fill] (0.65,-0.775) circle (2pt);
   \draw (0.5,-0.6) node {$z$};

   \draw [fill] (0,-2.5) circle (2pt);
   \path [green,->] (0.05,-2.45) edge [bend right] (0.65cm,-0.9cm);
   \path [green,->] (0,-2.45) edge [bend left] (-0.15,-1.1);
   \path [green,->] (-0.05,-2.45) edge [bend left] (-0.9cm,-0.6cm);
   \draw (1cm,-1.8cm) node {$b_z$};

   \path [->] (0.5cm,-2.5cm) edge [bend right] (4.5cm,-1cm);
   \draw (3cm,-2.6cm) node {$\varphi_0$};
  \end{tikzpicture}
 \end{center}
  \caption{A minimal cell decomposition of $S^2$.}
  \label{lifts_of_0-cells}
 \end{figure}
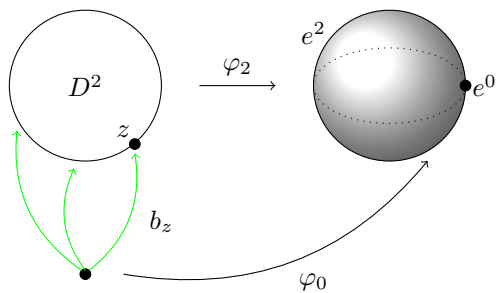
%
%
%
%
%
%
 Although $\partial D^2$ is covered by the images of $b_z$'s
 \[
  \partial D^2 = \bigcup_{z\in S^1} b_{z}(D^0),
 \]
 this is not a cell decomposition of $\partial D^2$. Hence this is not
 totally normal.
\end{example}

%% file: topological_face_category.tex
\section{Cylindrically Normal Cellular Stratified Spaces}
\label{topological_face_category}

\input{cylindrical_structure}
\input{locally_polyhedral}
\input{cylindrical_examples}

%% file: cylindrical_structure.tex
\subsection{Cylindrical Structures}
\label{cylindrical_structure}

In Example \ref{minimal_cell_decomposition_of_S2}, lifts of $\varphi_0$
are parametrized by points in 
$S^1=\partial D^2$. This example suggests that we need to topologize the
set of all lifts. Inspired by the work of
Cohen, Jones, and Segal on Morse theory \cite{Cohen-Jones-SegalMorse}
and this example, we introduce the following definition.

\begin{definition}
 \label{cylindrical_structure_definition}
 A \emph{cylindrical structure} on a normal cellular stratified space
 $(X,\pi)$ consists of
 \begin{itemize}
  \item a normal stratification on $\partial D^{n}$ containing
	$D_{\lambda}$ as a strict stratified subspace for each $n$-cell
	$\varphi_{\lambda} : D_{\lambda}\to \overline{e_{\lambda}}$,
  \item a stratified space $P_{\mu,\lambda}$ and a strict morphism of
	stratified spaces 
	\[
	 b_{\mu,\lambda} : P_{\mu,\lambda}\times D_{\mu} \longrightarrow
	\partial D_{\lambda}
	\]
	for each pair of cells
	$e_{\mu} \subset \partial e_{\lambda}$, and
  \item a strict morphism of stratified spaces
	\[
	 c_{\lambda_0,\lambda_1,\lambda_2} :
	 P_{\lambda_1,\lambda_2}\times P_{\lambda_0,\lambda_1} 
	 \longrightarrow P_{\lambda_0,\lambda_2}
	\]
	 for each sequence $\overline{e_{\lambda_0}} \subset
	\overline{e_{\lambda_1}} \subset \overline{e_{\lambda_2}}$
 \end{itemize}
 satisfying the following conditions:
 \begin{enumerate}
  \item The restriction of $b_{\mu,\lambda}$ to
	$P_{\mu,\lambda}\times\Int(D_{\mu})$ is a homeomorphism onto its
	image. 
  \item The following three types of diagrams are commutative. 
	\[
	 \begin{diagram}
	  \node{D_{\lambda}} \arrow{e,t}{\varphi_{\lambda}} \node{X} \\
	  \node{P_{\mu,\lambda}\times D_{\mu}} \arrow{n,l}{b_{\mu,\lambda}}
	  \arrow{e,t}{\pr_2} 
	  \node{D_{\mu}.} \arrow{n,r}{\varphi_{\mu}}
	 \end{diagram}
	\]
	\[
	 \begin{diagram}
	  \node{P_{\lambda_1,\lambda_2}\times
	  P_{\lambda_0,\lambda_1}\times D_{\lambda_0}}
	  \arrow{s,l}{c_{\lambda_0,\lambda_1,\lambda_2}\times 1} 
	  \arrow{e,t}{1\times b_{\lambda_0,\lambda_1}}
	  \node{P_{\lambda_1,\lambda_2}\times D_{\lambda_1}}
	  \arrow{s,r}{b_{\lambda_1,\lambda_2}} \\ 
	  \node{P_{\lambda_0,\lambda_2}\times D_{\lambda_0}}
	  \arrow{e,b}{b_{\lambda_0,\lambda_2}} \node{D_{\lambda_2}} 
	 \end{diagram}
	\]
	\[
	 \begin{diagram}
	  \node{P_{\lambda_2,\lambda_3}\times
	  P_{\lambda_1,\lambda_2}\times P_{\lambda_0,\lambda_1}}
	  \arrow{e,t}{c\times 1} \arrow{s,l}{1\times c} 
	  \node{P_{\lambda_1,\lambda_3}\times P_{\lambda_0,\lambda_1}}
	  \arrow{s,r}{c} \\ 
	  \node{P_{\lambda_2,\lambda_3}\times P_{\lambda_0,\lambda_2}}
	  \arrow{e,b}{c} 
	  \node{P_{\lambda_0,\lambda_3}.} 
	 \end{diagram}
	\]
  \item We have 
	\[
	 \partial D_{\lambda} = \bigcup_{e_{\mu}\subset \partial
	e_{\lambda}}  b_{\mu,\lambda}(P_{\mu,\lambda}\times \Int(D_{\mu}))
	\]
	as a stratified space.
 \end{enumerate}

 The space $P_{\mu,\lambda}$ is called the \emph{parameter space} for the
 inclusion $e_{\mu}\subset \overline{e_{\lambda}}$. When $\mu=\lambda$,
 we define $P_{\lambda,\lambda}$ to be a single point.
 A stellar stratified space equipped with a cylindrical structure is
 called a \emph{cylindrically normal stellar stratified space}.

 When
 the map $b_{\lambda,\mu}$ is
 an embedding for each pair $e_{\mu}\subset \overline{e_{\lambda}}$,
 the stratification is said to be \emph{strictly cylindrical}. 

 When $X$ is a stellar stratified space, we require that the normal
 stratification on each $\partial D_{\lambda}$ is 
 a coarsening\footnote{Definition \ref{morphism_of_stratified_spaces}}
 of the stratification in the definition of stellar stratified space.

\end{definition}

\begin{remark}
 The author first intended to call such a structure as ``locally
 product-like'' or 
 ``locally trivial'' cellular stratification. But it turns out the term
 ``locally trivial stratification'' is already used in \cite{Pflaum01}
 in a different sense.
\end{remark}

We require morphisms between cylindrically normal cellular or stellar
stratified 
spaces to preserve cylindrical structures.

\begin{definition}
 Let $(X,\pi_X,\Phi_X)$ and $(Y,\pi_Y,\Phi_Y)$ be cylindrically normal
 cellular stratified spaces with cylindrical structures given
 by 
 $\{b^X_{\mu,\lambda} : P^X_{\mu,\lambda}\times D_{\mu}\to D_{\lambda}\}$
 and
 $\{b^Y_{\alpha,\beta} : P^Y_{\alpha,\beta}\times D_{\alpha}\to
 D_{\beta}\}$, respectively.

 A \emph{morphism of cylindrically normal cellular stratified spaces}
 from $(X,\pi_X,\Phi_X)$ to $(Y,\pi_Y,\Phi_Y)$ is a morphism of stellar
 stratified spaces 
 \[
 \bm{f}=(f,\underline{f}) : (X,\pi_X,\Phi_X) \longrightarrow
 (Y,\pi_Y,\Phi_Y) 
 \]
 together with maps
 $f_{\mu,\lambda} : P_{\mu,\lambda}^{X} \to
 P^Y_{\underline{f}(\mu),\underline{f}(\lambda)}$ making the diagram
 \[
  \begin{diagram}
   \node{P^X_{\mu,\lambda}\times D_{\mu}}
   \arrow{e,t}{f_{\mu,\lambda}\times f_{\mu}}
   \arrow{s,l}{b^X_{\mu,\lambda}} 
   \node{P^Y_{\underline{f}(\mu),\underline{f}(\lambda)}\times
   D_{\underline{f}(\mu)}}
   \arrow{s,r}{b^Y_{\underline{f}(\mu),\underline{f}(\lambda)}} \\  
   \node{D_{\lambda}} \arrow{e,b}{f_{\lambda}}
   \node{D_{\underline{f}(\lambda)}.} 
  \end{diagram}
 \]

 Morphisms of cylindrically normal stellar stratified spaces are defined
 analogously. 

 The categories of cylindrically normal cellular and stellar stratified
 spaces are denoted by $\category{CSSpaces}^{\cyl}$ and
 $\category{SSSpaces}^{\cyl}$, respectively. 
\end{definition}

For cylindrically normal cellular/stellar stratified spaces, the
relative-compactness of closures of cells can be easily verified.

\begin{lemma}
 \label{compact_parameter_space}
 Let $X$ be a cylindrically normal cellular (stellar) stratified space
 with parameter spaces $\{P_{\mu,\lambda}\}_{\mu\le\lambda}$. A cell
 $\varphi_{\lambda} : D_{\lambda}\to\overline{e_{\lambda}}$ is
 relatively compact if and only if $P_{\mu,\lambda}$ is compact for each
 $\mu\le\lambda$. 
\end{lemma}

\begin{proof}
 For $y\in \partial e_{\lambda}$, there exists $\mu\le \lambda$ with
 $y\in e_{\mu} \subset \partial e_{\lambda}$. In the commutative diagram
 \[
  \begin{diagram}
   \node{D_{\lambda}} \arrow[2]{e,t}{\varphi_{\lambda}} \node{}
   \node{\overline{e_{\lambda}}} \\
   \node{P_{\mu,\lambda}\times \Int(D_{\mu})} \arrow{e}
   \arrow{n,l}{b_{\mu,\lambda}|_{P_{\mu,\lambda}\times\Int(D_{\mu})}}
   \node{\Int(D_{\mu})} \arrow{e,b}{\varphi_{\mu}|_{\Int(D_{\mu})}}
   \node{e_{\mu},} \arrow{n,J}
  \end{diagram}
 \]
 the restriction $b_{\mu,\lambda}|_{P_{\mu,\lambda}\times\Int(D_{\mu})}$
 is an embedding and $\partial D_{\lambda}$ is covered by the disjoint
 union of such images. Thus we have
 $\varphi_{\lambda}^{-1}(y) \cong P_{\mu,\lambda}\times\{y\}$ and the
 result follows from Corollary \ref{compact_fiber} and the assumption. 
\end{proof}

\begin{corollary}
 Let $X$ be a CW cylindrically normal cellular (stellar) stratified
 space. If all parameter spaces are compact, $X$ is paracompact.
\end{corollary}

%
%

%% file: locally_polyhedral.tex
\subsection{Polyhedral Cellular Stratified Spaces}
\label{locally_polyhedral}

Although total normality and cylindrical normality play
central roles in our study of cellular and stellar stratified spaces, it
is not easy to prove that a given cellular stratification is totally
normal or cylindrically normal. 

In order to address this problem, we found polyhedral complexes and PL
maps are useful.
It turns out that PL maps also play an important role in our proof of
Theorem \ref{deformation_theorem}. In
this section, we generalize polyhedral complexes and
introduce the notion of polyhedral cellular stratified spaces. 

Let us first recall the definition of polyhedral complexes. See \S2.2.4
of Kozlov's book \cite{KozlovCombinatorialAlgebraicTopology} for
details.  

\begin{definition}
  \label{polyhedral_complex_definition}
 A \emph{Euclidean polyhedral complex} is a subspace $K$ of
 $\R^N$ for some $N$ equipped with a finite family of maps
 \[
 \{\varphi_{i} : F_{i} \longrightarrow K \mid i=1,\ldots,n \}
 \]
 satisfying the conditions that
 \begin{enumerate}
  \item each $F_i$ is a convex polytope;
  \item each $\varphi_i$ is an affine equivalence onto its image;
  \item $K = \bigcup_{i=1}^n \varphi_i(\Int(F_i))$, where $\Int(F_i)$ is
	the relative interior\footnote{The interior in its affine hull.}
	of $F_i$; 
  \item for $i\neq j$, $\varphi_i(F_i)\cap \varphi_j(F_j)$ is a proper
	face of $\varphi_i(F_i)$ and $\varphi_j(F_j)$. 
 \end{enumerate}
 The polytopes $F_i$'s are called \emph{generating polytopes}. 
\end{definition}

Obviously a polyhedral complex is a regular cell complex. By replacing
affine cell structure maps $\varphi_i$ by continuous maps, 
Kozlov defined a more general kind of polyhedral complexes in his
book.
The requirement of being a subspace of $\R^N$ can be removed if we
assume cylindrical normality. We may also remove the condition that
all cells are closed. 
Here is our definition of ``polyhedral'' structure.

\begin{definition}
 \label{polyhedral_normality_definition}
 A \emph{polyhedral stellar stratified space} consists of
 \begin{itemize}
  \item a CW cylindrically normal stellar stratified space
	$X$, 
  \item a family of Euclidean polyhedral complexes
	$\widetilde{F}_{\lambda}$ indexed by $\lambda\in P(X)$ and
  \item a family of homeomorphisms
	$\alpha_{\lambda}:\widetilde{F}_{\lambda}\to\overline{D_{\lambda}}$
	indexed by $\lambda\in P(X)$, where $\overline{D_{\lambda}}$ is
	the closure of the domain stellar cell $D_{\lambda}$ for
	$e_{\lambda}$ in a disk containing $D_{\lambda}$,
 \end{itemize}
 satisfying the following conditions:
 \begin{enumerate}
  \item For each cell $e_{\lambda}$, 
	$\alpha_{\lambda}:\widetilde{F}_{\lambda}\to\overline{D_{\lambda}}$
	is a subdivision of stratified space, where the
	stratification on $\overline{D_{\lambda}}$ is defined by
	the cylindrical structure\footnote{Definition
	\ref{cylindrical_structure_definition}}.  

  \item For each pair $e_{\mu}<e_{\lambda}$, the parameter
	space $P_{\mu,\lambda}$ is a locally cone-like
	space\footnote{Definition
	\ref{locally_cone-like_space_definition}} and  
	the composition 
	\[
	P_{\mu,\lambda}\times F_{\mu}
	\rarrow{1\times\alpha_{\mu}} P_{\mu,\lambda}\times D_{\mu}
	\rarrow{b_{\mu,\lambda}} 
	D_{\lambda} \rarrow{\alpha_{\lambda}^{-1}} F_{\lambda}
	\]
	is a PL map\footnote{Definition \ref{PL_map_definition}}, where
	$F_{\lambda}=\alpha_{\lambda}^{-1}(D_{\lambda})$. 

 \end{enumerate}

 Each $\alpha_{\lambda}$ is called a \emph{polyhedral replacement} of
 the cell structure map of $e_{\lambda}$. The collection
 $A=\{\alpha_{\lambda}\}_{\lambda\in P(X)}$ is called a
 \emph{polyhedral structure} on $X$. 
\end{definition}

\begin{remark}
 When $X$ is cellular, $\overline{D_{\lambda}}= D^{\dim e_{\lambda}}$.
\end{remark}

\begin{remark}
 An analogous structure for cell complexes is introduced by A.~Kirillov,   
 Jr.\ in \cite{1009.4227} as PLCW complexes for a different purpose. 
 $M_{\kappa}$-polyhedral complexes in the book \cite{Bridson-Haefliger}
 by Bridson and Haefliger are also closely related.
\end{remark}

\begin{definition}
 A \emph{morphism of polyhedral stellar stratified spaces} from
 $(X,\pi,\Phi,A)$ to $(X',\pi',\Phi',A')$ consists of 
 a morphism of stellar stratified spaces
 \[
 (\bm{f},\{f_{\lambda}\}) : (X,\pi,\Phi) \longrightarrow
 (X',\pi',\Phi') 
 \] and
 a family of PL maps
 $\tilde{f}_{\lambda} : \widetilde{F}_{\lambda} \to \widetilde{F}'_{\underline{f}(\lambda)}$
 for $\lambda\in P(X)$ that are compatible with
 polyhedral structures.
\end{definition}

Any polyhedral complex is polyhedral. More generally, 
we have the following criterion.

\begin{lemma}
 \label{Euclidean_and_PL_imply_locally_polyhedral}
 Let $X$ be a subspace of $\R^N$ equipped with a structure of
 cylindrically normal CW stellar stratified space whose parameter spaces 
 $P_{\mu,\lambda}$ are locally cone-like spaces. Suppose, for each
 $e_{\lambda}\in P(X)$, there exists a polyhedral complex
 $\widetilde{F}_{\lambda}$ 
 and a homeomorphism 
 $\alpha_{\lambda}:\widetilde{F}_{\lambda} \to D^{\dim e_{\lambda}}$
 such that the composition
 \[
 P_{\mu,\lambda}\times F_{\mu} \rarrow{1\times\alpha_{\mu}}
 P_{\mu,\lambda}\times D_{\mu} \rarrow{b_{\mu,\lambda}}
 D_{\lambda} \rarrow{\varphi_{\lambda}} X \hookrightarrow \R^N 
 \]
 is a PL map, where
 $F_{\lambda}=\alpha_{\lambda}^{-1}(D_{\lambda})$. Suppose further that
 $\alpha_{\lambda} : \widetilde{F}_{\lambda}\to D^{\dim e_{\lambda}}$ is 
 a cellular subdivision, where the cell decomposition on
 $D^{\dim e_{\lambda}}$ is the one in the definition of cylindrical
 structure.  Then the collection  
 $\{\alpha_{\lambda}\}_{\lambda\in P(X)}$ defines a polyhedral
 structure on $X$.
\end{lemma}

\begin{proof}
 For each pair $e_{\mu}< e_{\lambda}$, define
 $\tilde{b}_{\mu,\lambda} : P_{\mu,\lambda}\times F_{\mu}\to F_{\lambda}$
 to be the composition 
 \[
 \tilde{b}_{\mu,\lambda} : P_{\mu,\lambda} \times F_{\mu}
 \rarrow{1\times \alpha_{\mu}} P_{\mu,\lambda}\times D_{\mu}
 \rarrow{b_{\mu,\lambda}}  
 D_{\lambda} \rarrow{\alpha_{\lambda}^{-1}} F_{\lambda}.
 \]
 Note that, when $\mu=\lambda$, $b_{\mu,\lambda}$ is the
 identity map. Thus the top horizontal composition in the
 following diagram is a PL map:
 \[
  \begin{diagram}
   \node{F_{\lambda}} \arrow{e,t}{\alpha_{\lambda}} \node{D_{\lambda}}
   \arrow{e,t}{\varphi_{\lambda}} \node{X} \arrow[2]{e,J} 
   \node{} \node{\R^N} \\ 
   \node{P_{\mu,\lambda}\times F_{\mu}}
   \arrow{n,l}{\tilde{b}_{\mu,\lambda}}
   \arrow{e,b}{1\times\alpha_{\mu}}
   \node{P_{\mu,\lambda}\times D_{\mu}}
   \arrow{e,b}{b_{\mu,\lambda}} \node{D_{\lambda}}
   \arrow{e,b}{\varphi_{\lambda}} 
   \node{X} \arrow{e,J} \node{\R^N.} \arrow{n,=}
  \end{diagram}
 \]
 The bottom horizontal composition is also a PL map by assumption and
 $\tilde{b}_{\mu,\lambda}$ is an embedding when restricted to
 $P_{\mu,\lambda}\times \alpha_{\mu}^{-1}(\Int D^{\dim e_{\mu}})$. Thus 
 $\tilde{b}_{\mu,\lambda}$ is also PL by Lemma \ref{PL_map_is_extendalbe}.
%
\end{proof}



\begin{example}
 Consider the minimal cell decomposition $S^2=e^0\cup e^2$. We have an
 embedding $f : S^2 \to \R^3$ whose image is the boundary
 $\partial\Delta^{3}$ of the standard $3$-simplex and $f(e^0)$ is a
 vertex of $\partial\Delta^3$. This is a cylindrically normal cellular
 stratified space by Example \ref{sphere_is_cylindrically_regular}. Let
 $P$ be a $2$-dimensional polyhedral complex in $\R^2$ described in
 Figure \ref{local_polyhedral_structure_on_sphere}. 
 \begin{figure}[ht]
 \begin{center}
  \begin{tikzpicture}
   \draw (0,2.25) -- (1.25,0);
   \draw (0,2.25) -- (-1.25,0);
   \draw (-1.25,0) -- (1.25,0);
   \draw (0,1.5) -- (-0.5,0.5);
   \draw (0,1.5) -- (0.5,0.5);
   \draw (-0.5,0.5) -- (-1.25,0);
   \draw (0.5,0.5) -- (1.25,0);
   \draw (-0.5,0.5) -- (0.5,0.5);
   \draw (0,2.25) -- (0,1.5);
  \end{tikzpicture}
 \end{center}
  \caption{A polyhedral structure on $S^2$}
  \label{local_polyhedral_structure_on_sphere}
 \end{figure}
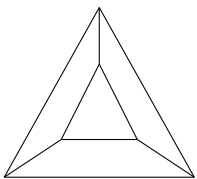
 By collapsing the outer triangle, we obtain a map
 \[
  \psi : P \longrightarrow \partial\Delta^3
 \]
 whose restriction to the interior is a homeomorphism. Let
 $\alpha : P \to D^2$ be a homeomorphism given by a radial
 expansion. Then 
 maps $f$ and $\varphi$ can be chosen in such a way they make the
 following diagram commutative
 \[
  \begin{diagram}
   \node{D^2} \arrow{e,t}{\varphi} \node{S^2}
   \arrow{s,r}{f} \\
   \node{P} \arrow{n,l}{\alpha} \arrow{e,b}{\psi}
   \node{\partial\Delta^3} 
   \arrow{e,J} \node{\R^3.}
  \end{diagram}
 \]
 By Lemma \ref{Euclidean_and_PL_imply_locally_polyhedral}, we obtain a
 structure of polyhedral cellular stratified space on $S^2$ or
 $\partial\Delta^3$. 
\end{example}

\begin{definition}
 A cellular stratified space satisfying the assumption of Lemma
 \ref{Euclidean_and_PL_imply_locally_polyhedral} 
 is called a \emph{Euclidean polyhedral cellular stratified
 space}. 
\end{definition}

Totally normal cellular stratified spaces form an important class of
polyhedral cellular stratified spaces.

\begin{lemma}
 \label{total_normality_implies_local_polyhedrality}
 Any CW totally normal cellular stratified space has a 
 polyhedral structure. 
\end{lemma}

\begin{proof}
 Let $X$ be a CW totally normal cellular stratified space. By definition,
 for each cell
 $\varphi_{\lambda} : D_{\lambda} \to \overline{e}_{\lambda}$ in $X$,
 there exists a regular cell decomposition of $D^{\dim e_{\lambda}}$
 containing $D_{\lambda}$ as a cellular stratified subspace. Since the
 barycentric subdivision of a regular cell complex has a structure of
 simplicial complex, $D^{\dim e_{\lambda}}$ can be embedded in a
 Euclidean space as a finite simplicial complex
 $\widetilde{F}_{\lambda}$. By induction on 
 dimensions of cells, we may choose
 homeomorphisms
 $\{\alpha_{\lambda} : \widetilde{F}_{\lambda}\to D^{\dim
 e_{\lambda}}\}_{\lambda\in P(X)}$ in such a way the composition
 \[
 F_{\mu}
 \rarrow{\alpha_{\mu}} D_{\mu}
 \rarrow{b} 
 D_{\lambda} \rarrow{\alpha_{\lambda}^{-1}} F_{\lambda}  
 \]
 is a PL map for each lift $b : D_{\mu}\to D_{\lambda}$ of the
 cell structure map of $e_{\mu}$ for each pair
 $e_{\mu}<e_{\lambda}$. Thus we obtain a polyhedral structure.
\end{proof}

Conversely, the polyhedrality condition provides us with a useful
criterion for a stratification to be totally normal.
The following fact first appeared in \cite{1009.1851v5}.

\begin{proposition}
 \label{locally_polyhedral_implies_totally_normal}
 Let $X$ be a normal CW complex embedded in $\R^N$ for some $N$. Suppose
 for each cell $e\subset X$ with cell structure map $\varphi$, there
 exists a Euclidean polyhedral complex $F$ and a homeomorphism
 $\alpha : F \to D^{\dim e}$ such that the composition
 \[
  F \rarrow{\alpha} D^{\dim e} \rarrow{\varphi} X \subset \R^N
 \]
 is a PL map. Then any regular cellular stratified subspace of $X$ is  
 polyhedral.
\end{proposition}

\begin{proof}
 Let $A$ be a regular cellular stratified subspace of $X$. It suffices
 to find polyhedral replacements of cell structures for $A$. For an
 $n$-cell $e$ in $A$, let $\varphi : D^n\to X$ be the cell structure of
 $e$ in $X$. The cell structure for $e$ in $A$ is, by definition, given
 by
 \[
  \varphi_{A} = \varphi|_{D_A} : D_{A}=\varphi^{-1}(\overline{e}\cap A)
 \longrightarrow A.
 \]
 By assumption, there exists a polyhedral complex $P$ and a
 homeomorphism $\alpha : P\to D^n$ with
 $\varphi\circ\alpha : P\to X\hookrightarrow \R^{N}$ a PL map.
\end{proof}

\begin{example}
 The product cell decomposition of $(S^n)^k$ induced by the minimal cell
 decomposition on $S^n$ is polyhedral.
 In \cite{1009.1851v5}, a subdivision of the product cell decomposition
 containing $\Conf_k(S^n)$ as a stratified subspace
 was defined by using the stratification on $(\R^n)^{\ell}$  associated
 with the braid arrangement $\mathcal{A}_{\ell-1}$ for $1\le \ell\le k$.

 The author expected that the induced stratification on $\Conf_{k}(S^n)$
 to be totally normal. It turned out that the stratification is much
 more complicated than the author imagined.
\end{example}

We propose the following strategy to prove a given
cellular stratification on $X$ is cylindrically normal or totally
normal: 
\begin{enumerate}
 \item Embed $X$ into a cylindrically normal or totally normal cell
       complex $\widetilde{X}$.
 \item Find an appropriate subdivision of $\widetilde{X}$ which includes
       $X$ as a stratified subspace.
\end{enumerate}

In fact, it is easy to prove that, if $X$ is polyhedral, $X$ can
be embedded in a cell complex as a stratified subspace.

%

\begin{lemma}
 Let $X$ be a polyhedral cellular
 stratified space.   
 Given a pair of cells $e_{\mu}\subset\partial e_{\lambda}$, the
 structure map  
 \[
 b_{\mu,\lambda} : P_{\mu,\lambda}\times D_{\mu} \rarrow{}
 D_{\lambda}
 \]
 has a unique extension to the whole disks
 \[
 \overline{b_{\mu,\lambda}} : P_{\mu,\lambda} \times D^{\dim e_{\mu}}
 \longrightarrow D^{\dim e_{\lambda}}.
 \]
\end{lemma}

\begin{proof}
 By Lemma \ref{PL_map_is_extendalbe}.
\end{proof}

These maps allow us to construct a canonical closure of any
polyhedral cellular stratified space.

\begin{definition}
 \label{cellular_closure_definition}
 Let $X$ be a polyhedral cellular stratified space with  
 cells $\{e_{\lambda}\}_{\lambda\in P(X)}$. Define a
 cell complex $U(X)$ by
 \[
  U(X) = \left.\left(\coprod_{\lambda\in P(X)} D^{\dim
 e_{\lambda}}\right)\right/_{\displaystyle \sim}
 \]
 where the relation $\sim$ is the equivalence relation generated by the 
 following relation: For
 $x\in D^{\dim e_{\lambda}}$ and $y\in D^{\dim\mu}$, $x\sim y$ if
 $e_{\mu}\subset \partial e_{\lambda}$ and and there exists
 $z\in P_{\mu,\lambda}$ such that $\overline{b_{\mu,\lambda}}(z,y)=x$.

 There is a canonical inclusion
 \[
  i : X \hookrightarrow U(X).
 \]
 This space $U(X)$ is called the \emph{cellular closure} of $X$.
\end{definition}

By definition, we have the following.

\begin{lemma}
 \label{cellular_closure}
 When $X$ is an polyhedral cellular
 stratified space, $U(X)$  
 is a cylindrically normal CW complex containing $X$ as a cellular
 stratified subspace. 
\end{lemma}

\begin{example}
 In the case of the punctured torus in Example \ref{punctured_torus},
 $U(X)$ is obtained by gluing parallel edges in $I^2$ and is homeomorphic
 to $T^2$. In other words, $U(X)$ is obtained by closing the hole in $X$.
\end{example}

The existence of cellular closure implies that cell structures of
a polyhedral cellular stratified space are
bi-quotient\footnote{Definition \ref{bi-quotient_definition}.}.

\begin{corollary}
 \label{characteristic_map_of_locally_polyhedral}
 Let $X$ be a polyhedral cellular stratified space. Then any
 cell structure
 $\varphi_{\lambda} : D_{\lambda}\to \overline{e_{\lambda}}$ is
 bi-quotient. 
\end{corollary}

\begin{proof}
 For an $n$-cell
 $\varphi_{\lambda} : D_{\lambda}\to \overline{e_{\lambda}}$ in $X$, let
 $\widetilde{\varphi_{\lambda}} : D^n \to U(X)$ the extension. Since
 $\widetilde{\varphi_{\lambda}}$ is proper, it is bi-quotient and hence
 is hereditarily quotient. By Lemma
 \ref{hereditarily_quotient_map_can_be_restricted}, $\varphi_{\lambda}$
 is also hereditarily quotient.

 For $y \in \overline{e_{\lambda}}\subset X$, the fiber
 $\varphi_{\lambda}^{-1}(y)$ can be identified with one of parameter
 spaces by the proof of Lemma \ref{compact_parameter_space},
 which is a cellular stratified subspace of a regular cell decomposition
 of $\partial D^n$. Thus the boundary
 $\partial\varphi_{\lambda}^{-1}(y)$ is compact. The result follows from
 3 in Lemma \ref{sufficient_conditions_for_bi-quotient}.
\end{proof}

\begin{corollary}
 Any polyhedral cellular stratified space is paracompact.
\end{corollary}

%% file: cylindrical_examples.tex
\subsection{Examples of Cylindrically Normal Cellular Stratified Spaces}
\label{cylindrical_examples}

Here is a collection of examples and nonexamples of cylindrically normal
cellular stratified spaces.

\begin{example}
 \label{totally_normal_implies_cylindrically_regular}
 A stellar stratified space $X$ is totally normal if and
 only if it is strictly cylindrically normal and each parameter space
 $P_{\mu,\lambda}$ is a finite set (with discrete topology).  

 Consider the cell decomposition of $D^3$ in Example \ref{PLCW}. It is
 easily seen to be cylindrically normal with 
 finite parameter spaces. However, it is not strictly cylindrical, as we
 have seen in Example \ref{PLCW}. In other words, a cylindrically
 normal cell complex with finite parameter spaces is a PLCW complex, if
 it satisfies the PL requirement in the definition of PLCW complexes.
\end{example}

\begin{example}
 \label{sphere_is_cylindrically_regular}
 Consider the minimal cell decomposition of $S^2=\CP^1$ in Example
 \ref{minimal_cell_decomposition_of_S2}.
 The trivial stratification on $\partial D^2$ and the canonical
 inclusion 
 \[
  b_{0,2} : S^1\times D^0 \longrightarrow \partial D^2\subset D^2
 \]
 define a cylindrical structure on $S^2$ with $P_{0,2}=S^1$, for we have
 a commutative diagram
 \[
  \begin{diagram}
   \node{D^2} \arrow{e,t}{\varphi_2} \node{\CP^1} \\
   \node{S^1\times D^0} \arrow{e,b}{\pr_2} \arrow{n,l}{b_{0,2}}
   \node{D^0.} \arrow{n,r}{\varphi_0} 
  \end{diagram}
 \]

 We may also consider $S^2$ as a stellar stratified space, by choosing a
 regular cell decomposition of $\partial D^2$. For example, we may use
 the minimal regular cell decomposition of $\partial D^2$.

 Then we have embeddings
 \begin{eqnarray*}
  b_{0,+} & : & D^0\times D^0 \longrightarrow \partial D^2 \\
  b_{0,-} & : & D^0\times D^0 \longrightarrow \partial D^2 \\
  b_{1,+} & : & D^0\times D^1 \longrightarrow \partial D^2 \\
  b_{1,-} & : & D^0\times D^1 \longrightarrow \partial D^2
 \end{eqnarray*}
 corresponding to cells $e^0_+\cup e^0_{-}\cup e^1_+\cup e^1_{-}$ in
 $\partial D^2$. And we obtain a cylindrical structure as a stellar
 stratified space.
\end{example}

\begin{example}
 \label{CP^2}
 Let us extend the cylindrically normal cell decomposition on
 $S^2=\CP^1$ to $\CP^2$. The minimal cell decomposition of $\CP^2$ is
 given by
 \[
  \CP^2 = S^2\cup e^4 = e^0 \cup e^2\cup e^4.
 \]
 Consider the cell structure map of the $4$-cell
 \[
  \varphi_4=\tilde{\eta} : D^4 \longrightarrow \CP^2
 \]
 whose restriction to the boundary is the Hopf map
 \[
  \eta : S^3 \longrightarrow S^2.
 \]
 This is a fiber bundle with fiber $S^1$ and thus the cell decomposition
 $S^2=e^0\cup e^2$ induces a decomposition
 \[
  S^3 \cong e^0\times S^1 \cup e^2\times S^1,
 \]
 as we have seen in Example \ref{stratification_on_Hopf_bundle}.

 Let
 \begin{eqnarray*}
  \varphi_0 : D^0 \to \overline{e^0} \subset S^2 \\
  \varphi_2 : D^2 \to \overline{e^2} \subset S^2
 \end{eqnarray*}
 be the cell structure maps of $e^0$ and $e^2$, respectively. We have a
 trivialization 
 \[
 t : \varphi_2^*(S^3) \rarrow{\cong} D^2\times S^1.
 \]
 Let
 $b_{2,4} : S^1\times D^2 \to S^3=\partial D^4$ be the composition
 \[
 S^1\times D^2 \rarrow{t^{-1}} \varphi_2^*(S^3)
 \rarrow{\tilde{\varphi}_2} S^3, 
 \]
 then we have the following commutative diagram
 \[
  \begin{diagram}
   \node{S^3} \arrow{e,=} \node{S^3} \arrow{e,t}{\eta} \node{\CP^1} \\
   \node{} \node{\varphi_2^*(S^3)} \arrow{n,l}{\tilde{\varphi}_2}
   \arrow{e} \node{D^2} 
   \arrow{n,r}{\varphi_2} \\ 
   \node{S^1\times D^2} \arrow[2]{n,l}{b_{2,4}} \arrow{ne,t}{t^{-1}}
   \arrow[2]{e,b}{\pr_2} \node{} \node{D^2.} \arrow{n,=}
  \end{diagram}
 \]
 Let $b_{0,4} : S^1\times D^0 \to S^3$ be the inclusion of the fiber
 over $e^0$. Then we have
 \[
  \partial D^4 = b_{0,4}(S^1\times D^0) \cup b_{2,4}(S^1\times D^2).
 \]
 Let $P_{0,2} = P_{2,4} = P_{0,4}=S^1$ and define
 \[
  c_{0,2,4} : P_{2,4}\times P_{0,2} \to P_{0,4}
 \]
 by the multiplication of $S^1$.

 Let us check that the above data define a cylindrical structure on
 $\CP^2=e^0\cup e^2\cup e^4$. It remains to verify the commutativity of
 the diagram
 \[
  \begin{diagram}
   \node{P_{2,4}\times P_{0,2}\times D^0} \arrow{e,t}{1\times b_{0,2}}
   \arrow{s,l}{c\times 1}
   \node{P_{2,4}\times D^2} \arrow{s,r}{b_{2,4}} \\
   \node{P_{0,4}\times D^0} \arrow{e,b}{b_{0,4}} \node{D^4.}
  \end{diagram}
 \]
 In other words, we need to show the restriction of $b_{2,4}$ to
 \[
  b_{2,4}|_{S^1\times S^1} : S^1\times S^1 \rarrow{} S^3=\partial D^4
 \]
 is given by the multiplication of $S^1$ followed by the inclusion of
 the fiber $\eta^{-1}(e^0)$. Recall that the Hopf map $\eta$ is given by
 \[
  \eta(z_1,z_2) = (2|z_1|^2-1,2z_1\bar{z}_2),
 \]
 where we regard
 \begin{eqnarray*}
  S^2 & = & \{(x,z)\in\R\times\bbC \mid x^2+|z|^2=1\} \\
  S^3 & = & \{(z_1,z_2)\in\bbC^2 \mid |z_1|^2+|z_2|^2=1\},
 \end{eqnarray*}
 and, on $U_+=S^2-\{(1,0)\}$, the local trivialization
 \[
 \varphi_{+} : \eta^{-1}(U_+) \longrightarrow U_{+}\times S^1
 \]
 is given by
 \[
 \varphi_{+}(z_1,z_2) = \left(2|z_1|^2-1,2z_1\bar{z}_2,
 \frac{z_2}{|z_2|}\right). 
 \]
 The inverse of $\varphi_{+}$ is given by
 \[
  \varphi_{+}^{-1}(x,z,w) = \left(\frac{zw}{2\sqrt{\frac{1-x}{2}}},
 w\sqrt{\frac{1-x}{2}}\right). 
 \]
 Define
 \[
 \mathrm{wrap} : D^2 \longrightarrow S^2
 \]
 by
 \[
  \mathrm{wrap}(z) = \left(2|z|-1,2\sqrt{|z|(1-|z|)}\frac{z}{|z|}\right);
 \]
 then the restriction of $\mathrm{wrap}$ to $\Int(D^2)$ is a
 homeomorphism onto $S^2\setminus\{(1,0)\}$. The map $b_{2,4}$ is
 defined by the composition
 \[
  S^1\times\Int(D^2) \rarrow{1\times\mathrm{wrap}}
 S^1\times U_+ \cong  U_+\times S^1 \rarrow{\varphi_+^{-1}}
 \eta^{-1}(U_+) \hookrightarrow S^3, 
 \]
 which is given by
 \begin{eqnarray*}
  b_{2,4}(w,z) & = &
   \varphi_{+}^{-1}\left(2|z|-1,2\sqrt{|z|(1-|z|)}\frac{z}{|z|},
		    w\right) \\ 
  & = &
   \left(\frac{2\sqrt{|z|(1-|z|)}\frac{z}{|z|}w}{2\sqrt{\frac{1-(2|z|-1)}{2}}}, 
	 w\sqrt{\frac{1-(2|z|-1)}{2}}\right) \\
  & = & \left(\frac{zw}{\sqrt{|z|}}, w\sqrt{1-|z|}\right).
 \end{eqnarray*}
 From this calculation, we see $b_{2,4}(w,z)\to (zw,0)$ as $|z|\to 1$.

 Thus this is a cylindrically normal cellular stratification.
\end{example}

\begin{example}
 \label{moment_angle_complex}
 There is an alternative way of describing the cylindrical structure in
 the above Example.
 Recall that complex projective spaces are typical examples of
 quasitoric manifolds. Define an action of $T^n=(S^1)^n$ on $\CP^n$ by
 \[
  (t_1,\ldots,t_n)\cdot [z_0,\ldots,z_n] = [z_0,t_1z_1,\ldots,t_nz_n].
 \]
 As we have seen in Example
 \ref{stratification_by_group_action}, this action 
 induces a stratifications on $\CP^n$ which descends to  $\CP^n/T^n$
 \begin{eqnarray*}
  \pi_{\CP^n} & : & \CP^n \longrightarrow I(T^n) \\
  \pi_{\CP^n/T^n} & : & \CP^n/T^n \longrightarrow I(T^n).
 \end{eqnarray*}
 The quotient space $\CP^n/T^n$ is known to be homeomorphic to
 $\Delta^n$ and the stratification $\pi_{\CP^n/T^n}$ can be identified
 with the stratification $\pi_n$ on $\Delta^n$ in Example
 \ref{stratifications_on_simplex}. This stratification, however, does
 not induce the minimal cell decomposition of $\CP^n$.
 The other stratification $\pi_n^{\max}$ on $\Delta^n$ defined in
 Example \ref{stratifications_on_simplex} induces the minimal cell
 decomposition on $\CP^n$ by the composition
 \[
  \CP^n \rarrow{p} \CP^n/T^n \cong \Delta^n \rarrow{\pi_n^{\max}} [n]. 
 \]

 Let us show that this cell decomposition is cylindrically normal. To
 this end, we first rewrite $\CP^n$ by using the construction introduced
 in \cite{Davis-Januszkiewicz91} by Davis and Januszkiewicz. Given a
 simple polytope\footnote{A $d$-dimensional convex polytope is said to
 be \emph{simple} if each vertex is adjacent to exactly $d$ edges.}
 $P$ of dimension $n$ and a function
 $\lambda: \{\text{codimension-1 faces in }P\}\to \bbC^n$ satisfying
 certain conditions, they
 constructed a space $M(\lambda)$ with $T^n$-action. Suppose
 $P=\Delta^n$ and define
 \[
  \lambda_n(C_i) = \begin{cases}
		  (1,\cdots,1), & i=0 \\
		  (\underbrace{0,\cdots,0}_{i-1},1,
		  \underbrace{0,\cdots,0}_{n-i}), & i=1,\cdots,n, 
		 \end{cases}
 \]
 where $C_i$ is the codimension-$1$ face with vertices in $[n]-\{i\}$. 
 In this case, $M(\lambda_n)$ can be described as
 \[
  M(\lambda_n) = (T^n\times \Delta^n)/_{\sim},
 \]
 where the equivalence relation $\sim$ is generated by the
 following relations: Let $p=(p_0,\ldots,p_n)\in \Delta^n$.
 \begin{itemize}
  \item When $p_i=0$ for $1\le i\le n$, 
	\[
	(t_1,\ldots,t_i,\ldots,t_n;p_0,\ldots,p_n) \sim
	(t_1,\ldots,t'_i,\ldots,t_n;p_0,\ldots,p_n) 
	\]
	for any $t_i,t_i'\in S^1$.
  \item When $p_0=0$,
	\[
	 (t_1,\ldots,t_n;0,p_1,\ldots,p_n) \sim (\omega
	t_1,\ldots,\omega t_n;0,p_1,\ldots,p_n)
	\]
	for any $\omega\in S^1$.
 \end{itemize}
 An explicit homeomorphism $p_n : \CP^n \to M(\lambda_n)$ and its
 inverse $q_n$ are given by
 \begin{eqnarray*}
  p_n([z_0:\ldots:z_n]) & = & \begin{cases}
			     \left[\frac{z_1/z_0}{|z_1/z_0|}, \ldots,
			       \frac{z_n/z_0}{|z_n/z_0|};
			       \frac{|z_0|^2}{\sum_{i=0}^n|z_i|^2},
			       \ldots,
			       \frac{|z_n|^2}{\sum_{i=0}^n|z_i|^2}
			       \right] & z_0\neq 0 \\ 
			       \left[\frac{z_1}{|z_1|},\ldots,
			       \frac{z_n}{|z_n|}; 0,
			       \frac{|z_1|^2}{\sum_{i=0}^n|z_i|^2}, 
			       \ldots,
			       \frac{|z_n|^2}{\sum_{i=0}^n|z_i|^2}
			       \right] & z_0=0 
			    \end{cases} \\
  q_n([z_1,\ldots,z_n;x_0,\ldots,x_n]) & = & [\sqrt{x_0},\sqrt{x_1}z_1,
   \ldots, \sqrt{x_n}z_n].
 \end{eqnarray*}

 Under this identification, the minimal cell decomposition on $\CP^n$
 can be described as
 \[
  \CP^n\cong M(\lambda_n) = \bigcup_{i=1}^n 
 \left(T^n/T^{n-i}\times (\Delta^{i}\setminus\Delta^{i-1})\right)/_{\sim}.
 \]
 Regard
 $D^{2n}=\rset{(z_1,\ldots,z_n)\in \bbC^n}{\sum_{i=1}^n|z_i|^2\le 1}$ and
 define a map
 \[
  \varphi_{2n} : D^{2n} \longrightarrow M(\lambda_n)
 \]
 by
 \[
 \varphi_{2n}(z_1,\ldots,z_n) = \left[\frac{z_1}{|z_1|},\ldots,
 \frac{z_n}{|z_n|}; 1-\sum_{i=1}^n |z_i|^2,
 |z_1|^2, \ldots, |z_n|^2\right].  
 \]
 This is a cell structure map for the $2n$-cell. For $m<n$, define
 \[
  b_{2m,2n} : S^1\times D^{2m} \longrightarrow D^{2n}
 \]
 by
 \[
 b_{2m,2n}(\omega,z_1,\ldots,z_{m}) =
 \left(0,\ldots,0,\omega\sqrt{1-\sum_{i=1}^m |z_i|^2}, \omega
 z_1,\ldots, \omega z_{m}\right). 
 \]
 Then each $b_{2m,2n}|_{S^1\times\Int(D^{2m})}$ is a homeomorphism onto
 its image and we have a stratification
 \[
 \partial D^{2n}=S^{2n-1} = \bigcup_{m=0}^{n-1} b_{2m,2n}\left(S^1\times
 \Int(D^{2m})\right). 
 \]
 Furthermore the diagram
 \[
  \begin{diagram}
   \node{} \node{D^{2n}} \arrow{e,t}{\varphi_{2n}} \node{M(\lambda_n)}
   \\ 
   \node{S^1\times D^{2m}} \arrow{e,t}{b_{2m,2n}} \arrow{s,l}{\pr_2}
   \node{S^{2n-1}} \arrow{n,J} \node{} \\    
   \node{D^{2m}} \arrow[2]{e,b}{\varphi_{2m}}
   \node{} \node{M(\lambda_m),} \arrow[2]{n,J}
  \end{diagram}
 \]
 is commutative, where the inclusion
 $M(\lambda_m) \hookrightarrow M(\lambda_n)$ is 
 given by 
 \[
  [t_1,\ldots,t_m;p_0,\ldots,p_m] \longmapsto
 [1,\ldots,1,t_1,\ldots,t_m;0,\ldots, 0, p_0,\ldots,p_m].
 \]
 Now define $P_{2m,2n}=S^1$ for $m<n$. The group structure of $S^1$
 defines a map
 \[
  c_{2\ell,2m,2n} : P_{2m,2n}\times P_{2\ell,2n} \longrightarrow
 P_{2\ell,2n} 
 \]
making the diagram
 \[
  \begin{diagram}
   \node{S^1\times S^1\times D^{2\ell}} \arrow{e,t}{1\times
   b_{2\ell,2m}} \arrow{s,l}{c_{2\ell,2m,2n}\times 1}
   \node{S^1\times D^{2m}} \arrow{s,r}{b_{2m,2n}} \\
   \node{S^1\times D^{2\ell}} \arrow{e,b}{b_{2\ell,2n}} \node{D^{2n}}
  \end{diagram}
 \]
commutative and we have a cylindrical structure.

 More generally, Davis and Januszkiewicz \cite{Davis-Januszkiewicz91}
 proved that any quasitoric manifold $M$ of dimension $2n$ can be
 expressed as $ M \cong M(\lambda)$ for a simple convex polytope $P$ of
 dimension $n$ and a function $\lambda$. The right hand side 
 is a space constructed as a quotient of $T^n\times P$ under an
 equivalence relation analogous to the case of $\CP^n$. Davis and
 Januszkiewicz proved in \S3 of 
 their paper that there is a ``perfect Morse function'' on $M(\lambda)$
 which induces a cell decomposition of $M(\lambda)$ or $M$ with exactly
 $h_i(P)$ cells of dimension $2i$, where $(h_0(P), \ldots, h_n(P))$  is
 the $h$-vector of $P$. It seems very likely that the above construction
 of a cylindrical structure on $\CP^n$ can be extended to quasitoric
 manifolds. 
%
%
\end{example}

\begin{example}
 \label{small_cover}
 In the same paper, Davis and Januszkiewicz introduced the notion of
 small covers as a real 
 analogue of quasitoric manifolds by replacing $S^1$ by $\Z_2$. Small
 covers have many properties in common with quasitoric (or torus)
 manifolds. 

 For example, we have $\RP^n/(\Z_2)^n \cong \Delta^n$ and the
 stratification $\pi_{n}^{\max}$ on $\Delta^n$ induces the minimal cell
 decomposition of $\RP^n$. An argument analogous to the case of $\CP^n$
 can be used to prove that this stratification is totally normal.
\end{example}

%

\begin{example}
 Let $X=D^3$ and consider the cell decomposition 
 \[
  X = e^0_1\cup e^0_2 \cup e^1_1\cup e^1_2 \cup e^2_1\cup e^2_2 \cup e^3
 \]
 given as follows. 
 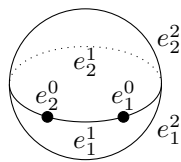
\begin{figure}[ht]
 \begin{center}
  \begin{tikzpicture}
   \draw (0,0) circle (1cm);
   \draw (-1,0) arc (180:360:1cm and 0.5cm); 
   \draw [dotted] (1,0) arc (0:180:1cm and 0.5cm); 

   \draw [fill] (0.5,-0.433) circle (2pt);
   \draw (0.5,-0.15) node {$e^0_1$};
   \draw [fill] (-0.5,-0.433) circle (2pt);
   \draw (-0.5,-0.15) node {$e^0_2$};
   \draw (0,-0.75) node {$e^1_1$};
   \draw (0,0.3) node {$e^1_2$};
   \draw (1.1,0.6) node {$e^2_2$};
   \draw (1.1,-0.6) node {$e^2_1$};
  \end{tikzpicture}
 \end{center}
  \caption{A regular cell decomposition of $S^2$}
 \end{figure}
%
 The interior of $D^3$ is the unique $3$-cell and the boundary
 $S^2$ is cut into two $2$-cells by the equator, which is cut into two
 $1$-cells by two $0$-cells on it.

 Consider the map between $S^2$
 given by collapsing the shaded region in the figure below (the wedge of
 two $2$-disks embedded in $S^2$) ``vertically'' and expanding the
 remaining part of $S^2$ continuously. Extend it to a continuous map
 $\varphi_3 : D^3\to D^3$.
 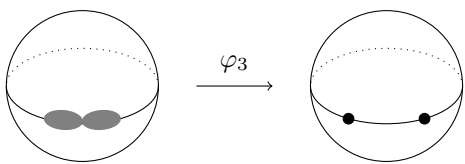
\begin{figure}[ht]
 \begin{center}
  \begin{tikzpicture}
   \draw (-4,0) circle (1cm);
   \draw (-5,0) arc (180:360:1cm and 0.5cm); 
   \draw [dotted] (-3,0) arc (0:180:1cm and 0.5cm); 

   \begin{scope}[xshift=-3.75cm, yshift=-0.45cm]
   \draw [rotate=5,fill,gray] (0,0) ellipse (0.25cm and 0.125cm);
   \end{scope}

   \begin{scope}[xshift=-4.25cm, yshift=-0.45cm]
   \draw [rotate=175,fill,gray] (0,0) ellipse (0.25cm and 0.125cm);
   \end{scope}

   \draw [->] (-2.5,0) -- (-1.5,0);
   \draw (-2,0.3) node {$\varphi_3$};

   \draw (0,0) circle (1cm);
   \draw (-1,0) arc (180:360:1cm and 0.5cm); 
   \draw [dotted] (1,0) arc (0:180:1cm and 0.5cm); 

   \draw [fill] (0.5,-0.433) circle (2pt);
   \draw [fill] (-0.5,-0.433) circle (2pt);
  \end{tikzpicture}
 \end{center}
  \caption{A 3-cell structure which is not cylindrically normal}
 \end{figure}
%
%
%
%
%
 It defines a $3$-cell structure on $e^3$. Let 
 $P_{(1,1),3}=D^1$. We have a continuous map
 \[
 b : P_{(1,1),3}\times D^1 \longrightarrow \partial D^3
 \]
 making the diagram
 \[
  \begin{diagram}
   \node{D^3} \arrow{e,t}{\varphi_3} \node{X} \\
   \node{P_{(1,1),3}\times D^1} \arrow{n,l}{b} \arrow{e} \node{D^1}
   \arrow{n,r}{\varphi_{1,1}} 
  \end{diagram}
 \]
 commutative. However, $b$ is not a homeomorphism when restricted to
 $P_{(1,1),3}\times \Int(D^1)$. And this is not a cylindrical
 structure. 
\end{example}

In Examples \ref{sphere_is_cylindrically_regular} and Example
\ref{CP^2}, the restrictions of cell structure maps to the boundary
spheres are fiber bundles onto their images. These facts seem to be
closely related to the existence of cylindrical structures in these
examples. Of course, there are many cylindrically normal cellular
stratified spaces that do not have such bundle structures. We may be
able to characterize cylindrically normal cellular stratified spaces by
using an appropriate notion of stratified fiber bundles. See
\cite{M.Davis78} for example.

\begin{example}
 Let $X$ be the subspace of the unit $3$-disk $D^3$ obtained by removing
 the interior of the $3$ -disk of radius $\frac{1}{2}$ centered at the
 origin. It has a cell decomposition with two $0$-cells, a $1$-cell, two
 $2$-cells, and a $3$-cell depicted as follows.
 \begin{figure}[ht]
 \begin{center}
  \begin{tikzpicture}
   \draw (0,0) circle (1cm);
   \draw [dotted] (-1,0) arc (180:360:1cm and 0.5cm); 
   \draw [dotted] (1,0) arc (0:180:1cm and 0.5cm); 
   \draw (0,0) circle (2cm);
   \draw [dotted] (-2,0) arc (180:360:2cm and 1cm); 
   \draw [dotted] (2,0) arc (0:180:2cm and 1cm); 

   \draw (1,0) -- (2,0);
   \draw [fill] (1,0) circle (1.5pt);
   \draw [fill] (2,0) circle (1.5pt);

   \draw (1.5,0.3) node {$e^1$};
   \draw (0.7,0) node {$e^0_1$};
   \draw (2.3,0) node {$e^0_2$};
   \draw (-0.5,0.5) node {$e^2_1$};
   \draw (-2.2,0.5) node {$e^2_2$};
   \draw (0,-1.5) node {$e^3$};
  \end{tikzpicture}
 \end{center}
  \caption{A cell decomposition of thick sphere}
 \end{figure}
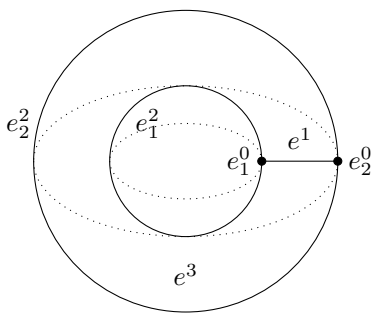
 Cells of dimension at most $1$ are regular. The
 cell structure maps for $2$-cells are given as in Example
 \ref{minimal_cell_decomposition_of_S2}. The restriction of the
 cell structure map $\varphi_3$ for the $3$-cell to $S^2$ is given by
 collapsing a middle part of a sphere to a segment.
 \begin{figure}[ht]
 \begin{center}
  \begin{tikzpicture}
   \draw (0,0) circle (2cm);
   \draw [dotted] (-1,1.73) arc (90:270:0.4cm and 1.73cm);
   \draw (-1,-1.73) arc (270:360:0.4cm and 1.73cm);
   \draw (-0.6,0) arc (0:90:0.4cm and 1.73cm);
   \draw [dotted] (1,1.73) arc (90:270:0.4cm and 1.73cm);
   \draw (1,-1.73) arc (270:360:0.4cm and 1.73cm);
   \draw (1.4,0) arc (0:90:0.4cm and 1.73cm);
   \draw [->] (0,1.8) -- (0,1);
   \draw [->] (0,-1.8) -- (0,-1);

   \draw [->] (2.5,0) -- (3,0);
   
   \draw (4,0) circle (0.5cm);
   \draw [dotted] (4,0) ellipse (0.5cm and 0.25cm);
   \draw (4.5,0) -- (5.5,0);
   \draw (6,0) circle (0.5cm);
   \draw [dotted] (6,0) ellipse (0.5cm and 0.25cm);

   \draw (7,0) node {$\cong$};

   \draw (9.5,0) circle (1cm);
   \draw [dotted] (9.5,0) ellipse (1cm and 0.5cm);
   \draw (9.5,0) circle (2cm);
   \draw [dotted] (9.5,0) ellipse (2cm and 1cm);
   \draw (10.5,0) -- (11.5,0);
  \end{tikzpicture}
 \end{center}
  \caption{A 3-cell structure on thick sphere}
 \end{figure}
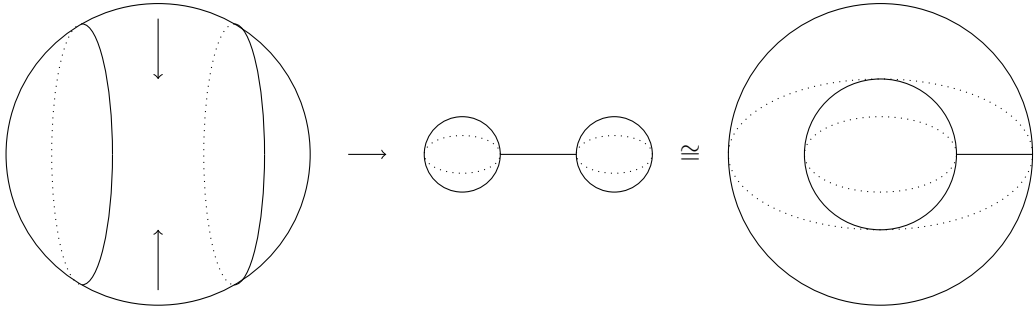

 This cell decomposition is cylindrically normal, because of the
 ``local triviality'' of the middle band. However the restriction
 $\varphi_3|_{S^2}$ is not a fiber bundle onto its image. 
\end{example}

\begin{example}
 We have seen in Example \ref{Delta-set_is_totally_normal} that the
 standard cell decomposition of the geometric realization of a
 $\Delta$-set is totally normal. Let us consider the geometric
 realization of a simplicial set $X$. Define
 \[
  P(X) = \coprod_{n=0}^{\infty} \left(X_n \setminus \bigcup_{i=0}^{n}
 s_i(X_{n-1})\right) 
 \]
 to be the set of nondegenerate simplices. For $\sigma,\tau\in P(X)$,
 define $\tau\le \sigma$ if there exists an injective morphism
 $u : [m]\to [n]$ with $X(u)(\sigma)=\tau$. Define
 \[
  \pi_X : |X| \longrightarrow P(X)
 \]
 analogously to the case of $\Delta$-sets. Then $\pi_X$ is a cellular
 stratification. In other words, cells in $|X|$ are in one-to-one
 correspondence to nondegenerate simplices.

 Suppose $\tau\le\sigma$ in $P(X)$. Then the set
 \[
 \mathcal{P}(\tau,\sigma) = \set{u : [m]\to [n]}{ u \text{ injective and }
 X(u)(\sigma)=\tau} 
 \]
 is nonempty. For $u,v\in\mathcal{P}(\tau,\sigma)$, define $u\le v$ if
 and only if $u(i)\le v(i)$ for all $i\in [m]$. Let us denote the order
 complex of this poset by $P_{\tau,\sigma}$. This is a
 simplicial complex whose simplices are indexed by chains in
 $\mathcal{P}(\tau,\sigma)$ and can be written as
 \[
  P_{\tau,\sigma} = \bigcup_{k=0}^{\infty}\bigcup_{\bm{u}\in
 N_k(\mathcal{P}(\tau,\sigma))} \Delta^k\times\{\bm{u}\}. 
 \]

 For a $k$-chain
 $\bm{u}=\{u_0<\cdots<u_k\} \in N_k(\mathcal{P}(\tau,\sigma))$, define
 a map
 \[
  \beta^{\bm{u}} : [k]\times [m] \longrightarrow [n]
 \]
 by
 \[
  \beta^{\bm{u}}(i,j) = u_i(j).
 \]
 This map induces an affine map
 \[
  b^{\bm{u}} : (\Delta^k\times\{\bm{u}\})\times\Delta^m
 \longrightarrow \Delta^n.  
 \]
 These maps can be glued together to give us a map
 \[
  b_{\tau,\sigma} : P_{\tau,\sigma}\times \Delta^m
 \longrightarrow \Delta^n. 
 \] 

 For $\sigma_0<\sigma_1<\sigma_2$ in $P(X)$, the composition 
 \[
  \mathcal{P}(\sigma_1,\sigma_2)\times\mathcal{P}(\sigma_0,\sigma_1)
 \longrightarrow \mathcal{P}(\sigma_0,\sigma_2) 
 \]
 is a morphism of posets and induces a map
 \[
 c_{\sigma_0,\sigma_1,\sigma_2} : P_{\sigma_1,\sigma_2} \times 
 P_{\sigma_0,\sigma_1} \longrightarrow
 P_{\sigma_0,\sigma_2}. 
 \]

 It is straightforward to check that these maps satisfy the requirements
 of a cylindrical structure. 
 Thus the standard cell decomposition of
 the geometric realization $|X|$ is cylindrically normal.
\end{example}

%% file: BC.tex
\section{Topological Face Categories and Their Classifying Spaces}
\label{BC}

Recall that the collection of all cells in a regular cell complex $X$
forms a poset whose order complex is homeomorphic to $X$. We also
have a face poset for any cellular stratified space.
When $X$ is a non-regular cell complex or a cellular stratified space,
however, we cannot expect to recover the homotopy type of $X$ from its
face poset, as we will see in Example
\ref{minimal_cell_decomposition_of_circle2}.

The aim of this section is to show that there is a canonical way to
construct an acyclic topological category\footnote{See Appendix
\ref{topological_category} for basics of topological 
categories.} $C(X)$
from a cylindrically normal cellular stratified space $X$ and that its
classifying space $BC(X)$ has the same homotopy type as $X$ under
appropriate conditions.  

\input{face_category_definition}

\input{barycentric_subdivision}

%% file: face_category_definition.tex
\subsection{Face Categories}
\label{face_category_definition}

There are several ways to construct a category from a cellular or stellar
stratified space. A naive idea is the following.

\begin{definition}
 For a cellular or a stellar stratified space $(X,\pi,\Phi)$ and cells
 $\varphi_{\mu}:D_{\mu}\to \overline{e_{\mu}}$ and
 $\varphi_{\lambda}:D_{\lambda}\to \overline{e_{\lambda}}$ with
 $e_{\mu}\subset\overline{e_{\lambda}}$,   
 define $F(X)(e_{\mu},e_{\lambda})$ to be the set of all maps
 $b:D_{\mu}\to D_{\lambda}$ making the following diagram commutative
 \[
 \begin{diagram}
  \node{D_{\lambda}} \arrow{e,t}{\varphi_{\lambda}}
  \node{\overline{e_{\lambda}}} \\ 
  \node{D_{\mu}} \arrow{n,l}{b} \arrow{e,b}{\varphi_{\mu}}
  \node{\overline{e_{\mu}}.} 
  \arrow{n,J} 
 \end{diagram}
 \]
 The set $F(X)(e_{\mu},e_{\lambda})$ is 
 topologized by the compact-open topology as a subspace of
 $\Map(D_{\mu},D_{\lambda})$.  $F(X)(e_{\mu},e_{\lambda})$ is defined to
 be empty if $e_{\mu}\not\subset\overline{e_{\lambda}}$.

 By defining the set of objects to be cells in $X$ and morphisms from
 $e_{\mu}$ to $e_{\lambda}$ to be $F(X)(e_{\mu},e_{\lambda})$, we obtain 
 a topological category $F(X)$.
 The composition is given by the composition of maps. 
 This topological category $F(X)$ is called the
 \emph{naive face category} of $X$. It is also denoted by $F(X,\pi)$ or 
 $F(X,\pi,\Phi)$.
\end{definition}

\begin{lemma}
 The naive face category $F(X)$ is an acyclic category. When $X$ is
 regular, $F(X)$ is a poset and coincides with the face poset $P(X)$. In  
 particular, when $X$ is a regular cell complex, our
 construction coincides with the classical face poset construction.
\end{lemma}

\begin{proof}
 If both $F(X)(e,e')$ and $F(X)(e',e)$ are
 nonempty, we have $\dim e=\dim e'$, since the existence of a morphism
 $e\to e'$ in $F(X)$ implies $\dim e \le \dim e'$.
 The compatibility of lifts with cell structure maps 
 implies that the only case we have a morphism is $e=e'$ and the
 morphism should be the identity. Thus $F(X)$ is acyclic.

 When $X$ is regular, the regularity implies that there is
 at most one morphism between two objects. Hence it is a poset.
\end{proof}

Even when $X$ is not regular, we have the underlying
poset\footnote{Definition \ref{underlying_poset}}, since $F(X)$ is an
acyclic category. Obviously it is isomorphic to the face poset 
$P(X,\pi)= \Ima\pi$.

\begin{example}
 \label{minimal_cell_decomposition_of_circle2}
 Consider the minimal cell decomposition
 \[
  \pi_n : S^n = e^0\cup e^n \rarrow{} \{0<n\}.
 \]

 This is a typical non-regular cell complex.
 The face poset $P(S^n,\pi_n)$ is a totally ordered set of two elements
 and its order complex is homeomorphic to an interval
 $BP(S^n,\pi_n)\cong [0,1]$. The homotopy type of $S^n$ cannot be
 recovered from its face poset.
 
 On the other hand, the face category $F(S^n,\pi_n)$ has more
 information. It has two objects $e^0$ and $e^n$. 
 When $n=1$, as we have seen
 in Example \ref{minimal_cell_decomposition_of_circle}, the set of
 morphisms from $e^0$ to $e^1$ is given by
 \[
  F(S^1,\pi_1)(e^0,e^1)=\{b_1,b_{-1}\}.
 \]
 We also have the identity morphism on each object. The resulting
 category is depicted in Figure \ref{F(S^1)}.
 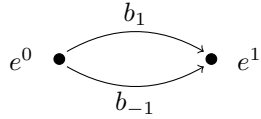
\begin{figure}[ht]
 \begin{center}
  \begin{tikzpicture}
   \draw [fill] (0,0) circle (2pt);
   \draw (-0.5cm,0) node {$e^0$};
   \draw [fill] (2cm,0) circle (2pt);
   \draw (2.5cm,0) node {$e^1$};
   \path [->] (0.1cm,-0.1cm) edge [bend right] (1.9cm,-0.1cm);
   \draw (1cm,-0.6cm) node {$b_{-1}$};
   \path [->] (0.1cm,0.1cm) edge [bend left] (1.9cm,0.1cm);
   \draw (1cm,0.6cm) node {$b_{1}$};
  \end{tikzpicture}
 \end{center}
  \caption{The face category of $S^1=e^0\cup e^1$}
  \label{F(S^1)}
 \end{figure}
 When $n>1$, there are infinitely many morphisms from $e^0$ to $e^n$
 parametrized by $\partial D^n$. And we have a homeomorphism
 \[
  F(S^n,\pi_n)(e^0,e^n) \cong S^{n-1}.
 \]
\end{example}

In the above example of $S^n$, the compact-open topology on the morphism
space $F(S^n,\pi_n)(e^0,e^n)$ can be replaced with a more
understandable topology of $S^{n-1}$. In general, we cannot expect such
simplicity. Under the assumption of cylindrical normality, however,
we may define a smaller face category.

\begin{definition}
 \label{cylindrical_face_category_definition}
 Let $X$ be a cylindrically normal stellar stratified
 space. Define a category $C(X)$ as follows. Objects are
 cells in $X$.
 For each pair $e_{\mu} \subset \overline{e_{\lambda}}$, define
 \[
  C(X)(e_{\mu},e_{\lambda}) = P_{\mu,\lambda}.
 \]
 The composition of morphisms is given by
 \[
  c_{\lambda_0,\lambda_1,\lambda_2} : P_{\lambda_1,\lambda_2}\times
 P_{\lambda_0,\lambda_1} \longrightarrow P_{\lambda_0,\lambda_2}.
 \]
 The category $C(X)$ is called the \emph{cylindrical face
 category} of $X$.
\end{definition}

\begin{lemma}
 \label{face_category_of_CNCSS}
 For any cylindrically normal stellar stratified space $X$,
 its face category $C(X)$ is an acyclic topological category.
 the
 maps $b_{\mu,\lambda}$ induces a continuous functor
 \[
  b : C(X) \longrightarrow F(X),
 \]
 which is natural with respect to morphisms of cylindrically normal
 stellar stratified spaces.

 Furthermore the underlying poset\footnote{Definition
 \ref{underlying_poset}} of $C(X)$ is also 
 $P(X)$ and the diagram
 \[
  \begin{diagram}
   \node{C(X)} \arrow[2]{e,t}{b} \arrow{se} \node{}
   \node{F(X)} \arrow{sw} \\
   \node{} \node{P(X)}
  \end{diagram}
 \]
 is commutative.
\end{lemma}

\begin{proof}
 The continuity of 
 \[
 b_{\mu,\lambda} : P_{\mu,\lambda}\times D_{\mu} \longrightarrow
 D_{\lambda} 
 \]
 implies the continuity of its adjoint
 \[
  \ad(b_{\mu,\lambda}) : P_{\mu,\lambda} \longrightarrow
 \Map(D_{\mu},D_{\lambda}),
 \]
 which factors through $F(X)(e_{\mu},e_{\lambda})$. It is immediate to
 verify that these maps form a continuous functor
 \[
  b : C(X) \longrightarrow F(X).
 \]
 Morphisms of cylindrically normal stellar stratified spaces are
 required to be compatible with maps $b_{\mu,\lambda}$ and thus the
 functor $b$ is natural with respect to morphisms of cylindrically
 nornal stellar stratified spaces.
 The commutativity of the triangle is obvious from the definition.
\end{proof}

\begin{example}
 \label{face_category_of_CP^2}
 Consider the minimal cell decomposition on $\CP^2$
 \[
  \CP^2 = e^0\cup e^2\cup e^4.
 \]
 It is shown in Example \ref{CP^2} that it has a cylindrical
 structure.  
 The cylindrical face category $C(\CP^2)$ has three objects, $e^0$,
 $e^2$, and $e^4$. We have seen that
 \[
  F(\CP^2)(e^0,e^2) = F(S^2,\pi_2)(e^0,e^2) \cong S^1 = C(\CP^2)(e^0,e^2).
 \]
 Since the attaching map of $e^4$ is the Hopf map
 \[
  \eta : S^3 \longrightarrow \CP^1,
 \]
 we have
 \[
  F(\CP^2)(e^0,e^4) \cong \eta^{-1}(e^0) \cong S^1 = C(\CP^2)(e^0,e^4).
 \]
 By using the local trivialization
 \[
  \eta^{-1}(e^2) \cong e^2\times S^1,
 \]
 we see that $F(\CP^2)(e^2,e^4)$ is the set of sections of the trivial
 bundle 
 \[
  D^2\times S^1 \longrightarrow D^2
 \]
 and thus $F(\CP^2)(e^2,e^4) = \Map(D^2,S^1)$. On the other hand,
 we have
 \[
  C(\CP^2)(e^2,e^4) = S^1
 \]
 by definition. The composition
 \[
 C(\CP^2)(e^2,e^4)\times C(\CP^2)(e^0,e^2) \longrightarrow
 C(\CP^2)(e^0,e^4) 
 \]
 is given by the multiplication of $S^1$, as is shown in Example
 \ref{CP^2}. 

 In general, Example \ref{moment_angle_complex} says that the face
 category $C(\CP^n)$ of the minimal cell decomposition of $\CP^n$ can be
 described as a ``poset enriched by $S^1$''
 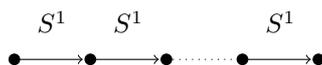
\begin{figure}[ht]
 \begin{center}
  \begin{tikzpicture}
  \draw [fill] (0,0) circle (2pt);
  \draw [->] (0,0) -- (0.9,0);
  \draw (0.5,0.5) node {$S^1$};
  \draw [fill] (1,0) circle (2pt);
  \draw [->](1,0) -- (1.9,0);
  \draw (1.5,0.5) node {$S^1$};
  \draw [fill] (2,0) circle (2pt);
  \draw [dotted] (2,0) -- (3,0);
  \draw [fill] (3,0) circle (2pt);
  \draw [->] (3,0) -- (3.9,0);
  \draw (3.5,0.5) node {$S^1$};
  \draw [fill] (4,0) circle (2pt);
  \end{tikzpicture}
 \end{center}
  \caption{The face category of $\CP^n$}
 \end{figure}
 in the sense that, for any pair of objects $e^{2k}$, $e^{2m}$ ($k<m$),
 the space of morphisms $C(\CP^n)(e^{2k},e^{2m})$ is $S^1$ and the
 composition of morphisms is given by the group structure of $S^1$.
\end{example}

Recall that the order complex of the face poset of a regular cell
complex $X$ is the barycentric subdivision of $X$. With this fact in
mind, we introduce the following notation. 

\begin{definition}
 Let $X$ be a cylindrically normal stellar stratified
 space. Define its 
 \emph{barycentric subdivision} $\Sd(X)$ to be the
 classifying space of the cylindrical face category
 \[
  \Sd(X) = BC(X).
 \]
\end{definition}

\begin{remark}
 There is a notion of barycentric subdivision $\Sd(C)$
 of a small category $C$. A good reference is a paper
 \cite{0707.1718} by del Hoyo. See also Noguchi's papers
 \cite{1004.2547,1104.3630}. We will show that, for a totally
 normal stellar stratified space $X$, there is an isomorphism of
 categories $\Sd(C(X))\cong C(\Sd(X))$ in \S\ref{duality}.
\end{remark}

When $X$ is not a regular cell complex, we usually do not have a 
homeomorphism between $\Sd(X)$ and $X$.

\begin{example}
 Consider $X=\R^n$. This is a regular totally normal cellular
 stratification consisting of a single $n$-cell. The barycentric
 subdivision is a single point.
\end{example}


\begin{example}
 Consider the minimal cell decomposition $\pi_n$ of $S^n$. When
 $n=1$, it is easy to see that $\Sd(S^1,\pi_1)$ is the 
 cell complex in Figure \ref{Sd_of_circle} and is homeomorphic to
 $S^1$. 
 \begin{figure}[ht]
 \begin{center}
  \begin{tikzpicture}
   \draw [fill] (0,0) circle (2pt);
   \draw [fill] (2cm,0) circle (2pt);
   \path (0,0) edge [bend right] (2cm,0);
   \path (0,0) edge [bend left] (2cm,0);
  \end{tikzpicture}
 \end{center}
  \caption{$\Sd(S^1,\pi_1)$}
  \label{Sd_of_circle}
 \end{figure}
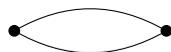
 Note that this complex is obtained by subdividing the $1$-cell in
 $\pi_1$ and can be regarded as the barycentric subdivision
 of $\pi_1$.

 When $n>1$, $C(S^n,\pi_n)$ is a topological category with
 nontrivial topology on $C(S^n,\pi_n)(e^0,e^n)$. Since we have a
 homeomorphism 
 \[
  C(S^n,\pi_n)(e^0,e^n) \cong S^{n-1},
 \]
 it is easy to determine $\Sd(S^n,\pi_n)$ and we have 
 \[
  \Sd(S^n,\pi_n) = BC(S^n,\pi_n) \cong
 \Sigma(S^{n-1})\cong S^n.
 \]
 Again we recovered $S^n$.
\end{example}

\begin{example}
 \label{n_and_0}
 Consider $X=\Int D^n\cup \{(1,0)\}$ with the obvious
 stratification.
 \begin{figure}[ht]
 \begin{center}
  \begin{tikzpicture}
   \draw [dotted] (0,0) circle (1cm);
   \draw [fill] (1,0) circle (1pt);
   \draw (0,0) node {$e^n$};
   \draw (1.3,0) node {$e^0$};
  \end{tikzpicture}
 \end{center}
  \caption{$\Int D^n\cup\{(1,0)\}$}
 \end{figure}
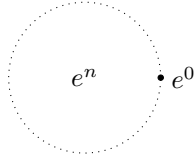

 This is a regular cellular stratification and $\Sd(X)$ is a
 $1$-simplex $[0,1]$. 
 We have an embedding
 \[
  i : [0,1] \longrightarrow X
 \]
 by
 \[
  i(t) = (1-t)(1,0) + t(0,0).
 \] 
 See Figure \ref{image_of_interval}.
 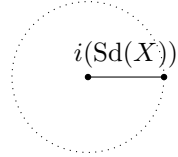
\begin{figure}[ht]
 \begin{center}
  \begin{tikzpicture}
   \draw [dotted] (0,0) circle (1cm);

   \filldraw (1cm,0) circle (1pt);
   \filldraw (0,0) circle (1pt);
   \draw (0,0) -- (1cm,0);
   \draw (0.5,0.3) node {$i(\Sd(X))$};
  \end{tikzpicture}
 \end{center}
  \caption{$i(\Sd(\Int D^2\cup\{(1,0)\}))$}
  \label{image_of_interval}
 \end{figure}
 Obviously $i([0,1])$ is a strong deformation retract of $X$.
\end{example}

\begin{example}
 Consider the punctured torus in Example \ref{punctured_torus}. There is
 a totally normal cellular stratification on
 $X=S^1\times S^1\setminus e^0\times e^0$ induced from the product
 cell decomposition $\pi_1^2$
 \[
  S^1\times S^1 = e^0\times e^0 \cup e^0\times e^1
 \cup e^1\times e^0 \cup e^1\times e^1.
 \]
 Let
 \[
  \varphi_{1,1} : D_{1,1}=[-1,1]^2\setminus\{(-1,-1),(-1,1),(1,-1),(1,1)\}
 \longrightarrow X 
 \]
 be the cell structure map of the $2$-cell in $X$ and
 \begin{eqnarray*}
  \varphi_{0,1} & : & D_{0,1} = (-1,1) \longrightarrow X \\
  \varphi_{1,0} & : & D_{1,0} = (-1,1) \longrightarrow X
 \end{eqnarray*}
 be the cell structure maps for $1$-cells.

 As we have seen in Example \ref{punctured_torus}, there are two ways to
 lift each cell structure map of a $1$-cell and 
 these four lifts cover $\partial D_{1,1}$.
 \[
 \partial D_{1,1} = b_{0,1}'(D_{0,1}) \cup
 b_{0,1}''(D_{0,1}) \cup b_{1,0}'(D_{1,0})\cup
 b_{1,0}''(D_{1,0}). 
 \]
 The cylindrical face category $C(X)$ consists of three
 objects $e^0\times e^1$, $e^1\times e^0$, $e^1\times e^1$. Nontrivial
 morphisms are 
 \begin{eqnarray*}
  C(X)(e^0\times e^1,e^1\times e^1) & = & \{b_{0,1}',
   b_{0,1}''\}, \\
  C(X)(e^1\times e^0,e^1\times e^1) & = & \{b_{1,0}',
   b_{1,0}''\}.
 \end{eqnarray*}

 By allowing multiple edges, we can draw a ``Hasse diagram'' of this
 acyclic category as in Figure \ref{Sd_of_punctured_torus}.

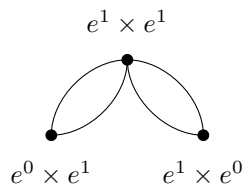
\begin{figure}[ht]
\begin{center}
 \begin{tikzpicture}
  \draw [fill] (0,0) circle (2pt);
  \draw (0,0) .. controls (0,0.5) and (0.5,1) .. (1,1);
  \draw (0,0) .. controls (0.5,0) and (1,0.5) .. (1,1);
  \draw (0,-0.5) node {$e^0\times e^1$};

  \draw [fill] (2,0) circle (2pt);
  \draw (2,0) .. controls (2,0.5) and (1.5,1) .. (1,1);
  \draw (2,0) .. controls (1.5,0) and (1,0.5) .. (1,1);
  \draw (2,-0.5) node {$e^1\times e^0$};

  \draw [fill] (1,1) circle (2pt);
  \draw (1,1.5) node {$e^1\times e^1$};

 \end{tikzpicture}
\end{center}
 \caption{The barycentric subdivision of the punctured torus}
 \label{Sd_of_punctured_torus}
\end{figure}

 Obviously, the classifying space of this category is the wedge of two
 circles,
 \[
  \Sd(X)=BC(X) = S^1\vee S^1.
 \]

 It is easy to find an embedding of $\Sd(X)$ into $X$ by
 using lifts of cell structure maps.
 The images of $0$ under the four maps $b_{0,1}'$,
 $b_{0,1}''$, $b_{1,0}'$, $b_{1,0}''$ constitute four
 points in 
 $\partial D_{1,1}$ that are mapped to a single point under
 $\varphi_{1,1}$. By connecting each of these four points and $(0,0)$ in 
 $D_{1,1}$ by a segment, respectively, we obtain a $1$-dimensional
 stratified space $\widetilde{K}$ in $D_{1,1}$. See Figure
 \ref{thin_subcomplex}. 
 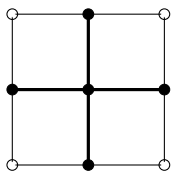
\begin{figure}[ht]
 \begin{center}
  \begin{tikzpicture}
   \draw (0,0) -- (2,0);
   \draw (0,0) -- (0,2);
   \draw (2,0) -- (2,2);
   \draw (0,2) -- (2,2);

   \draw (0,0) circle (2pt);
   \draw [fill,white] (0,0) circle (1pt);
   \draw (2,0) circle (2pt);
   \draw [fill,white] (2,0) circle (1pt);
   \draw (0,2) circle (2pt);
   \draw [fill,white] (0,2) circle (1pt);
   \draw (2,2) circle (2pt);
   \draw [fill,white] (2,2) circle (1pt);

   \draw [fill] (1,0) circle (2pt);
   \draw [fill] (2,1) circle (2pt);
   \draw [fill] (1,2) circle (2pt);
   \draw [fill] (0,1) circle (2pt);
   \draw [fill] (1,1) circle (2pt);

   \draw [very thick] (1,1) -- (1,0);
   \draw [very thick] (1,1) -- (1,2);
   \draw [very thick] (1,1) -- (0,1);
   \draw [very thick] (1,1) -- (2,1);
  \end{tikzpicture}
 \end{center}
  \caption{A thin subcomplex $\widetilde{K}$ in $D_{1,1}$}
  \label{thin_subcomplex}
 \end{figure}

 The complex $\widetilde{K}$ in Figure \ref{thin_subcomplex} corresponds
 to cells in $\Sd(X)$ and we have an embedding
 \[
  \Sd(X) \hookrightarrow X.
 \]
 Figure \ref{thin_subcomplex} can be also used to construct a deformation
 retraction of $X$ onto $\Sd(X)$.
\end{example}

The above examples show that, when a cellular stratification
$\pi$ is totally normal, or more generally, cylindrically
normal, the barycentric subdivision $\Sd(X,\pi)$ is closely
related to $X$. 

The work of Cohen, Jones, and Segal \cite{Cohen-Jones-SegalMorse}
suggests that one should analyze the nerve of the face category of a
cylindrically normal stellar stratified space by using the underlying
poset functor 
\[
 \pi : C(X,\pi) \longrightarrow P(X,\pi).
\]

The following easily verifiable fact will be used in the next section.

\begin{lemma}
 \label{underlying_poset_nerve}
 For a cylindrically normal stellar stratified space $(X,\pi)$,
 consider the induced morphism of simplicial sets\footnote{Here we
 forget the topology on $C(X,\pi)$ temporarily.}
 \[
  N(\pi) : N(C(X,\pi)) \longrightarrow N(P(X,\pi)).
 \]
 For each $n$-chain in the face poset
 $\bm{e}=(e_{\lambda_0},\ldots,e_{\lambda_n})\in N_n(P(X,\pi))$,
 we have 
 \[
 N(\pi)_n^{-1}(\bm{e}) = P_{\lambda_{n-1},\lambda_{n}}\times\cdots\times
 P_{\lambda_{0},\lambda_{1}}.
 \]
 Consequently the space of $n$-chains has the following decomposition
 \[
  N_n(C(X,\pi)) = \coprod_{\bm{e}\in N_n(P(X,\pi))}
 \{\bm{e}\}\times P_{\lambda_{n-1},\lambda_n}\times\cdots\times
 P_{\lambda_{0},\lambda_{1}}. 
 \]
 The space of nondegenerate $n$-chains is, therefore, given by
 \[
  \overline{N}_n(C(X,\pi)) = \coprod_{\bm{e}\in \overline{N}_n(P(X,\pi))}
 \{\bm{e}\}\times P_{\lambda_{n-1},\lambda_n}\times\cdots\times
 P_{\lambda_{0},\lambda_{1}}. 
 \]
\end{lemma}

By using the face category, we may regard a cylindrically normal
stellar stratified space as a functor.

\begin{definition}
 For a cylindrically normal stellar stratified space $X$, define a
 functor
 \[
  D^{X} : C(X) \rarrow{} \Spaces
 \]
 by assigning the domain $D_{\lambda}$ to each cell
 $\varphi_{\lambda}:D_{\lambda}\to \overline{e_{\lambda}}$. For
 a morphism $p\in C(X)(e_{\mu},e_{\lambda})=P_{\mu,\lambda}$, define
 $D^{X}(p)=b_{\mu,\lambda}(p)\in \Map(D_{\mu},D_{\lambda})$.  
\end{definition}

The following is an extension of Proposition 2.47 of \cite{1312.7368}.

\begin{proposition}
 For a CW cylindrically normal stellar stratified space $X$, the functor
 $D^{X}$ is a continuous functor\footnote{See Definition
 \ref{continuous_functor_to_Top} for the 
 definitions of the continuity of a functor to $\category{Spaces}$ and
 its colimit.} and we have a natural homeomorphism
 \[
  \colim_{C(X)} D^{X} \rarrow{\cong} X.
 \]
\end{proposition}

\begin{proof}
 The continuity of the functor $D^{X}$ is obvious from the
 definition.

 By definition, $\colim_{C(X)} D^{X}$ is a quotient space of
 $\tot(D^{X})=D(X)$. Let $\sim_{c}$ be the defining equivalence relation
 so that $\colim_{C(X)}D^{X} = D(X)/_{\sim_{c}}$. On the other hand, by
 Lemma \ref{quotient_of_D(X)}, we have a description of $X$ as a
 quotient of $D(X)$
 \[
  X \cong D(X)/_{\sim_{\Phi}},
 \]
 where the relation $\sim_{\Phi}$ is defined by $x\sim_{\Phi} y$ if and
 only if $\varphi_{\mu}(x) =\varphi_{\lambda}(y)$ for $x\in D_{\mu}$ and
 $y\in D_{\lambda}$.

 It remains to verify that $\sim_{c}$ is identical to $\sim_{\Phi}$. The
 proof is essentially the same as that of Proposition 2.47 in
 \cite{1312.7368}. Details are omitted.
\end{proof}

%% file: barycentric_subdivision.tex
\subsection{Barycentric Subdivisions of Cellular
 Stratified Spaces}
\label{barycentric_subdivision}

As we can see from the examples in \S\ref{face_category_definition}, we
often have an embedding of the barycentric subdivision $\Sd(X)$ in $X$,
even if $X$ is neither a cell complex nor regular. In this section, we first
show that such an embedding always exists for any cylindrically normal
cellular stratified space. And then we show that it often admits a
strong deformation retraction.

In order to describe our embedding, we use
the language of $\Delta$-spaces\footnote{A $\Delta$-space is a
simplicial space without degeneracies. See Definition
\ref{Delta-set_definition}.}. 

\begin{theorem}
 \label{embedding_Sd}
 Let $(X,\pi)$ be a cylindrically normal CW stellar stratified space.
 There exists an embedding
 \[
  i_{X} : \Sd(X,\pi) \hookrightarrow X,
 \]
 which is natural with respect to
 strict morphisms of cylindrically normal stellar stratified spaces.
 Furthermore, when all cells in $X$ are closed, $i_{X}$ is a
 homeomorphism onto $X$. 
\end{theorem}

The construction of $i_{X}$ and the proof of the fact that $i_{X}$ is an
embedding are essentially the same as those of Proposition 2.51 in
\cite{1312.7368}. We record them for the convenience of the reader.

\begin{definition}
 We construct $i_{X}$ as follows.
 Since the face category $C(X,\pi)=C(X)$ is acyclic,
 $\Sd(X)=BC(X)=\|\overline{N}(C(X))\|$. 
 Thus it suffices to construct a series of maps
 \[
  i_n : \overline{N}_n(C(X))\times\Delta^n \longrightarrow X
 \]
 making the following diagram commutative for all $i$
 \[
  \begin{diagram}
   \node{\overline{N}_n(C(X,\pi))\times\Delta^{n-1}} \arrow{s,l}{d_i\times 1}
   \arrow{e,t}{1\times d^i} 
   \node{\overline{N}_n(C(X,\pi))\times\Delta^n} \arrow{s,r}{i_n} \\ 
   \node{\overline{N}_{n-1}(C(X,\pi))\times\Delta^{n-1}}
   \arrow{e,b}{i_{n-1}} \node{X.}
  \end{diagram}
 \]

 Let us construct $i_n$ by induction on $n$.
 The set $\overline{N}_0(C(X))$ is the set of cells for $\pi$.
 For each cell $e_{\lambda}$ with cell structure map
 $\varphi_{\lambda} : D_{\lambda} \to \overline{e_{\lambda}} \subset X$,
 define
 \[
  i_0(e_{\lambda}) = \varphi_{\lambda}(0)
 \]
 and we obtain a map
 \[
  i_0 : \overline{N}_0(C(X))\cong \overline{N}_0(C(X))\times\Delta^0
 \hookrightarrow X. 
 \]
 Note that this map makes the diagram
 \[
 \begin{diagram}
  \node{\overline{N}_0(C(X,\pi))\times\Delta^0} \arrow{e,t}{i_0}
  \arrow{s,r}{z_0} \node{X} \\ 
  \node{D(X)}
  \arrow{ne,b}{\Phi} 
 \end{diagram}
 \]
 commutative, where $z_0$ is induced by the inclusion of $\Delta^0$ to
 the center of each disk and $D(X)$ and $\Phi$ are defined in 
 Lemma \ref{quotient_of_D(X)}.

 Suppose that we have constructed a map
 \[
  i_{k-1} : \overline{N}_{k-1}(C(X)) \times\Delta^{k-1} \rarrow{}
 X   
 \]
 satisfying the above compatibility conditions and that the restriction
 of $i_{k-1}$ to $\overline{N}_{k-1}(C(X)) \times\Int(\Delta^{k-1})$ is an
 embedding. Suppose, further, that there exists a map
 \[
  z_{k-1} : \overline{N}_{k-1}(C(X))\times\Delta^{k-1} \rarrow{}
 D(X) 
 \]
 making the diagram
 \[
 \begin{diagram}
  \node{\overline{N}_{k-1}(C(X))\times\Delta^{k-1}}
  \arrow{e,t}{i_{k-1}} 
  \arrow{s,r}{z_{k-1}} \node{X} \\ 
  \node{D(X)}
  \arrow{ne,b}{\Phi} 
 \end{diagram}  
 \]
 commutative. We construct an embedding
 \[
  z_k : \overline{N}_k(C(X))\times \Delta^k \longrightarrow D(X)  
 \]
 satisfying the compatibility conditions corresponding to those of $i_k$
 and define $i_k$ to be $\Phi\circ z_k$.

 Under the decomposition in Lemma \ref{underlying_poset_nerve}
 \[
  \overline{N}_k(C(X)) = \coprod_{\bm{e}\in \overline{N}_k(P(X))}
 \{\bm{e}\}\times \overline{N}(\pi)_k^{-1}(\bm{e}), 
 \]
 it suffices to construct a map
 \[
 z_{\bm{e}} : \overline{N}(\pi)_k^{-1}(\bm{e})\times \Delta^k \rarrow{}
 D_{\lambda_k}
 \]
 for each nodegenerate $k$-chain
 $\bm{e}:e_{\lambda_0}<\cdots < e_{\lambda_k}$ in
 $P(X)$.
 The maps $\{z_{\bm{e}}\}$ should satisfy the following conditions:  
 \begin{enumerate}
  \item For each $0\le j< k$, the following diagram is commutative:
	\begin{equation}
	 \begin{diagram}
	  \node{\overline{N}(\pi)_k^{-1}(\bm{e})\times \Delta^k}
	  \arrow[3]{e,t}{z_{\bm{e}}} \node{} \node{}
	  \node{D_{\lambda_k}} \\
	  \node{\overline{N}(\pi)_{k}^{-1}(\bm{e})\times \Delta^{k-1}}
	  \arrow{n,l}{1\times d^j} \arrow[3]{e,b}{d_j\times 1} \node{}
	  \node{} 
	  \node{\overline{N}(\pi)_{k-1}^{-1}(d_j(\bm{e}))\times \Delta^{k-1}.} 
	  \arrow{n,b}{z_{d_j(\bm{e})}} 
	 \end{diagram}
	 \label{dj}
	\end{equation}

  \item When $j=k$, the following diagram is commutative:
	\begin{equation}
	 \begin{diagram}
	  \node{\overline{N}(\pi)_k^{-1}(\bm{e})\times \Delta^k}
	  \arrow{e,t}{z_{\bm{e}}} 
	  \node{D_{\lambda_k}} \\
	  \node{\overline{N}(\pi)_k^{-1}(\bm{e})\times \Delta^{k-1}}
	  \arrow{n,l}{1\times d^k} \arrow{s,=} \node{} \\
	  \node{P_{\lambda_{k-1},\lambda_{k}}\times
	  \overline{N}(\pi)_{k-1}^{-1}(d_k(\bm{e}))\times \Delta^{k-1}}  
	  \arrow{e,b}{1\times z_{d_k(\bm{e})}} 
	  \node{P_{\lambda_{k-1},\lambda_{k}}\times D_{\lambda_{k-1}}.}
	  \arrow[2]{n,r}{b_{\lambda_{k-1},\lambda_k}} 
	 \end{diagram}
	 \label{dk}
	\end{equation}
 \end{enumerate}

 By the inductive assumption, we have an embedding
 \[
  z_{d_k(\bm{e})} : \overline{N}(\pi)_{k-1}^{-1}(d_k(\bm{e}))\times
 \Delta^{k-1} \rarrow{} D_{\lambda_{k-1}}
 \]
 corresponding to the $(k-1)$-chain
 $d_k(\bm{e}) : e_{\lambda_0}<\cdots< e_{\lambda_{k-1}}$. Note that
 \[
 \overline{N}(\pi)_k^{-1}(\bm{e}) = P_{\lambda_{k-1},\lambda_k} \times
 \overline{N}(\pi)_{k-1}^{-1}(d_k(\bm{e})).  
 \]
 Compose with $b_{\lambda_{k-1},\lambda_k}$ and we obtain a map
 \begin{eqnarray*}
  \overline{N}(\pi)_k^{-1}(\bm{e})\times \Delta^{k-1} & = &
   P_{\lambda_{k-1},\lambda_k} \times
   \overline{N}(\pi)_{k-1}^{-1}(d_k(\bm{e}))\times \Delta^{k-1} \\
  & \rarrow{1\times z_{d_k(\bm{e})}} & P_{\lambda_{k-1},\lambda_k} \times
   D_{\lambda_{k-1}} \\
  & \rarrow{b_{\lambda_{k-1},\lambda_k}} & \partial D_{\lambda_k}.
 \end{eqnarray*}
 Since $D_{\lambda_k}$ is an aster\footnote{Definition
 \ref{aster_definition}}, it can be extended to an embedding
 \[
  z_{\bm{e}} : \overline{N}(\pi)_k^{-1}(\bm{e})\times \Delta^k =
 \overline{N}(\pi)_k^{-1}(\bm{e})\times  
 \Delta^{k-1}\ast \bm{v}_k \rarrow{} \partial D_{\lambda_k} \ast 0
 \subset D_{\lambda_k}
 \]
 by 
\[
 z_{\bm{e}}(\bm{p},(1-t)\bm{s}+t \bm{v}_k) =
 (1-t)z_{d_k(\bm{e})}(\bm{p},\bm{s})+t\cdot 0 =
 (1-t)z_{d_k(\bm{e})}(\bm{p},\bm{s}), 
\]
 where $\bm{v}_k = (0,\ldots,0,1)$ is the last vertex in
 $\Delta^k$. Recall that the join operation $\ast$ used in this paper is
 defined by connecting points by line segments\footnote{Definition
 \ref{join_definition}.} and
 is not the same as the join operation used in algebraic topology, for
 example in Milnor's paper \cite{MilnorBG}.

 By definition, $z_{\bm{e}}$ makes the diagram (\ref{dk}) commutative.
 Let us verify that $z_{\bm{e}}$ also makes the diagram (\ref{dj})
 commutative for $0\le j< k$. By the inductive assumption, the
 following diagram is commutative:
 \[
  \begin{diagram}
   \node{\overline{N}(\pi)_{k-1}^{-1}(d_k(\bm{e}))\times \Delta^{k-2}}
   \arrow{s,l}{1\times d^j} \arrow{e,t}{d_j\times 1}
   \node{\overline{N}(\pi)_{k-2}^{-1}(d_jd_k(\bm{e}))\times \Delta^{k-2}} 
   \arrow{s,r}{z_{d_jd_k(\bm{e})}} \arrow{e,=}
   \node{\overline{N}(\pi)_{k-2}^{-1}(d_{k-1}d_j(\bm{e}))\times \Delta^{k-2}} 
   \arrow{sw,b}{z_{d_{k-1}d_j(\bm{e})}} \\ 
   \node{\overline{N}(\pi)_{k-1}^{-1}(d_k(\bm{e}))\times \Delta^{k-1}}
   \arrow{e,b}{z_{d_{k}(\bm{e})}} \node{D_{\lambda_{k-1}}.}
  \end{diagram}
 \]
 Under the identification $\Delta^{k}=\Delta^{k-1}\ast\bm{v}_k$,
 $d^j : \Delta^{k-1}\to\Delta^k$ can be identified with the composition
 \[
  \Delta^{k-2}\ast\bm{v}_{k-1} \cong \Delta^{k-2}\ast\bm{v}_{k}
 \rarrow{d^j\ast 1_{\bm{v}_k}} \Delta^{k-1}\ast\bm{v}_k.
 \]
 On the other hand, since $j<k$, we have
 \begin{eqnarray*}
  \overline{N}(\pi)_k(\bm{e}) & = & P_{\lambda_{k-1},\lambda_k}\times
 \overline{N}(\pi)_{k-1}^{-1}(d_{k}(\bm{e})) \\  
  \overline{N}(\pi)_{k-1}(d_j(\bm{e})) & = & P_{\lambda_{k-1},\lambda_k}\times
 \overline{N}(\pi)_{k-2}^{-1}(d_{k-1}d_{j}(\bm{e})),
 \end{eqnarray*}
 and the face operator
 $d_j : \overline{N}(\pi)_k(\bm{e}) \to \overline{N}(\pi)_{k-1}(d_j(\bm{e}))$
 coincides with the map
\[
 P_{\lambda_{k-1},\lambda_{k}}\times
 \overline{N}(\pi)_{k-1}^{-1}(d_k(\bm{e})) 
 \rarrow{1\times d_j} P_{\lambda_{k-1},\lambda_k}\times
 \overline{N}(\pi)_{k-1}^{-1}(d_{j}d_k(\bm{e})) =
 P_{\lambda_{k-1},\lambda_k}\times 
 \overline{N}(\pi)_{k-1}^{-1}(d_{k-1}d_{j}(\bm{e})). 
\]
Thus we obtain the commutative diagram
\[
 \begin{diagram}
  \node{\overline{N}(\pi)_k^{-1}(\bm{e})\times\Delta^{k-1}}
  \arrow{e,t}{d_j\times 1} \arrow{s,=} 
  \node{\overline{N}(\pi)_{k-1}^{-1}(d_j(\bm{e}))\times\Delta^{k-1}}
  \arrow{s,=} \\ 
  \node{P_{\lambda_{k-1},\lambda_k}\times
  \overline{N}(\pi)_{k-1}^{-1}(d_k(\bm{e}))\times\Delta^{k-1}}
  \arrow{e,t}{1\times d_j\times 1} \arrow{s,l}{1\times 1\times d^j}
  \node{P_{\lambda_{k-1},\lambda_k}\times 
  \overline{N}(\pi)_{k-2}^{-1}(d_{k-1}d_j(\bm{e}))\times \Delta^{k-1}} 
  \arrow{s,r}{1\times \tilde{z}_{d_{k-1}d_j(\bm{e})}} 
  \\ 
 \node{P_{\lambda_{k-1},\lambda_k}\times
  \overline{N}(\pi)_{k-1}^{-1}(d_k(\bm{e})) 
  \times \Delta^{k}} \arrow{e,t}{1\times
  \tilde{z}_{d_{k}(\bm{e})}} 
  \node{P_{\lambda_{k-1},\lambda_k}\times  
  ((D_{\lambda_{k-1}}\times\{0\})\ast (0,1))}
 \end{diagram}
\] 
 where we embed $D_{\lambda_{k-1}}$ into $D_{\lambda_{k-1}}\times\R$ as
 $D_{\lambda_{k-1}}\times\{0\}$ and $\tilde{z}_{d_k(\bm{e})}$ is defined by
 \[
  \tilde{z}_{d_k(\bm{e})}(\bm{p}, (1-t)\bm{s}+t\bm{v}_k) =
 (1-t)(z_{d_k(\bm{e})}(\bm{p},\bm{s}),0) +t(0,1).
 \]
 The map $\tilde{z}_{d_{k-1}d_j(\bm{e})}$ is defined analogously. We
 also have an extension of $b_{\lambda_{k-1},\lambda_{k}}$
 \[
 \tilde{b}_{\lambda_{k-1},\lambda_{k}} :
 P_{\lambda_{k-1},\lambda_{k}}\times
 ((D_{\lambda_{k-1}}\times\{0\})\ast (0,1)) \longrightarrow D_{\lambda_{k}}. 
 \]
 By definition, we have the commutative diagram
 \[
  \begin{diagram}
   \node{P_{\lambda_{k-1},\lambda_{k}}\times
   N(\pi)_{k-2}^{-1}(d_{k-1}d_j(\bm{e}))\times \Delta^{k-1}}
   \arrow{s,l}{1\times \tilde{z}_{d_{k-1}d_j(\bm{e})}} \arrow{e,=}
   \node{N(\pi)_{k-1}^{-1}(d_{j}(\bm{e}))\times \Delta^{k-1}}
   \arrow{s,r}{z_{d_j(\bm{e})}} \\    
   \node{P_{\lambda_{k-1},\lambda_k}\times
   ((D_{\lambda_{k-1}}\times\{0\})\ast (0,1))}
   \arrow{e,t}{\tilde{b}_{\lambda_{k-1},\lambda_k}} 
    \node{D_{\lambda_k}} \\
   \node{P_{\lambda_{k-1},\lambda_k}\times
   N(\pi)_{k-1}^{-1}(d_k(\bm{e}))\times\Delta^k} \arrow{n,l}{1\times
   \tilde{z}_{d_k(\bm{e})}}  \arrow{e,=}
   \node{N(\pi)_k^{-1}(\bm{e})\times\Delta^k.} \arrow{n,r}{z_{\bm{e}}}
  \end{diagram}
 \]
 This completes the proof of the commutativity of the diagram
 (\ref{dj}).
 And we obtain a map
 \[
  z_{k} : \overline{N}_{k}(C(X))\times\Delta^k \rarrow{} D(X)
 \]
 for all $k$. By composing with $\Phi:D(X)\to X$,
 we obtain a familty of continuous maps
 \[
  \{i_k : \overline{N}_k(C(X))\times\Delta^k \rarrow{} X\}_{k\ge 0}
 \]
 that are compatible with $d_i$, which induces a continuous map
 \[
  i_{X} : \Sd(X) = \left\|\overline{N}(C(X))\right\| \rarrow{} X.
 \]
\end{definition}

The next step is to prove that $i_{X}$ is an embedding.

\begin{proof}[Proof of Theorem \ref{embedding_Sd}]
 Let us show that $i_{X}: BC(X)\to \Ima(\iota_{X})$ is a bijective
 closed map, hence is a homeomorphism.
 The bijectivity is obvious from the construction. In order to show that
 $i_{X}$ is a closed map, consider the diagram
 \[
  \begin{diagram}
   \node{\coprod_{n,\bm{e}\in \overline{N}_{n}(P(X))} \{\bm{e}\}\times
   \overline{N}(\pi)_n^{-1}(\bm{e})\times\Delta^n} \arrow{s,l}{p} 
   \arrow{e,t}{\coprod_{\bm{e}} z_{\bm{e}}} 
   \node{D(X)} \arrow{s,r}{\Phi} \\
   \node{\left\|\overline{N}(C(X))\right\|} \arrow{e,b}{i_{X}} \node{X}
  \end{diagram}
 \]
 whose vertical arrows are quotient maps.
 For simplicity, let us denote $z = \coprod_{\bm{e}} z_{\bm{e}}$.
 For $A\subset \Sd(X)$, we show that
 \begin{equation}
  \Phi^{-1}(i_{X}(A)) = z(p^{-1}(A)),  
   \label{inverse_image_of_Phi}
 \end{equation}
 which immediately implies that $i_{X}$ is a closed map onto its image,
 since $\Phi$ is a quotient map and $z$ is an embedding. 
 
 The commutativity of the above diagram implies
 \[
  \Phi^{-1}(i_{X}(A)) \supset z(p^{-1}(A)).
 \]
 On the other hand, suppose $x\in \Phi^{-1}(i_{X}(A))$.
 Then there exist $a\in A$ and $\lambda\in\Lambda$ with
 $\Phi(x)=i_{X}(a)$ and $x\in D_{\lambda}$.
 Write $a=p(\bm{b},\bm{t}))$ for
 $\bm{b}\in \overline{N}(\pi)_{k}^{-1}(\bm{e})$ and
 $\bm{t}\in \Delta^{k}$.
 
 If $x\in\Int(D_{\lambda})$, we have
 \[
 \varphi_{\lambda}(x) = i_{X}(a) = i_{X}(p(\bm{b},\bm{t})) =
 \Phi(z_{\bm{e}}(\bm{b},\bm{t})). 
 \]
 Since $x\in \Int(D_{\lambda})$, $\bm{t}$ is of the form
 $\bm{t}=(1-t)\bm{s}+t\bm{v}_k$ for some $0<t<1$ and
 $\bm{s}\in\Delta^{k-1}$. This implies that
 $z_{\bm{e}}(\bm{b},\bm{t})\in\Int(D_{\lambda})$ and
 $\varphi_{\lambda}(x)=\varphi_{\lambda}(z_{\bm{e}}(\bm{b},\bm{t}))$.
 Thus $x=z_{\bm{e}}(\bm{b},\bm{t})$ with $p(\bm{b},\bm{t})=a\in A$.
 In other words, $x\in z(p^{-1}(A))$.
 
 Suppose $x\in \partial D_{\lambda}$ and let $e_{\mu}$ be a cell
 containing $\varphi_{\lambda}(x)$.
 By the cylindrical normality,
 $\partial D_{\lambda}$ is a stratified space and
 \[
 b_{\mu,\lambda} : P_{\mu,\lambda}\times D_{\mu} \rarrow{} \partial
 D_{\lambda} 
 \]
 is a strict morphism of stratified space. Choose
 $(b,y)\in P_{\mu,\lambda}\times D_{\mu}$
 with $b_{\mu,\lambda}(b,y)=x$. Then
 $\varphi_{\mu}(y)=\varphi_{\lambda}(x)=\Phi(x)\in i_{X}(A)$.
 Since $y\in \Int(D_{\mu})$, the above argument implies that there
 exists $(\bm{b}',\bm{t}')\in p^{-1}(A)$ with
 $y=z_{\pi(\bm{b}')}(\bm{b}',\bm{t}')$. Define $\bm{b}=(b,\bm{b}')$ and
 $\bm{t}=d^k(\bm{t}')$. Then the defining relation of the geometric
 realization of a $\Delta$-space implies that
 \[
  p(\bm{b},\bm{t}) = [\bm{b},d^k(\bm{t}')] 
   = [d_{k}(\bm{b}),\bm{t}'] \\
  =  [\bm{b}',\bm{t}'] = p(\bm{b}',\bm{t}')\in A.
 \]
 On the other hand, the commutativity of (\ref{dk}) implies that
 \[
  x=b_{\mu,\lambda}(b,y)=
 b_{\mu,\lambda}(b,z_{p(\bm{b}')}(\bm{b}',\bm{t}')) =
 z_{p(b,\bm{b}')}(p(b,\bm{b}),d^k(\bm{t}')) =
 z_{p(\bm{b})}(\bm{b},\bm{t}).  
 \]
 And we have $x\in z(p^{-1}(A))$. This complets the proof of
 (\ref{inverse_image_of_Phi}). 
 
 Finally when all cells are closed,
 each $z_k$ is surjective and thus $i_{X}$ is a homeomorphism.
\end{proof}

The next task is to show that the image of $i_{X}$ is a strong
deformation retract of $X$ under a suitable condition.
The idea of proof is essentially the same as Theorem
\ref{embedding_Sd}. We construct deformation retractions on each cell
and glue them together. In order to define cell-wise deformation
retraction, we assume the polyhedrality.

\begin{theorem}
 \label{deformation_theorem}
 For a polyhedral stellar stratified space $X$, the image of
 embedding $i_{X} : \Sd(X,\pi) \to X$ is a strong deformation
 retract of $X$. The deformation retraction can be taken to be natural
 with respect to strict morphisms of polyhedral stellar stratified
 spaces. 
\end{theorem}

\begin{corollary}
 \label{subdivision_of_totally_normal_cellular_stratified_space}
 For a totally normal cellular stratified space $(X,\pi)$, the map
 $i_{X}$ embeds $\Sd(X,\pi)$ into $X$ as a strong deformation
 retract. 
\end{corollary}

The following fact is essential.  

\begin{lemma}
 \label{spherical_case}
 Let $\pi$ be a regular cell decomposition of $S^{n-1}$ and
 $L \subset S^{n-1}$ be a stratified subspace. Let $\tilde{\pi}$ be the
 cellular stratification on $K=\Int D^n\cup L$ obtained by adding
 $\Int D^n$ as an $n$-cell.  
 Then there is a deformation retraction $H$ of $K$ to
 $i_{K}(\Sd(K,\tilde{\pi}))$. Furthermore if a deformation
 retraction $h$ of $L$ 
 onto $i_{L}(\Sd(L))$ is given, $H$ can be taken to be an extension of
 $h$. 
\end{lemma}

This Lemma can be proved by using good old simplicial topology.
It was first proved in a joint work with Basabe, Gonz{\'a}lez, and
Rudyak \cite{1009.1851v5} but removed from the published version
\cite{1009.1851}. 
The proof is now contained in \cite{1312.7368} as Appendix A.
 
Now we are ready to prove Theorem
\ref{deformation_theorem}. 

\begin{proof}[Proof of Theorem
 \ref{deformation_theorem}] 
 Let us show the embedding $i_{X}$ constructed in the proof of
 Theorem \ref{embedding_Sd} has a homotopy inverse. 

 Since $X$ is polyhedral, each cell 
 $\varphi_{\lambda} : D_{\lambda}\to \overline{e_{\lambda}}$ has a
 polyhedral replacement 
 $\alpha_{\lambda} : \widetilde{F}_{\lambda} \to D^{\dim e_{\lambda}}$.
 For simplicity, we identify $D_{\lambda}$ with
 $\alpha_{\lambda}^{-1}(D_{\lambda})$. Thus $D_{\lambda}$ is a
 stratified subspace of a polyhedral complex and
 $b_{\mu,\lambda} : P_{\mu,\lambda}\times D_{\mu}\to D_{\lambda}$ is a
 PL map.

 We construct, by induction on $k$, a PL homotopy
 \[
  H_{\lambda} : D_{\lambda}\times [0,1] \longrightarrow D_{\lambda}
 \]
 for each $k$-cell $e_{\lambda}$ satisfying the following conditions: 
 \begin{enumerate}
  \item It is a strong deformation retraction of $D_{\lambda}$ onto
	$i_{D_{\lambda}}(\Sd(D_{\lambda}))$, 
	where the stratification on $D_{\lambda}$ is given by adding
	$\Int(D_{\lambda})$ as a unique $k$ cell to the 
	stratification on $\partial D_{\lambda}$.
  \item The diagram
	\begin{equation}
	 \begin{diagram}
	  \node{D_{\lambda}\times [0,1]} \arrow{e,t}{H_{\lambda}}
	  \node{D_{\lambda}} \\ 
	  \node{P_{\mu,\lambda}\times D_{\mu}\times [0,1]}
	  \arrow{n,l}{b_{\mu,\lambda}\times 1} 
	  \arrow{e,b}{1\times H_{\mu}}
	  \node{P_{\mu,\lambda}\times D_{\mu}.} \arrow{n,r}{b_{\mu,\lambda}}
	 \end{diagram}
	 \label{compatibility_for_H}
	\end{equation}
	 is commutative for any pair $e_{\mu}\subset\overline{e_{\lambda}}$.
 \end{enumerate}

 When $k=0$, the homotopy is the canonical projection.
 Suppose we have constructed $H_{\mu}$ for all $i$-cells $e_{\mu}$ with
 $i\le k-1$. We would like to extend them to $k$-cells.
 For a $k$-cell $e_{\lambda}$ with stellar structure map
 $\varphi_{\lambda} : D_{\lambda} \to X$ and a cell $e_{\mu}$ with
 $e_{\mu}\subset \overline{e_{\lambda}}$, consider the diagram
 \[
  \begin{diagram}
   \node{P_{\mu,\lambda}\times D_{\mu}\times [0,1]}
   \arrow{s,l}{b_{\mu,\lambda}\times 1}
   \arrow{e,t}{1\times 
   H_{\mu}} \node{P_{\mu,\lambda}\times D_{\mu}}
   \arrow{s,r}{b_{\mu,\lambda}} \\
   \node{b_{\mu,\lambda}(P_{\mu,\lambda}\times D_{\mu})\times [0,1]}
   \arrow{e,..} 
   \node{b_{\mu,\lambda}(P_{\mu,\lambda}\times D_{\mu}).} 
  \end{diagram}
 \]
 Since $b_{\mu,\lambda}$ is a homeomorphism when restricted to
 $P_{\mu,\lambda}\times \Int(D_{\mu})$ and $b_{\mu,\lambda}$ and
 $H_{\mu}$ are assumed to be PL, we have the dotted arrow making the
 diagram commutative by Lemma \ref{PL_map_is_extendalbe}. 

 The decomposition
 \[
  \partial D_{\lambda} = \bigcup_{e_{\mu}\subset
 \overline{e_{\lambda}}} b_{\mu,\lambda}(P_{\mu,\lambda}\times
 D_{\mu})
 \]
 allows us to glue these homotopies together. Thus we obtain a homotopy
 \[
 H_{\lambda} : \partial D_{\lambda}\times [0,1]
  \rarrow{} \partial D_{\lambda}. 
 \]
 By Lemma \ref{spherical_case}, this homotopy can be extended to a
 strong deformation retraction
 \[
  H_{\lambda} : D_{\lambda}\times [0,1] \longrightarrow D_{\lambda}
 \]
 of $D_{\lambda}$ onto $i_{D_{\lambda}}(\Sd(D_{\lambda}))$.
 This completes the inductive step.

 The second condition above (\ref{compatibility_for_H}) and the CW
 condition imply that 
 these deformation retractions can be assembled together to give a
 strong deformation retraction 
 \[
 H :  X \times[0,1] \longrightarrow X
 \]
 of $X$ onto $i_{X}(\Sd(X))$.
\end{proof}

The construction of the higher
order Salvetti complex in \cite{Bjorner-Ziegler92,DeConcini-Salvetti00}
can be regarded as a corollary to Theorem 
\ref{deformation_theorem}. 

\begin{example}
 \label{BZDS_stratification}
 Consider the stratification on $\R^n\otimes\R^{\ell}$
 \[
  \pi_{\mathcal{A}\otimes\R^{\ell}} : \R^{n}\otimes\R^{\ell}
 \rarrow{} \Map(L,S_{\ell}) 
 \]
 in Example \ref{stratification_by_arrangement}.
 As we have seen in Example
 \ref{cellular_stratification_by_arrangement}, this is a normal 
 cellular stratification. It is also regular and polyhedral.

 It contains
 \[
 \Lk(\mathcal{A}\otimes\R^{\ell})= \bigcup_{i=1}^{k}
 H_i\otimes\R^{\ell} 
 \]
 as a cellular stratified subspace. The complement
 \[
 M(\mathcal{A}\otimes\R^{\ell}) =
 \R^n\otimes\R^{\ell}\setminus\Lk(\mathcal{A}\otimes\R^{\ell})  
 \]
 is also a cellular stratified subspace. 
 Since $\pi_{\mathcal{A}\otimes\R^{\ell}}$ is regular, we have
 \[
 C(M(\mathcal{A}\otimes\R^{\ell})) = P(M(\mathcal{A}\otimes\R^{\ell})) = 
 P(\mathcal{A}\otimes\R^{\ell}) \setminus
 P(\Lk(\mathcal{A}\otimes\R^{\ell})). 
 \]

 By Theorem
 \ref{deformation_theorem},
 $BC(M(\mathcal{A}\otimes\R^{\ell}))$ can be embedded in the complement
 $M(\mathcal{A}\otimes \R^{\ell})$ 
 as a strong deformation retract.
 This simplicial complex $BC(M(\mathcal{A}\otimes\R^{\ell}))$ is nothing
 but the higher order Salvetti complex in
 \cite{Bjorner-Ziegler92,DeConcini-Salvetti00}. 
\end{example}

The next example is totally normal but not regular.

\begin{example}
 \label{model_for_configuration_space_of_graph}
 By Example \ref{graphs_are_totally_normal}, a graph $X$ can be regarded
 as a totally normal cellular stratified space. As is described in
 \cite{1312.7368}, we may define a totally normal cellular
 stratification on the $k$-fold product $X^k$ by using the braid
 arrangements, which can be 
 restricted to a totally normal 
 cellular stratification on $\Conf_k(X)$. By Corollary
 \ref{subdivision_of_totally_normal_cellular_stratified_space},
 $BC(\Conf_k(X))$ can be embedded in $\Conf_k(X)$ as a
 $\Sigma_k$-equivariant strong deformation retract.

 $BC(\Conf_k(X))$ is a regular CW complex model for $\Conf_k(X)$. This
 is better than Abrams' model in the sense that it works for all graphs.
\end{example}

The cellular stratification on $\CP^n$ in Example
\ref{moment_angle_complex} is cylindrically normal and 
Theorem \ref{embedding_Sd}
applies. 

\begin{example}
 Consider the cylindrically normal cellular stratification on $\CP^n$ in
 Example \ref{CP^2} and Example \ref{moment_angle_complex}.
 By Theorem \ref{embedding_Sd},
 $BC(\CP^n)$ is homeomorphic to $\CP^n$.
 Let us compare the description of $BC(\CP^n)$
 with the one in Example \ref{moment_angle_complex} as the
 Davis-Januszkiewicz construction $M(\lambda_n)$. 

 As we have
 seen in Example \ref{face_category_of_CP^2}, the
 cylindrical face category $C(\CP^n)$ can be obtained from the
 underlying poset $[n]$ of $C(\CP^n)$ by replacing the set of morphisms
 by $S^1$. Compositions of morphisms are given by the group structure of
 $S^1$. 

 By Lemma \ref{underlying_poset_nerve}, we have
 \[
  \overline{N}_k(C(\CP^n)) = \coprod_{\bm{e}\in
 \overline{N}_k([n])}\{\bm{e}\}\times(S^1)^k 
 \]
 and the classifying space of $C(\CP^n)$ can be described in the
 form 
 \begin{eqnarray*}
  BC(\CP^n) & = & \left\|\overline{N}(C(\CP^n))\right\| \\
  & = & \quotient{\coprod_k\coprod_{\bm{e}\in \overline{N}_k([n])}
   \{\bm{e}\}\times (S^1)^k\times\Delta^{k}}{\sim}. 
 \end{eqnarray*}

 Note that a nondegenerate $k$-chain $\bm{e}$ in the poset $[n]$
 can be described by a strictly increasing sequence of
 nonnegative integers $\bm{e}=(i_0,\ldots,i_k)$ with $i_k\le n$. For an
 element $(\bm{e};t_1,\ldots,t_k;p_0,\ldots,p_k)$ in 
 $\overline{N}_k(C(\CP^n))$, let $\tilde{\bm{t}}$ be an element of
 $(S^1)^n$ obtained from $(t_1,\ldots,t_k)$ by inserting $1$ in such a
 way that each $t_j$ is placed between $i_{j-1}$-th and $i_{j}$-th
 positions. Then there exists a face operator
 $d_I : N_n(C(\CP^n))\to N_k(C(\CP^n))$ such that 
 \[
  (\bm{e};t_1,\ldots,t_k) = (d_I(0,1,\ldots,n),d_{I}(\tilde{\bm{t}})).
 \]
 And we have
 \[
  (\bm{e};t_1,\ldots,t_k;p_0,\ldots,p_k) \sim
 (0,\ldots,n;\tilde{\bm{t}};d^I(p_0,\ldots, p_k)).
 \]
 Note that $d^I(p_0,\ldots, p_k)$ is obtained from $(p_0,\ldots,p_k)$ by
 inserting $0$ in appropriate coordinates.
 Thus any point in $BC(\CP^n)$ can be represented by a point in
 $(S^1)^n\times\Delta^n$ and $BC(\CP^n)$ can be written as
 \begin{equation}
  BC(\CP^n) = (S^1)^n\times \Delta^n/_{\sim}.
   \label{barycentric_subdivision_of_CP^n}
 \end{equation}
 The relation here is not exactly the same as the defining relation of
 $M(\lambda_n)$. 

 Define a map
 \[
  s_n : M(\lambda_n) \longrightarrow BC(\CP^n)
 \]
 by
 \[
  s_n([t_1,\ldots,t_n;p_0,\ldots,p_n]) =
 [t_1,t_1^{-1}t_2,\ldots,t_{n-1}^{-1}t_n; p_0,\ldots,p_n].
 \]
 It is left to the reader to verify that $s_n$ is a well-defined
 homeomorphism making the diagram
 \[
  \begin{diagram}
   \node{M(\lambda_n)} \arrow[2]{e,t}{s_n} \arrow{s} \node{}
   \node{BC(\CP^n)} \arrow{s,r}{B\pi_{C(\CP^n)}} \\
   \node{\Delta^n} \arrow{e,=} \node{B([n])} \arrow{e,=} \node{BP(\CP^n)}
  \end{diagram}
 \]
 commutative, where $\pi_{C(\CP^n)} : C(CP^n)\to P(\CP^n)$ is the
 canonical projection onto the underlying poset\footnote{Definition
 \ref{underlying_poset}}.  
 Thus we see that $BC(\CP^n)$ coincides with $M(\lambda_n)$ up to a
 homeomorphism. 
\end{example}

%% file: basic_operations.tex
\section{Basic Constructions on Cellular Stratified Spaces}
\label{basic_operations}

We study the following operations on cellular and stellar stratified
spaces in this section: 
\begin{itemize}
 \item taking stratified subspaces, cellular stratified subspaces, and
       stellar stratified subspaces,
 \item taking products of cellular stratified spaces, and
 \item taking subdivisions of cellular and stellar stratified spaces.
\end{itemize}

\input{stratified_subspace}

\input{product}

\input{subdivision_of_cell}

%% file: stratified_subspace.tex
\subsection{Stratified Subspaces}
\label{stratified_subspace}

In this section, we consider the problem of restricting cellular
stratifications to subspaces.
Obviously the category of stratified spaces is closed under taking
complements of stratified subspaces.

\begin{lemma}
 Let $X$ be a stratified space and $A$ be a stratified subspace. Then
 the complement $X\setminus A$ is also a stratified subspace of $X$.
\end{lemma}

This is one of the most useful facts when we work with stratifications
on configuration spaces and complements of arrangements. On the other
hand, when $X$ is a CW complex and $A$ is a subcomplex, $A$ is always
closed. This is no longer true for cellular stratified spaces. 
A typical example is the case of the complement
$M(\mathcal{A}\otimes\R^{\ell})=\R^n\otimes\R^{\ell}\setminus\Lk(\mathcal{A}\otimes\R^{\ell})$
of an arrangement in
Example \ref{BZDS_stratification}.

Let us consider cell structures on subspaces.
Suppose $X$ is a cellular stratified space and $A$ is a stratified
subspace. In order to incorporate cell
structures, we need to specify a cell structure for each cell contained
in $A$.

\begin{definition}
 \label{cell_in_subspace}
 Let $X$ be a topological space and $A$ be a subspace. For a
 hereditarily quotient\footnote{Definition
 \ref{hereditarily_quotient_definition}.} 
 $n$-cell structure $\varphi : D \to \overline{e} \subset X$ with
 $e\subset A$, 
 define an $n$-cell structure on $e$ in $A$ to be
 $(D_{A},\varphi|_{D_A})$ where
 $D_{A}=\varphi^{-1}(\overline{e}\cap A)$.
\end{definition}

The hereditarily-quotient assumption guarantees the restriction is a
quotient map by Lemma
\ref{hereditarily_quotient_map_can_be_restricted}. 

\begin{lemma}
 Let $X$ be a topological space and $A$ be a
 subset of $X$. If $(D,\varphi)$ 
 is a hereditarily quotient $n$-cell structure on $e\subset X$ and
 $e\subset A$, then 
 $(D_{A},\varphi|_{D_{A}})$ defined above is an $n$-cell structure on
 $e$ in $A$.
\end{lemma}

\begin{definition}
 \label{cellular_stratified_subspace_definition}
 Let $(X,\pi)$ be a cellular stratified space.
 A stratified subspace $(A,\pi|_{A})$ of $X$ is said to be
 a \emph{cellular stratified subspace}, provided
 cell structures on cells in $A$ are given as indicated
 in Definition \ref{cell_in_subspace}. When the inclusion
 $A\hookrightarrow X$ is a strict morphism, $A$ is said to be a
 \emph{strict cellular stratified subspace}.
\end{definition}


We need to take cell decompositions of domains of cells into account for
stellar stratified subspaces.

\begin{lemma}
 \label{stellar_cell_in_subspace}
 Let $X$ be a topological space, $A$ a subspace, and
 $\varphi : D \to \overline{e}$  a stellar 
 $n$-cell of $X$ with $e\subset A$. When
 $D_{A}=\varphi^{-1}(\overline{e}\cap A)$ is a strict stratified
 subspace of 
 $D$ and $\varphi$ is hereditarily quotient, the restriction
 $\varphi|_{D_A}: D_{A} \to A$ defines a stellar 
 $n$-cell structure on $e$ in $A$. 
\end{lemma}


\begin{definition}
 Let $(X,\pi)$ be a stellar stratified space and $A$ a subspace of
 $X$. If the assumption of 
 Lemma \ref{stellar_cell_in_subspace} is satisfied for each cell in $A$,
 $(A,\pi|_{A})$ is said to be a \emph{stellar stratified subspace} of $X$.
\end{definition}

\begin{remark}
 Let $X$ be a cellular or stellar stratified space and $A$ be a cellular
 or stellar stratified subspace. Then $A$ is a strict stratified
 subspace of $X$.
\end{remark}

The following fact can be regarded as a generalization of the fact that
any subcomplex of a regular cell complex is regular.

\begin{proposition}
 \label{cylindrically_normal_subspace}
 Any stellar stratified subspace $A$ of a cylindrically normal stellar
 stratified space $X$ is cylindrically normal. Furthermore the
 parameter space for a pair of cells $e_{\mu}<e_{\lambda}$ in
 $A$ can be identified with the parameter space for the same pair when
 regarded as cells in $X$.
\end{proposition}

\begin{proof}
 Let $e_{\mu}< e_{\lambda}$ be a pair of cells in $A$. Let
 $\varphi_{\mu} : D_{\mu} \to X$ and
 $\varphi_{\lambda} : D_{\lambda} \to X$ be the
 stellar structures for these cells in $X$. When regarded as cells in
 $A$, their stellar structures are denoted by 
 \begin{eqnarray*}
  \varphi_{A,\mu} & : & D_{A,\mu} \longrightarrow A \\
  \varphi_{A,\lambda} & : & D_{A,\lambda} \longrightarrow A,
 \end{eqnarray*}
 respectively.

 We need to show that the structure map
 \[
 b_{\mu,\lambda} : P_{\mu,\lambda}\times D_{\mu} \longrightarrow
 D_{\lambda} 
 \]
 of cylindrical structure for the pair in $X$ can be restricted to
 \[
 b_{\mu,\lambda}|_{P_{\mu,\lambda}\times D_{A,\mu}} :
 P_{\mu,\lambda}\times D_{A,\mu} \longrightarrow 
 D_{A,\lambda}. 
 \]

 This can be verified by the commutativity of the diagram
 \[
  \begin{diagram}
   \node{P_{\mu,\lambda}\times D_{A,\mu}} \arrow{se,J} \arrow[2]{e,..}
   \arrow[3]{s,l}{\pr_2} \node{} \node{D_{A,\lambda}} \arrow{s,J}
   \arrow[2]{e,t}{\varphi_{A,\lambda}} 
   \node{} \node{\overline{e_{\lambda}\cap A}} \arrow{sw,J} \\
   \node{} \node{P_{\mu,\lambda}\times D_{\mu}} \arrow{s,l}{\pr_2}
   \arrow{e,t}{b_{\mu,\lambda}} 
   \node{D_{\lambda}} \arrow{e,t}{\varphi_{\lambda}}
   \node{\overline{e}_{\lambda}} \node{} \\ 
   \node{} \node{D_{\mu}} \arrow[2]{e,b}{\varphi_{\mu}} \node{}
   \node{\overline{e}_{\mu}} \arrow{n,J}
   \node{} \\ 
   \node{D_{A,\mu}} \arrow{ne,J} \arrow[4]{e,b}{\varphi_{A,\mu}} \node{}
   \node{} \node{} 
   \node{\overline{e}_{\mu}\cap A} \arrow{nw,J} \arrow[3]{n,J}
  \end{diagram}
 \]
 and the definition of stellar structures on $A$. 
\end{proof}

\begin{example}
 Consider the stratification in Example \ref{BZDS_stratification}. The
 link 
 $\Lk(\mathcal{A}\otimes\R^{\ell})$ and the complement
 $M(\mathcal{A}\otimes\R^{\ell})$ are both cellular
 stratified subspaces 
 of $(\R^n\otimes\R^{\ell}, \pi_{\mathcal{A}\otimes\R^{\ell}})$, which
 is regular, hence cylindrically normal.
\end{example}

Thanks to Corollary \ref{characteristic_map_of_locally_polyhedral} and
Lemma \ref{biquotient_is_quotient},
cell structures in a polyhedral stellar stratified space
$X$ are hereditarily quotient. Thus any stellar stratified subspace $A$
inherits a cylindrically normal structure with structure maps satisfying
the PL conditions in the definition of polyhedral cellular
stratification. The problem is the CW condition.
It is easy to see that the closure finiteness condition can be restricted
freely. The question is when a stratified subspace of a CW stratified
subspace inherits the weak topology.

\begin{lemma}
 \label{subspace_of_CW_is_CW}
 A closed or an open stratified subspace $A$ of a CW stratified space
 $X$ is CW. 
\end{lemma}

\begin{proof}
 This follows from the corresponding property of weak topology.
\end{proof}


%

By Lemma \ref{cellular_closure}, any polyhedral cellular
stratified space $X$ can be embedded in a CW complex $U(X)$. In general,
however, a cellular stratified subspace $A$ of $X$ is neither closed nor
open in $U(X)$ and it is not easy to 
verify the weak topology condition. One of the practical conditions is
the locally finiteness. The CW condition is guaranteed by Proposition
\ref{locally_finite_implies_CW}. 

\begin{proposition}
 \label{locally_polyhedral_subspace}
 Let $X$ be a polyhedral cellular stratified space. Any
 locally finite cellular stratified subspace $A$ is polyhedral.
\end{proposition}


%
%
%

We end this section by an example which shows another difference between
cellular stratified subspaces and subcomplexes.
In the case of CW complexes, the colimit of an increasing sequence of
finite subcomplexes 
\[
 X_0 \subset X_1\subset \cdots \subset \colim_n X
\]
is automatically a CW complex. This is not true for cellular stratified
spaces. 

\begin{example}
 Consider the space
 \[
  X = \set{(x,y)\in\R^2}{y>0}\cup \Z\times\{0\}.
 \]
 The homeomorphism
 \[
  p : D^2\setminus\{(0,1)\} \longrightarrow \set{(x,y)\in\R^2}{y\ge 0}
 \]
 given by extending the stereographic projection
 $S^1\setminus\{(0,1)\} \to \R$, 
 \begin{center}
  \begin{tikzpicture}
   \shade (0,0) circle (1.2cm);
   \draw (0,0) circle (1.2cm);
   \draw (0,1.2) circle (2pt);
   \draw [fill,white] (0,1.2) circle (1pt);
   \draw (0,0) node {$D^2\setminus \{(0,1)\}$};

   \draw [->] (2,0) -- (3,0);
   \draw (2.5,0.3) node {$p$};

   \shade (3.8,-1.2) rectangle (6.2,1.2);
   \draw (5,0) node {$\set{(x,y)}{y\ge 0}$};
   \draw (3.8,-1.2) -- (6.2,-1.2);
  \end{tikzpicture}
 \end{center}
 defines a $2$-cell structure on 
 \[
  e^2 = \set{(x,y)\in\R^2}{y>0} \subset X
 \]
 by restricting $p$ to $D=p^{-1}(X)$.

 Each $\{(n,0)\} \subset \Z\times\{0\}$ can be regarded as a $0$-cell
 $e_n^0$. And we have a cellular stratification on $X$
 \[
  X = \left(\bigcup_{n\in\Z} e_n^0\right)\cup e^2.
 \]
 \begin{center}
  \begin{tikzpicture}
   \shade (0,0) circle (1.2cm);
   \draw [dotted] (0,0) circle (1.2cm);
   \draw (0,0) node {$D$};
   \draw [fill] (0,-1.2) circle (1pt);
   \draw [fill] (0.9,-0.78) circle (1pt);
   \draw [fill] (-0.9,-0.78) circle (1pt);
   \draw [fill] (1.2,0) circle (1pt);
   \draw [fill] (-1.2,0) circle (1pt);

   \draw [->] (2,0) -- (3,0);
   \draw (2.5,0.3) node {$p|_{D}$};

   \shade (3.8,-1.2) rectangle (6.2,1.2);
   \draw (5,0) node {$X$};
   \draw [dotted] (3.8,-1.2) -- (6.2,-1.2);
   \draw [fill] (4,-1.2) circle (1pt);
   \draw [fill] (4.5,-1.2) circle (1pt);
   \draw [fill] (5,-1.2) circle (1pt);
   \draw [fill] (5.5,-1.2) circle (1pt);
   \draw [fill] (6,-1.2) circle (1pt);
  \end{tikzpicture}
 \end{center}

 $X$ is a colimit of 
 \[
  X_n = \set{(x,y)\in\R^2}{y>0} \cup \set{i\in\Z}{|i|\le n}\times\{0\}.
 \]
 Each $X_n$ is a finite stratified subspace of $X$, hence is CW. But $X$
 is not CW. 
\end{example}

%

%% file: product.tex
\subsection{Products}
\label{product}

In this section, we study products of stratifications, cell structures,
and cellular stratified spaces and deduce a couple of conditions under
which we may take products.

It is not difficult to define a stratification on the product of
two stratified spaces, as we have seen in Lemma
\ref{product_stratification}. 
We have to be careful, however, when we take products of cellular
stratified spaces. 
Even in the category of CW complexes, there is a well-known difficulty
in taking products. Given CW complexes $X$ and $Y$, we need to impose the 
local-finiteness on $X$ or $Y$ or to redefine a topology on $X\times Y$
in order to make $X\times Y$ into a CW complex.

In the case of cellular stratified spaces, we have another
difficulty because of our requirement on cell structures. The product of
two quotient maps may not be a quotient map.

\begin{definition}
 \label{product_cellular_stratification}
 Let $(X,\pi_X,\Phi_X)$ and $(Y,\pi_Y,\Phi_Y)$ be cellular stratified
 spaces and consider the product stratification
 \[
  \pi_X\times\pi_Y : X\times Y \longrightarrow P(X)\times P(Y)
 \]
 on $X\times Y$ in Lemma \ref{product_stratification}. 
 For a pair of cells $e_{\lambda}\subset X$ and $e_{\mu}\subset Y$,
 consider the composition 
 \[
 \varphi_{\lambda,\mu} : D \cong D_{\lambda}\times D_{\mu}
 \rarrow{\varphi_{\lambda}\times\varphi_{\mu}} 
 \overline{e_{\lambda}}\times \overline{e_{\mu}} =
 \overline{e_{\lambda}\times e_{\mu}} \subset X\times Y, 
 \]
 where $D$ is the subspace of $D^{\dim e_{\lambda}+\dim e_{\mu}}$
 defined by pulling back $D_{\lambda}\times D_{\mu}$ via the standard
 homeomorphism  
 \[
 D^{\dim e_{\lambda}+\dim e_{\mu}} \cong D^{\dim e_{\lambda}}\times
 D^{\dim e_{\mu}}. 
 \]
 If $\varphi_{\lambda,\mu}$ is a quotient map for each pair of cells,
 the resulting cellular stratification is called the
 \emph{product cellular stratification} on $X\times Y$. 
\end{definition}

The above definition is incomplete. Unless we have a
general criterion for $\varphi_{\lambda,\mu}$ to be a quotient map, this 
definition is useless. 

By Lemma \ref{biquotient_is_quotient} and Proposition
\ref{biquotients_are_closed_under_product}, we can take products of
bi-quotient cell structures. The question is when a cell
structure is a bi-quotient map. By Corollary \ref{compact_fiber} and
Corollary \ref{characteristic_map_of_locally_polyhedral}, we obtain the
following practical conditions under the assumption of cylindrical
normality. 

\begin{proposition}
 \label{relatively_compact_or_locally_polyhedral}
 Let $X$ and $Y$ be cylindrically normal cellular stratified spaces.
 If they satisfy one of the following conditions, any product
 $e_{\lambda}\times e_{\mu}$ has the product cell structure for
 $\lambda\in P(X)$ and $\mu\in P(Y)$:
 \begin{enumerate}
  \item All parameter spaces are compact.
  \item Polyhedral.
 \end{enumerate}
\end{proposition}

\begin{proof}
 If $X$ or $Y$ is  polyhedral, all cell structures are
 bi-quotient by Corollary
 \ref{characteristic_map_of_locally_polyhedral}. By Lemma
 \ref{compact_parameter_space}, each fiber of a cell structure in a
 cylindrically normal cellular stratified space can be identified with a
 parameter space. Thus, when all parameter spaces are compact, the cell
 structures are bi-quotient by Corollary \ref{compact_fiber}.
\end{proof}

If $X$ and $Y$ satisfy one of the above conditions, the product
$X\times Y$ has a structure of cellular stratified space. It is
reasonable to expect that $X\times Y$ is again cylindrically normal.

\begin{theorem}
 \label{product_cylindrical_structure}
 Let $X$ and $Y$ be cylindrically normal cellular stratified spaces
 satisfying one of the conditions in Proposition
 \ref{relatively_compact_or_locally_polyhedral}. 
 Then the product $X\times Y$ is a cylindrically normal cellular
 stratified space.
\end{theorem}

\begin{proof}
 Let
 $\{\varphi_{\lambda}^{X}:D_{\lambda}\to\overline{e_{\lambda}}\}$
 and 
 $\{\varphi_{\mu}^{Y}: D_{\mu}\to \overline{e_{\mu}}\}$
 be cell structures on $X$ and $Y$, respectively. Cylindrical structures
 on $X$ and $Y$ are denoted by
 $\{b^X_{\lambda,\lambda'} : P^X_{\lambda,\lambda'}\times D_{\lambda}\to
 D_{\lambda'}\}$ 
 and
 $\{b^Y_{\mu,\mu'} : P^Y_{\mu,\mu'}\times D_{\mu}\to D_{\mu'}\}$, respectively.

 For a pair of cell structures
 $\varphi_{\lambda}^{X} : D_{\lambda}\to X$ and
 $\varphi_{\mu}^{Y} : D_{\mu}\to Y$,
 consider the product cell structure
 \[
  \varphi_{\lambda,\mu} : D_{\lambda,\mu} \cong D_{\lambda}\times
 D_{\mu} \rarrow{\varphi^X_{\lambda}\times\varphi^{Y}_{\mu}}
 \overline{e_{\lambda}}\times\overline{e_{\mu}}. 
 \]
 For $(\lambda,\mu)\le (\lambda',\mu')$, define
 \[
 P_{(\lambda,\mu),(\lambda',\mu')} = P^X_{\lambda,\lambda'}\times
 P^Y_{\mu,\mu'}. 
 \]
 When $(\lambda,\mu) < (\lambda',\mu')$, we have either 
 $\lambda<\lambda'$ or $\mu<\mu'$ and thus the image of the composition 
 \begin{eqnarray*}
  P_{(\lambda,\mu),(\lambda',\mu')}\times D_{\lambda,\mu} & \cong &
   P^X_{\lambda,\lambda'}\times P^Y_{\mu,\mu'}\times D_{\lambda}\times
   D_{\mu} \\ 
  & \rarrow{\cong} & P^X_{\lambda,\lambda'}\times D_{\lambda}\times
   P^Y_{\mu,\mu'}\times D_{\mu} \\
  & \rarrow{b^X_{\lambda,\lambda'}\times b^Y_{\mu,\mu'}} &
   D_{\lambda'}\times D_{\mu'} \\ 
  & \cong & D_{\lambda',\mu'} 
 \end{eqnarray*}
 lies in
 $\partial D_{\lambda,\mu}\cong (D_{\lambda}\times\partial D_{\mu})\cup (\partial D_{\lambda}\times D_{\mu})$.
 And we obtain a map
 \[
 b_{(\lambda,\mu),(\lambda',\mu')} :
 P_{(\lambda,\mu),(\lambda',\mu')}\times D_{\lambda,\mu} \rarrow{}
 \partial D_{\lambda',\mu'}.
 \]
 The composition operations
 \[
  P_{(\lambda_1,\mu_1),(\lambda_2,\mu_2)}\times
 P_{(\lambda_0,\mu_0),(\lambda_1,\mu_1)} \longrightarrow
 P_{(\lambda_0,\mu_0),(\lambda_2,\mu_2)} 
 \]
 are defined in an obvious way.

 It is straightforward to verify that these maps define a cylindrical
 structure on $X\times Y$ under the product stratification and the
 product cell structures.
\end{proof}

Let us consider the CW conditions on products next. The closure
finiteness condition is automatic.

\begin{lemma}
 If $X$ and $Y$ are stratified spaces satisfying the closure finiteness
 condition, then so is $X\times Y$.
\end{lemma}

As is the case of CW complexes \cite{Dowker1952}, the product of two CW 
stratifications may not satisfy the weak topology condition. In
the case of CW complexes, the local finiteness of $X$ implies that
$X\times Y$ is CW for any CW complex $Y$. The author does not know if an
analogous fact holds for CW stratified spaces in general.
The following obvious fact is still useful in many cases. 

\begin{lemma}
 \label{product_of_CW_stratification}
 Let $X$ and $Y$ be CW stratified spaces. If both $X$ and $Y$ are
 locally finite, then $X\times Y$ is a CW stratified space. 
\end{lemma}

\begin{proof}
 The product of locally finite stratified spaces is again 
 locally finite. The result follows from Proposition
 \ref{locally_finite_implies_CW}. 
\end{proof}

%

Thus, by Theorem \ref{product_cylindrical_structure}, the product
$X\times Y$ of locally finite cylindrically normal cellular stratified
spaces $X$ and $Y$ is a CW cylindrically normal cellular stratified
space. When $X$ and $Y$ are polyhedral, it is easy to verify
that $X\times Y$ inherits a polyhedral structure.

\begin{corollary}
 \label{product_of_locally_polyhedral}
 Let $X$ and $Y$ be polyhedral cellular stratified spaces.
 Suppose both $X$ and $Y$ are locally finite. Then
 the product $X\times Y$ is a polyhedral cellular stratified
 space.
\end{corollary}

Thus we may freely take finite products of Euclidean polyhedral
cellular stratified spaces.

\begin{corollary}
 \label{product_of_Euclidean_locally_polyhedral}
 For any Euclidean polyhedral cellular stratified spaces $X$ and
 $Y$, the product $X\times Y$ is polyhedral.
\end{corollary}

\begin{proof}
 By Lemma \ref{metrizable_implies_locally_finite}, $X$ and
 $Y$ are locally finite.
\end{proof}

%
%
%
%
%
%

In Proposition \ref{product_cylindrical_structure} and Definition
\ref{product_cellular_stratification}, we 
made an implicit choice of a homeomorphism 
\[
 D^{m+n} \cong D^m\times D^n.
\]
This procedure can be avoided by using cubes as domains of
cell structures, in which case it is a reasonable idea to require
the structure maps to be compatible with
cubical structures. Thus we introduce the following variant of
polyhedral cellular stratified spaces.

\begin{definition}
 \label{cubical_structure_definition}
 Let $X$ be a CW cellular stratified space. We
 consider $I^n=(\Delta^1)^n$ as a stratified space (regular cell
 complex) under the product
 stratification of the simplicial stratification on $\Delta^1$. 
 A \emph{cubical structure} on $X$ consists of
 \begin{itemize}
  \item a cylindrical structure 
	$(\{b_{\mu,\lambda} : P_{\mu,\lambda}\times D_{\mu}\to\partial
	D_{\lambda}\}$, 
	$\{c_{\lambda_0,\lambda_1,\lambda_2}:
	P_{\lambda_1,\lambda_2}\times P_{\lambda_0,\lambda_1} \to
	P_{\lambda_0,\lambda_2}\})$, 
  \item a stratified subspace $Q_{\lambda}$ of $I^{\dim e_{\lambda}}$
	for each $\lambda\in P(X)$ under a suitable regular cellular
	subdivision of $I^{\dim e_{\lambda}}$, and
  \item a homeomorphism
	$\alpha_{\lambda} : Q_{\lambda}\to D_{\lambda}$ for each
	$\lambda\in P(X)$, 
 \end{itemize}
 satisfying the following conditions:
 \begin{enumerate}
  \item Each $P_{\mu,\lambda}$ is a stratified subspace of
	$I^{\dim e_{\lambda}-\dim e_{\mu}}$.
  \item The composition
	\[
	 \tilde{b}_{\mu,\lambda} : P_{\mu,\lambda}\times Q_{\mu}
	\rarrow{1_{P_{\mu,\lambda}}\times 
	\alpha_{\mu}} P_{\mu,\lambda}\times
	D_{\mu} \rarrow{b_{\mu,\lambda}} \partial D_{\lambda}
	\rarrow{\alpha_{\lambda}^{-1}} 
	Q_{\lambda} 
	\]
	is a strict morphism of stratified spaces.
  \item The map $\tilde{b}_{\mu,\lambda}$ is an affine embedding onto
	its image when restricted to each face.
 \end{enumerate}
 The family of maps
 $\{\alpha_{\lambda} : Q_{\lambda}\to D_{\lambda}\}_{\lambda\in P(X)}$
 is also called a \emph{cubical structure}. 

 A cylindrically normal CW cellular stratified space equipped with a 
 cubical structure is called a \emph{cubically normal cellular
 stratified space}.
\end{definition}

\begin{example}
 The minimal cell decomposition of $S^n$ is cubically normal. The
 radial expansion
 \[
  \alpha_n : I^n \longrightarrow D^n
 \]
 defines a cubical structure. The parameter space between the $0$-cell
 and the $n$-cell is $\partial I^n$ and is a stratified subspace of
 $I^n$. 
\end{example}

\begin{example}
 Recall that any $1$-dimensional cellular stratified space is totally
 normal, hence is cylindrically normal. They are cubically normal for
 an obvious reason, if they are CW.
\end{example}

\begin{example}
 \label{cubical_RP^2}
 Recall that $\RP^2$ can be obtained by gluing the edges of $I^2$ as
 follows. 
 \begin{center}
  \begin{tikzpicture}
   \draw (0,0) -- (0,2);
   \draw [->] (0,0) -- (0,1);
   \draw (0,0) -- (2,0);
   \draw [->>] (0,0) -- (1,0);
   \draw (0,2) -- (2,2);
   \draw [->>] (2,2) -- (1,2);
   \draw (2,0) -- (2,2);
   \draw [->] (2,2) -- (2,1);
  \end{tikzpicture}
 \end{center}
 This can be regarded as a description of a cell decomposition of
 $\RP^2$, consisting of two $0$-cells $e^0_1$, $e^0_2$, two $1$-cells
 $e^1_1$, $e^1_2$ and a $2$-cell $e^2$. This is obviously cubically
 normal. 
\end{example}

%
%

\begin{remark}
 When the domains of cells are not globular\footnote{Definition
 \ref{globular_definition}}, the situation is more 
 complicated and we do not discuss the product structure here. 
\end{remark}

%% file: subdivision_of_cell.tex
\subsection{Subdivisions of Cells}
\label{subdivision_of_cell}

We defined the notion of cellular stratified subspace in Definition
\ref{cellular_stratified_subspace_definition}. It often happens that we
need to subdivide cells before we take a stratified subspace.
For example, the complement of an arrangement
$M(\mathcal{A}\otimes\R^{\ell})$ in Example 
\ref{BZDS_stratification} was defined first by taking a subdivision of
the trivial stratification on $\R^n\otimes\R^{\ell}$ and then by taking
the complement of $\Lk(\mathcal{A}\otimes\R^{\ell})$.

We have already defined subdivisions of stratified spaces in Definition
\ref{morphism_of_stratified_spaces}. 
We impose the following ``regularity condition'' on the
definition of subdivisions of cell structures.

\begin{definition}
 \label{cellular_subdivision_definition}
 Let $(\pi,\Phi)$ be a cellular stratification on $X$. A
 \emph{cellular subdivision} of $(\pi,\Phi)$ consists of
 \begin{itemize}
  \item a subdivision of stratified spaces
	\[
	\bm{s}=(1_X,\underline{s}) : (X,\pi')
	\longrightarrow (X,\pi) 
	\]
	and
  \item a regular cellular stratification
	$(\pi_{\lambda},\Phi_{\lambda})$ on 
	the domain $D_{\lambda}$ of each cell $e_{\lambda}$ in
	$(\pi,\Phi)$ containing $\Int(D_{\lambda})$ as a strict
	stratified subspace, 
 \end{itemize}
 satisfying the following conditions: 
 \begin{enumerate}
  \item For each $\lambda\in P(X,\pi)$, 
	the cell structure
	\[
	\varphi_{\lambda} : (D_{\lambda},\pi_{\lambda})
	\longrightarrow (X,\pi') 
	\]
	of $e_{\lambda}$ is a strict morphism of stratified
	spaces.
  \item The maps
	\[
	P(\varphi_{\lambda}) : P(\Int(D_{\lambda})) \longrightarrow
	P(X,\pi) 
	\]
	induced by the cell structures $\{\varphi_{\lambda}\}$ 
	give rise to a bijection
	\[
	P(\Phi)=\coprod_{\lambda\in P(X,\pi)}P(\varphi) :
	\coprod_{\lambda\in P(X,\pi)} 
	P(\Int(D_{\lambda}),\pi_{\lambda}) \longrightarrow P(X,\pi').
	\]
 \end{enumerate}
\end{definition}

\begin{remark}
 \label{cellular_subdivision_remark}
 The morphism $\bm{s}$ induces a surjective map 
 \[
  P(\bm{s}) : P(X,\pi') \longrightarrow P(X,\pi),
 \]
 which gives rise to a decomposition
 \[
  P(X,\pi') = \coprod_{\lambda\in P(X,\pi)} P(\bm{s})^{-1}(\lambda)
 \]
 of sets.
 Since $\varphi_{\lambda}:(D_{\lambda},\pi_{\lambda})\to (X,\pi')$ is a
 strict morphism, it induces an isomorphism of posets
 \[
  P(\varphi_{\lambda}) : P(\Int(D_{\lambda}),\pi_{\lambda})
 \rarrow{\cong} P(\bm{s})^{-1}(\lambda).  
 \]
 In other words, the restriction of the bijection $P(\Phi)$ in the
 second condition to each $P(\Int(D_{\lambda}),\pi_{\lambda})$ 
 is an embedding of posets, although $P(\Phi)$ itself is rarely an
 isomorphism of posets.  
\end{remark}

\begin{remark}
 Cellular subdivisions of stellar stratified spaces can be defined in a
 similar way. We may also define stellar subdivisions of cellular or
 stellar stratified spaces. We do not pursue such generalizations in
 this paper.
\end{remark}

It is easy to verify that the category of cellular stratified spaces is
closed under cellular subdivisions, when all cell structures are
hereditarily quotient.

\begin{lemma}
 \label{cellular_subdivision_is_cellular}
 Let $(X,\pi,\Phi)$ be a cellular stratified space whose cell structures
 are hereditarily quotient. Then any cellular subdivision of
 $(X,\pi,\Phi)$ defines a structure of cellular stratified space on $X$,
 under which, for $\lambda\in P(X,\pi)$ and 
 $\lambda'\in P(\Int(D_{\lambda}),\pi_{\lambda})$, the composition 
 \[
 D_{\lambda'} \rarrow{s_{\lambda'}} D_{\lambda}
 \rarrow{\varphi_{\lambda}} X
 \]
 is the cell structure of the cell $e_{\lambda'}$ in
 $P(X,\pi')$, where 
 $s_{\lambda'} : D_{\lambda'} \to D_{\lambda}$ is the
 cell structure of the cell in
 $P(D_{\lambda},\pi_{\lambda})$ indexed by $\lambda'$. 
\end{lemma}

\begin{proof}
 By assumption, each cell structure $\varphi_{\lambda}$ is hereditarily
 quotient and hence the composition 
 $\varphi_{\lambda}\circ s_{\lambda'}:D_{\lambda'}\to\overline{e_{\lambda'}}$
 is a quotient map. 
 Each new stratum is connected, since $e_{\lambda'}$ is connected. It is
 also locally closed, since $\varphi_{\lambda}|_{\Int(D_{\lambda})}$ is
 a homeomorphism onto its image.
 Other conditions can be verified immediately. 
\end{proof}

\begin{remark}
 \label{subdivision_of_domain}
 The reader might want to define a subdivision of a cellular stratified
 space $(X,\pi,\Phi)$ as a morphism
 \[
 \bm{s}=(1_{X},\underline{s},\{s_{\lambda'}\}_{\lambda'\in P(X,\pi')}) :
 (X,\pi',\Phi') \longrightarrow (X,\pi,\Phi)   
 \]
 of cellular stratified spaces satisfying the condition that,
 for each cell $e_{\lambda}$ in $(\pi,\Phi)$, the stratification of the
 interior of the domain $D_{\lambda}$ is indexed by
 $\underline{s}^{-1}(\lambda)$ 
 \[
  \Int(D_{\lambda}) = \bigcup_{\lambda'\in\underline{s}^{-1}(\lambda)}
 s_{\lambda'}(\Int(D_{\lambda'})). 
 \]
 However, we also need to specify the behavior of each cell structure
 on the boundary $\partial D_{\lambda}$.
\end{remark}

\begin{example}
 \label{symmetric_stratification_by_arrangement}
 Let $\mathcal{A} = \{H_1,\ldots,H_k\}$ be a hyperplane arrangement in
 $\R^n$ defined by 
 affine $1$-forms $L=\{\ell_1,\ldots,\ell_k\}$. The stratification
 $\pi_{\mathcal{A}\otimes\R^{\ell}}$ on $\R^n\otimes\R^{\ell}$ in
 Example \ref{BZDS_stratification} is one of the coarsest
 cellular stratifications containing
 $\Lk(\mathcal{A}\otimes\R^{\ell})=\bigcup_{i=1}^{k} H_k\otimes\R^{\ell}$
 as a stratified subspace.
 
 This efficiency is achieved by sacrificing symmetry. The stratification
 $\pi_{\mathcal{A}\otimes\R^{\ell}}$ is not compatible with the action
 of the symmetric group $\Sigma_{\ell}$ on $\R^{\ell}$. 

 As is stated by Bj{\"o}rner and
 Ziegler in \cite{Bjorner-Ziegler92}, we may subdivide
 $\pi_{\mathcal{A}\otimes\R^{\ell}}$ by using the product sign vector
 $S_1^{\ell}$. Define
 \[
  \pi_{\mathcal{A}\otimes\R^{\ell}}^s : \R^n\otimes\R^{\ell}
 \longrightarrow \Map(L,S_1^{\ell})
 \]
 by
 \[
  \pi_{\mathcal{A}\otimes\R^{\ell}}^s(x_1,\ldots,x_{\ell})(\ell_i) =
 (\sign(\ell_i(x_1)),\ldots,\sign(\ell_i(x_{\ell}))). 
 \]
 Define a map $c : S_1^{\ell}\to S_{\ell}$ of posets by
 \[
  c(\varepsilon_1,\ldots,\varepsilon_{\ell}) = \varepsilon_ie_i,
 \]
 where $i = \max\set{j}{\varepsilon_j\neq 0}$. This is surjective and
 the diagram
 \[
  \begin{diagram}
   \node{\R^n\otimes\R^{\ell}} \arrow{e,=}
   \arrow{s,l}{\pi^s_{\mathcal{A}\otimes\R^{\ell}}}
   \node{\R^n\otimes\R^{\ell}}
   \arrow{s,r}{\pi_{\mathcal{A}\otimes\R^{\ell}}} \\
   \node{\Map(L,S_1^{\ell})} \arrow{e,b}{\Map(L,c)}
   \node{\Map(L,S_{\ell})}
  \end{diagram}
 \]
 is commutative. Thus $\pi^s_{\mathcal{A}\otimes\R^{\ell}}$ is a
 subdivision of the stratification
 $\pi_{\mathcal{A}\otimes\R^{\ell}}$. The regularity of the cellular
 stratification $\pi_{\mathcal{A}\otimes\R^{\ell}}$ implies that this is
 a cellular subdivision. Note that
 $(\R^n\otimes\R^{\ell},\pi^s_{\mathcal{A}\otimes\R^{\ell}})$ is a
 $\Sigma_{\ell}$-cellular stratified space. 
\end{example}

\begin{example}
 \label{subdivision_of_S^2}
 Let $\pi_{\min}$ be the minimal cell decomposition of $S^2$. By
 dividing the $2$-cell by the equator, we obtain a subdivision $\pi$ of
 $\pi_{\min}$ as a stratified space. By dividing the domain of the
 cell structure of the $2$-cell in $\pi_{\min}$, this subdivision
 can be made into a subdivision of cellular stratified spaces, as is
 depicted in Figure \ref{subdivision_of_pi_min}.
 \begin{figure}[ht]
 \begin{center}
  \begin{tikzpicture}
   \draw (-3.5,1) arc (0:180:0.5cm);
   \draw [very thick] (-4.5,1) arc (180:360:0.5cm);
   \draw [->] (-3,1) -- (-2.5,0.8);

   \draw [thick] (-4.5,0) -- (-3.5,0);
   \draw [->] (-3,0) -- (-2.5,0);

   \draw [very thick] (-3.5,-1) arc (0:180:0.5cm);
   \draw (-4.5,-1) arc (180:360:0.5cm);
   \draw [->] (-3,-1) -- (-2.5,-0.8);

   \draw (-1,0) circle (1cm);
   \draw [very thick] (-2,0) -- (0,0);

   \draw [->] (0.5,0) -- (1.5,0);
   \draw (1,0.3) node {$\varphi$};

   \draw (3,0) circle (1cm);
   \draw (2,0) arc (180:360:1cm and 0.5cm); 
   \draw [dotted] (4,0) arc (0:180:1cm and 0.5cm); 
   \draw [fill] (4,0) circle (2pt);
  \end{tikzpicture}
 \end{center}
  \caption{A subdivision of the minimal cell decomposition of $S^2$.}
  \label{subdivision_of_pi_min}
 \end{figure}
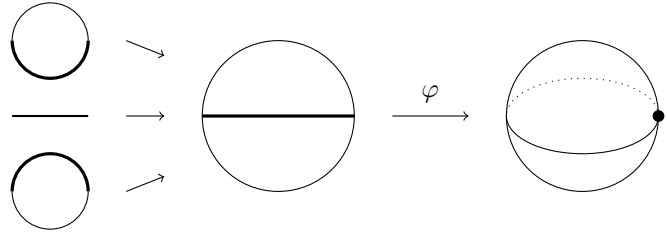

 There is another way to make the stratification $\pi$ into a cellular
 stratified space. Choose a small disk $D$ in $D^2$ touching 
 $\partial D^2$ at a point $p$.
 \begin{figure}[ht]
 \begin{center}
  \begin{tikzpicture}
   \draw (0,0) circle (1cm);
   \draw (0.5,0) circle (0.5cm);
   \draw [fill] (1,0) circle (2pt);
   \draw (1.2,0) node {$p$};
   \draw (0.5,0) node {$D$};

   \draw [->] (1.6,0) -- (2.5,0);
   \draw (2.1,0.3) node {$\varphi'$};

   \draw (4,0) circle (1cm);
   \draw (3,0) arc (180:360:1cm and 0.5cm); 
   \draw [dotted] (5,0) arc (0:180:1cm and 0.5cm); 
   \draw [fill] (5,0) circle (2pt);
  \end{tikzpicture}
 \end{center}
  \caption{Not a subdivision of $\pi_{\min}$.}
  \label{}
 \end{figure}
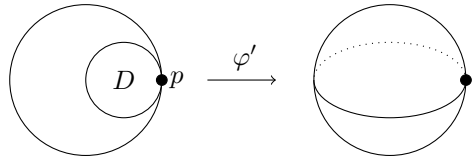
 The small disk $D$ is mapped homeomorphically onto the lower hemisphere
 via $\varphi'$. There is a map
 \[
  \psi : D^2 \longrightarrow D^2\setminus \Int(D)
 \]
 such that the composition
 $\varphi\circ\psi$ is a cell structure
 for the upper hemisphere. For example, such a map can be defined by
 identifying the two points in Figure \ref{definition_of_psi} to $p$.
 \begin{figure}[ht]
 \begin{center}
  \begin{tikzpicture}
   \draw (2,0) arc (0:90:2cm);
   \draw (2,0) arc (0:-90:2cm);
   \draw (0,2) arc (90:270:2cm);
   \draw [fill] (0,2) circle (2pt);
   \draw [fill] (0,-2) circle (2pt);

   \draw [->] (2.5,0) -- (3.5,0);
   \draw (3,0.5) node {$\psi$};

   \draw (6,0) circle (2cm);
   \draw (7,0) circle (1cm);
   \draw [fill] (8,0) circle (2pt);
   \draw (7.7,0) node {$p$};
   \draw (5,0) node {$D^2\setminus\Int(D)$};
  \end{tikzpicture}
 \end{center}
  \caption{A definition of $\psi$.}
  \label{definition_of_psi}
 \end{figure}
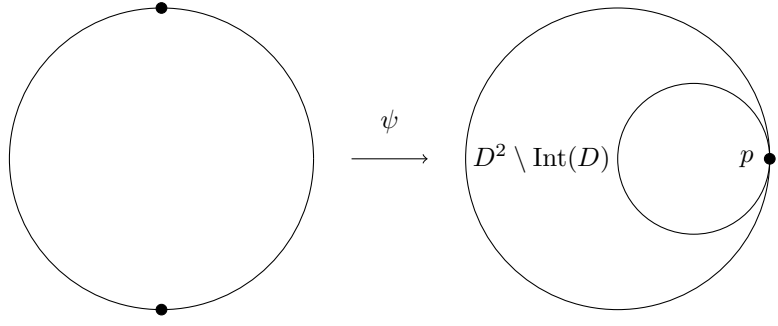
 The morphism of stratified spaces $\pi \to \pi_{\min}$ also becomes a
 morphism of cellular stratified spaces with this cellular
 stratification on $\pi$, but it is not a subdivision of cellular
 stratified spaces, since the above stratification on the domain $D^2$
 is not a regular cellular stratification.
\end{example}

We often define a cellular subdivision of a cellular stratified space
$(X,\pi,\Phi)$ by defining subdivisions of domains of cell structures.
The problem is of course the compatibility.

\begin{lemma}
 \label{compatibility_condition_for_subdivision}
 Let $(X,\pi,\Phi)$ be a cellular stratified space whose cell structures
 are hereditarily quotient. Suppose that a regular
 cellular stratification $\pi_{\lambda}$ on each domain
 $D_{\lambda}$ is given. Define a stratification $\pi'$ on $X$ by
 \[
 X = \bigcup_{\lambda\in P(X,\pi)} \bigcup_{\lambda'\in
 P(\Int(D_{\lambda}))} \varphi_{\lambda}(e_{\lambda'}). 
 \]
 Suppose further that the
 following conditions are satisfied:
 \begin{enumerate}
  \item For each $\lambda\in P(X)$, $\Int(D_{\lambda})$ is a strict
	stratified subspace of $D_{\lambda}$.
  \item For $\lambda'\in P(D_{\lambda})$ and $,\mu'\in P(D_{\mu})$, 
	 \[
	 \varphi_{\lambda}\circ
	 s_{\lambda'}\left(\Int(D_{\lambda'})\right) \cap
	 \varphi_{\mu}\circ s_{\mu'}\left(\Int(D_{\mu'})\right) \neq
	 \emptyset 
	 \]
	 implies
	 \[
	 \varphi_{\lambda}\circ s_{\lambda'}\left(\Int(D_{\lambda'})\right)
	 = \varphi_{\mu}\circ s_{\mu'}\left(\Int(D_{\mu'})\right).	 
	 \]
 \end{enumerate}

 For $\lambda'\in P(\Int(D_{\lambda}))$, define a map
 $\psi_{\lambda'}:D_{\lambda'}\to X$
 by the composition
 \[
  D_{\lambda'} \rarrow{s_{\lambda'}} D_{\lambda}
 \rarrow{\varphi_{\lambda}} X,
 \]
 where $s_{\lambda'}$ is the cell structure of $e_{\lambda'}$.
 Then these structures define a cellular stratification on $X$ with
 $\{\psi_{\lambda'}\}$ cell structure maps.
\end{lemma}

\begin{proof}
 It suffices to verify that $\varphi_{\lambda}$ is a strict morphism of
 cellular stratified spaces from $(D_{\lambda},\pi_{\lambda})$ to
 $(X,\pi')$. 

 By assumption,
 \[
 \varphi_{\lambda}|_{\Int(D_{\lambda})} : \Int(D_{\lambda})
 \longrightarrow e_{\lambda}
 \]
 is an isomorphism of stratified spaces. We need to show that
 $\varphi_{\lambda}$ is a strict morphism of stratified spaces on
 $\partial D_{\lambda}$. For a cell
 $e_{\lambda'}\subset \partial D_{\lambda}$, $\varphi(e_{\lambda})$ is
 contained in $X\setminus e_{\lambda}$ and there exist $\mu\in P(X)$
 and $\mu'\in P(\Int(D_{\mu}))$ such that
 \[
 \varphi_{\lambda}\circ s_{\lambda'}(\Int(D_{\lambda'}))\cap
 \varphi_{\mu}\circ s_{\mu'}(\Int(D_{\mu'}))\neq \emptyset.
 \]
 By assumption, 
 \[
 \varphi_{\lambda}\circ s_{\lambda'}(\Int(D_{\lambda'})) =
 \varphi_{\mu}\circ s_{\mu'}(\Int(D_{\mu'}))
 \]
 and thus $\varphi_{\lambda} : (D_{\lambda},\pi_{\lambda})\to (X,\pi')$
 is a strict morphism of stratified spaces.
\end{proof}

\begin{definition}
 \label{induced_subdivision}
 We say that the above cellular stratification is \emph{induced} by the
 family $\{\pi_{\lambda}\}_{\lambda\in P(X)}$ of cellular stratifications
 on the domains of cells. 
\end{definition}

\begin{example}
 \label{tilted_subdivision}
 Consider the cellular stratification $A=e^1\cup e^2$ of an open annulus
 $A=S^1\times(0,1)$ in Figure \ref{stratification_on_annulus}.  
 \begin{figure}[ht]
 \begin{center}
  \begin{tikzpicture}
   \draw [dotted] (0,0) -- (2,0);
   \draw (0,0) -- (0,2);
   \draw [->] (0,0) -- (0,1);
   \draw (0,1) -- (0,2);
   \draw [dotted] (0,2) -- (2,2);
   \draw (2,0) -- (2,2);
   \draw [->] (2,0) -- (2,1);
   \draw (2,1) -- (2,2);

   \draw [->] (3,1) -- (4,1);
   \draw (3.5,1.3) node {$\varphi_2$};

   \draw [dotted] (6,2) ellipse (1cm and 0.5cm);
   \draw [dotted] (6,0) ellipse (1cm and 0.5cm);
   \draw (5,0) -- (5,2);
   \draw (7,0) -- (7,2);
   \draw [->] (6.5,-0.43) -- (6.5,0.57);
   \draw (6.5,0.57) -- (6.5,1.57);
  \end{tikzpicture}
 \end{center}
  \label{stratification_on_annulus}
  \caption{A cellular stratification on an open annulus.}
 \end{figure}
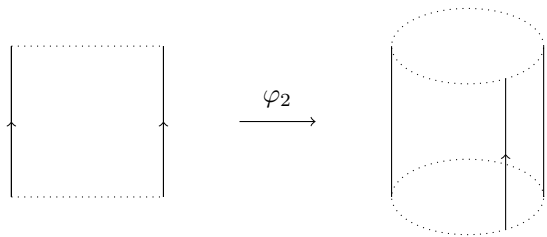

 We use $(0,1)$ and $(0,1)\times [0,1]$ as the domains of the
 cell structures for the $1$-cell and the $2$-cell, respectively.
 The subdivisions of $(0,1)$ in the middle and of $(0,1)\times[0,1]$
 by horizontal cut in the middle induce a subdivision of this
 stratification. 

 The following ``tilted subdivision'' of $(0,1)\times [0,1]$, however,
 does not induce a subdivision of the 
 annulus, since it does not satisfy 
 the second condition of cellular subdivision.
 \begin{figure}[ht]
 \begin{center}
  \begin{tikzpicture}
   \draw [dotted] (0,0) -- (2,0);
   \draw (0,0) -- (0,2);
   \draw [->] (0,0) -- (0,1);
   \draw (0,1) -- (0,2);
   \draw [dotted] (0,2) -- (2,2);
   \draw (2,0) -- (2,2);
   \draw [->] (2,0) -- (2,1);
   \draw (2,1) -- (2,2);
   \draw (0,0.5) -- (2,1.5);
   \draw [fill] (0,0.5) circle (2pt);
   \draw [fill] (2,1.5) circle (2pt);

   \draw [dotted] (3,0) -- (5,0);
   \draw (3,0) -- (3,2);
   \draw [->] (3,0) -- (3,1);
   \draw (3,1) -- (3,2);
   \draw [dotted] (3,2) -- (5,2);
   \draw (5,0) -- (5,2);
   \draw [->] (5,0) -- (5,1);
   \draw (5,1) -- (5,2);
   \draw (3,0.5) -- (5,1.5);
   \draw [fill] (3,0.5) circle (2pt);
   \draw [fill] (3,1.5) circle (2pt);
   \draw [fill] (5,0.5) circle (2pt);
   \draw [fill] (5,1.5) circle (2pt);
  \end{tikzpicture}
 \end{center}
  \caption{A tilted subdivision of an open annulus.}
 \end{figure}
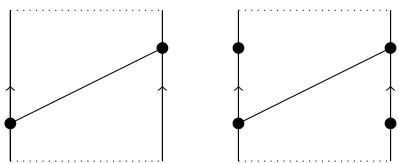

 By subdividing the boundary further, we obtain an induced subdivision,
 as is shown in the right figure. 
\end{example}

For totally normal cellular stratified spaces, we require the following
conditions.

\begin{definition}
 \label{totally_normal_subdivision}
 Let $(X,\pi,\Phi)$ be a totally normal cellular stratified space. We
 say a cellular subdivision $(X,\pi',\Phi')$ is a \emph{subdivision} of
 $(X,\pi,\Phi)$ as a totally normal cellular stratified space if the
 following conditions are satisfied:
 \begin{enumerate}
  \item For each $\lambda'\in P(X,\pi')$, there exists a structure of
	regular cell complex on $\partial D^{\dim e_{\lambda'}}$
	containing $\partial D_{\lambda'}$ as a strict stratified
	subspace. 
  \item For each $b\in C(X,\pi)(e_{\mu},e_{\lambda})$, the associated
	map on domains of cells is a strict morphism
	$b:(D_{\mu},\pi_{\mu})\to (D_{\lambda},\pi_{\lambda})$ of
	cellular stratified spaces.
  \end{enumerate}
\end{definition}

\begin{proposition}
 \label{subdivision_of_totally_normal_css}
 A cellular subdivision of a totally normal cellular
 stratified space satisfying the conditions in Definition
 \ref{totally_normal_subdivision} is totally normal. 
\end{proposition}

\begin{remark}
 An analogous statement was stated as Proposition 2.45 in
 \cite{1312.7368} without the first condition of Definition
 \ref{totally_normal_subdivision}. It was pointed out by Priyavrat
 Deshpande that this condition was missing. He also pointed out typos
 and errors in the proof, that are corrected in the following proof.
\end{remark}

\begin{proof}
 Let $(X,\pi,\Phi)$ be a totally normal cellular stratified space and
 $s: P(X,\pi')\to P(X,\pi)$ be a cellular subdivision of $(X,\pi)$
 satisfying the conditions in Definition
 \ref{totally_normal_subdivision}. Its cell structure is denoted by
 $\Phi'$. 
 
 We need to verify that, for each cell $e_{\lambda'}$ in
 $(X,\pi',\Phi')$ 
 and a cell $e'$ in $\partial D_{\lambda'}$, there exists a cell
 $e_{\mu'}$ in $(X,\pi',\Phi')$ and a map $b': D_{\mu'}\to D_{\lambda'}$
 making the diagram
 \[
  \begin{diagram}
   \node{D_{\lambda'}} \arrow{e,t}{\varphi'_{\lambda'}}
   \node{\overline{e_{\lambda'}}} \arrow{e,J} \node{X} \\
   \node{D_{\mu'}} \arrow{n,l,..}{b'} \arrow{e,b}{\varphi'_{\mu'}}
   \node{\overline{e_{\mu'}}} \arrow{ne,J}
  \end{diagram}
 \]
 commutative and satisfying $b'(\Int(D_{\mu'}))=e'$.

 Let $\lambda=s(\lambda')$. By the definition of cellular subdivision,
 $D_{\lambda}$ has a structure of regular cellular stratified space
 $(D_{\lambda},\pi_{\lambda},\Phi_{\lambda})$ under which $e'$ is a
 cell. Under the identification
 \[
  P(\Int(D_{\lambda}),\pi_{\lambda}) \cong P(\bm{s})^{-1}(\lambda), 
 \]
 there exists a cell in $(\Int(D_{\lambda}),\pi_{\lambda})$
 corresponding to $e_{\lambda'}$.
 Let $s_{\lambda'}: D_{\lambda'}\to D_{\lambda}$ be its characteristic
 map. Then the characteristic map for $e_{\lambda'}$ in $(X,\pi',\Phi')$
 is given by the composition $\varphi_{\lambda}\circ s_{\lambda'}$.

 On the other hand, there exists a cell $e$ in
 $\partial D_{\lambda}$, before subdivision, containing
 $s_{\lambda'}(e')$ in its interior.
 By the total normality of $(X,\pi,\Phi)$, there exists a cell $e_{\mu}$
 in $(X,\pi,\Phi)$ and a morphism 
 $b: e_{\mu}\to e_{\lambda}$ in $C(X,\pi)$ with $b(\Int(D_{\mu}))=e$ and
 $\varphi_{\lambda}\circ b=\varphi_{\mu}$.
 Since $b$ is a strict morphism of stratified spaces and
 $s_{\lambda'}(e')$ is contained in the interior of $e$, there exists a
 unique cell $e''$ 
 in $(D_{\mu},\pi_{\mu},\Phi_{\mu})$ such that
 $e'=b(e'')$. $\varphi_{\mu}$ is also a strict
 morphism of stratified spaces and thus $\varphi_{\mu}(e'')$ is a cell
 in $(X,\pi',\Phi')$. Let us denote this cell by $e_{\mu'}$.
 Then the characteristic map for $e_{\mu'}$ is given
 by the composition $\varphi_{\mu}\circ s_{\mu'}$, where
 $s_{\mu'}:D_{\mu'}\to D_{\mu}$ is the characteristic map for $e''$. 

 When $\mu=\lambda$, both $e$ and $e'$ are cells in
 $P(D_{\lambda},\pi_{\lambda})$ and $b$ is the identity map. Hence the
 regularity of $(D_{\lambda},\pi_{\lambda},\Phi_{\lambda})$
 implies the existence of a unique map $b': D_{\mu'}\to D_{\lambda'}$
 satisfying the required conditions.

 Suppose $\mu<\lambda$. 
 The relation among these cells is depicted in the following diagram.
 \[
  \begin{diagram}
   \node{D_{\lambda'}} \arrow{ese,t}{s_{\lambda'}}
   \arrow[4]{e,t}{\varphi'_{\lambda'}} \node{} \node{} \node{} 
   \node{\overline{e_{\lambda'}}} \arrow{sw,J} \\
   \node{} \node{} \node{D_{\lambda}} \arrow{e,t}{\varphi_{\lambda}} 
   \node{\overline{e_{\lambda}}} \\
   \node{} \node{\partial D_{\lambda'}} \arrow{nnw,J}
   \arrow{e,t}{s_{\lambda'}} 
   \node{\partial D_{\lambda}} \arrow{n,J} \\ 
   \node{} \node{\overline{e'}} \arrow{n,J} \arrow{e,t}{s_{\lambda'}}
   \node{\overline{e}} \arrow{n,J} \\
   \node{} \node{\overline{e''}} \arrow{n,r}{\exists} \arrow{e,J}
   \node{D_{\mu}} \arrow{n,r}{b} 
   \arrow{e,t}{\varphi_{\mu}} \node{\overline{e_{\mu}}}
   \arrow[3]{n,J} \\ 
   \node{D_{\mu'}} \arrow[5]{n,l,..}{\exists b'?} \arrow{ne,t}{s_{\mu'}} 
   \arrow[4]{e,b}{\varphi'_{\mu'}} \node{} \node{} \node{} 
   \node{\overline{e_{\mu'}}.} \arrow{nw,J} \arrow[5]{n,J}
  \end{diagram}
 \]
 The image of the composition $b\circ s_{\mu'}$ is contained in 
 the image of $s_{\lambda'}$, i.e.\ $\overline{e}$. The regularity of
 $(D_{\lambda},\pi_{\lambda},\Phi_{\lambda})$ and the fact that $b$ is
 an embedding imply that the above diagram can be completed by a map
 $b'$. 
\end{proof}

\begin{remark}
 A subdivision $(\bm{s},\{(\pi_{\lambda},\Phi_{\lambda})\})$ of a
 totally normal cellular stratified space 
 $(X,\pi,\Phi)$ gives rise to a functor
 \[
 P(\bm{s}) : C(X,\pi) \rarrow{} \category{Posets}
 \]
 in a canonical way.
 On objects it is defined by
 $P(\bm{s})(\lambda)= P(D_{\lambda},\pi_{\lambda})$. The second
 condition in Definition \ref{totally_normal_subdivision} guarantees
 that this extends to a functor.

 In general, for a functor $F: C\to \category{Cats}$ from a small
 category to the category $\category{Cats}$ of small categories, there
 is a construction of a single category $\Gr(F)$, called the
 Grothendieck construction of $F$. A definition can be found in
 Thomason's paper \cite{Thomason79-1}, 
 for example.

 It is easy to verify that, in the case of a subdivision of a totally
 normal cellular stratified space, we have an isomorphism categories
 \[
  \Gr(P(\bm{s})) \cong C(X,\pi').
 \]
\end{remark}

The next step would be to study subidivisions of cylindrical and
polyhedral structures but it is not easy to describe appropriate
conditions on parameter spaces.

%% file: duality.tex
\section{Duality}
\label{duality}

In \cite{Salvetti87}, Salvetti first constructed
a simplicial complex
$\Sd(M(\mathcal{A}\otimes\bbC))=BC(M(\mathcal{A}\otimes\bbC))$
modelling the homotopy type of the complement of the complexification of a
real hyperplane arrangement $\mathcal{A}$ and then defined a structure
of regular cell complex by gluing simplices. This process reduces the
number of cells and allows us to relate the combinatorics of the 
arrangement $\mathcal{A}$ and the topology of Salvetti's model
$\Sd(M(\mathcal{A}\otimes\bbC))$. The process is closely related to the
classical concept of duality in PL topology. 

In this section, we introduce an analogous process for
polyhedral stellar stratified spaces. With stellar 
stratifications, we are able to extend both Salvetti's construction and
the duality in PL topology.

\input{cellular_structure_on_Sd}
\input{star_in_acyclic_category}
\input{duality_of_totally_normal_cellular_stratified_space}

\input{subdivision_of_category}

%% file: cellular_structure_on_Sd.tex
\subsection{A Canonical Cellular Stratification on the Barycentric
  Subdivision} 
\label{cellular_structure_on_Sd}

Let us first study stratifications on the barycentric subdivision
$\Sd(X)$ of a cylindrically normal stellar stratified space $X$.

In the case of totally normal stellar stratified spaces, Lemma
\ref{face_category_of_CNCSS} and Lemma
\ref{classifying_space_of_acyclic_category} imply
the following.

\begin{proposition}
 \label{Sd_of_totally_normal}
 For a totally normal stellar stratified space $X$, $\Sd(X)$ has a
 structure of regular cell complex.
\end{proposition}

Let us extend this structure
to a cellular stratification on $\Sd(X)$ for a cylindrically normal
stellar stratified space $X$. 

Given a cylindrically normal stellar stratified space $X$, we have an
acyclic topological category $C(X)$. By forgetting the topology, Lemma
\ref{classifying_space_of_acyclic_category} and Lemma
\ref{Delta-set_is_totally_normal} give us a map
\[
 \pi_{\Sd(X)} : \Sd(X) = BC(X) = \left\|\overline{N}(C(X))\right\|
 \rarrow{\pi_{\overline{N}(C(X))}} \coprod_{k}\overline{N}_k(C(X)).   
\]
As we will see in Lemma \ref{Sd_of_acyclic_category}, the set
$\coprod_{k}\overline{N}_k(C(X))$ can be identified with the set of
objects of the barycentric
subdivision of the acyclic category
$C(X)$ and has a structure of poset. The partial order is defined as
follows.

\begin{lemma}
 For $\bm{b}\in \overline{N}_k(C(X))$ and
 $\bm{b}'\in\overline{N}_{\ell}(C(X))$, regarded as functors
 $\bm{b}:[k]\to C(X)$ and $\bm{b}':[\ell]\to C(X)$,
 define $\bm{b}\le \bm{b}'$ if and only if there exists a morphism
 $\varphi:[k]\to[\ell]$ in $\Delta_{\inj}$ such that
 $\bm{b}'\circ \varphi=\bm{b}$. 
 Then the relation $\le$ is a partial order.
\end{lemma}

\begin{proof}
 See Lemma \ref{Sd_of_acyclic_category}.
\end{proof}

Let us verify that the map $\pi_{\Sd(X)}$ defines a cellular
stratification on $\Sd(X)$. 

\begin{proposition}
 \label{stratification_on_Sd}
 For a cylindrically normal stellar stratified space $X$, $\pi_{\Sd(X)}$
 defines a cellular stratification. 
\end{proposition}

\begin{proof}
 Let $X$ be a cylindrically normal stellar stratified space.
 For each nondegenerate $k$-chain $\bm{b}\in \overline{N}_k(C(X))$,
 $\pi_{\Sd(X)}^{-1}(\bm{b})$ is homeomorphic to $\Int(\Delta^k)$. The
 closure of $\pi_{\Sd(X)}^{-1}(\bm{b})$ in $\Sd(X)$ is an identification
 space of $\Delta^n$. Since the identification is defined only on the
 boundary, $\pi_{\Sd(X)}^{-1}(\bm{b})$ is open in its 
 closure and is locally closed. For nondegenerate chains $\bm{b}$ and
 $\bm{b}'$, 
 $\pi_{\Sd(X)}^{-1}(\bm{b})\subset\overline{\pi_{\Sd(X)}^{-1}(\bm{b}')}$
 if and only if $\pi_{\Sd(X)}^{-1}(\bm{b})$ is included in
 $\pi_{\Sd(X)}^{-1}(\bm{b}')$ as a face. In other words, this holds if
 and only if there exists a sequence of face operators mapping
 $\bm{b}'$ to $\bm{b}$, which is equivalent to saying
 $\bm{b}\le \bm{b}'$.
 Thus $\pi_{\Sd(X)}$ is a stratification
 in the sense of Definition \ref{stratification_definition}.

 Under the standard homeomorphism $\Delta^k\cong D^k$, we obtain a
 continuous map
 \[
  \varphi_{\bm{b}} : D^k \cong \{\bm{b}\}\times\Delta^k \hookrightarrow
 \coprod_{k} \overline{N}_k(C(X))\times\Delta^k \rarrow{}
 \|\overline{N}(C(X))\| \cong \Sd(X).
 \]
 By the compactness of $D^k$, $\varphi_{\bm{b}}$ is a quotient map onto
 its image, which is the closure of $\pi_{\Sd(X)}^{-1}(\bm{b})$. 
 The boundary of $\pi_{\Sd(X)}^{-1}(\bm{b})$ consists of points in
 $\Sd(X)$ of the form $[\bm{b},t]$ with $t\in\partial\Delta^k$. Under
 the defining relation, such a point is equivalent to a point in
 $\overline{N}_{\ell}(X)\times \Delta^{\ell}$ for $\ell<k$. Thus we
 obtain a cellular stratification.
\end{proof}

Note that this cellular stratification is rarely a CW stratification,
when parameter spaces have non-discrete topology.

\begin{example}
 \label{wired_sphere}
 Consider the minimal cell decomposition $\pi_n : S^n=e^0\cup e^n$ of the
 $n$-sphere. We have
 \begin{eqnarray*}
  \overline{N}_0(C(S^n)) & = & C(S^n)_0 = \{e^0,e^n\} \\
  \overline{N}_1(C(S^n)) & = & C(S^n)(e^0,e^n) =S^{n-1} \\
  \coprod_{k\ge 2} \overline{N}_k(C(S^n)) & = & \emptyset.
 \end{eqnarray*}

 Thus $\Sd(S^n;\pi_n)$ has only $0$-cells and $1$-cells. There are two
 $0$-cells $v(e_0)$ and $v(e^n)$ corresponding to $e^0$ and $e^n$. There
 are infinitely many 
 $1$-cells parametrized by the equator $S^{n-1}$. For $x\in S^{n-1}$,
 the $1$-cell $e^1_x$ corresponding to $x$ is given by the great half
 circle from the south pole to the north pole through $x$ under the
 identification $\Sd(S^n;\pi_n)\cong S^n$.
 \begin{figure}[ht]
 \begin{center}
  \begin{tikzpicture}
   \draw (0,0) circle (1);
   \draw [fill] (0,1) circle (2pt);
   \draw (0,1.3) node {$v(e^n)$};
   \draw [fill] (0,-1) circle (2pt);
   \draw (0,-1.3) node {$v(e^0)$};
   \draw (-1,0) arc (180:360:1cm and 0.5cm); 
   \draw [dotted] (1,0) arc (0:180:1cm and 0.5cm); 
   \draw (-2,0.7) node {$\Sd(S^n;\pi_n)$};

   \draw (0,-1) arc (270:450:0.5cm and 1cm);
   \draw (-0.3,-0.3) node {$x$};
   \draw (-0.2,0.3) node {$e^1_x$};
  \end{tikzpicture}
 \end{center}
  \caption{The canonical cellular stratification on $\Sd(S^n,\pi_n)$.}
 \end{figure}
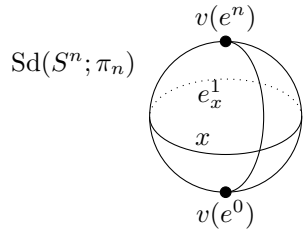

\end{example}


%% file: star_in_acyclic_category.tex
\subsection{Stars}
\label{star_in_acyclic_category}

As Example \ref{wired_sphere} shows, the cellular stratification on
$\Sd(X)$ defined in the previous section is not very useful. However,
Example \ref{wired_sphere} also suggests that by gluing cells in
$\Sd(X)$ together, we may construct a good cellular stratification.
In the classical PL topology, such a construction is called
star\footnote{See Definition \ref{star_definition} for the classical
definition.}. 

\begin{definition}
 Let $X$ be a stellar stratified space. For $x\in X$, define
 \[
  \St(x;X) = \bigcup_{x\in \overline{e_{\lambda}}} e_{\lambda}.
 \]
 For a subset $A\subset X$, define
 \[
  \St(A;X) = \bigcup_{x\in A} \St(x;X).
 \]
 The stratified subspace $\St(A;X)$ is called the \emph{star} of $A$ in
 $X$. 
\end{definition}

\begin{example}
 In the cellular stratification on $\Sd(S^n;\pi_n)$ in Example
 \ref{wired_sphere}, the star $\St(v(e^n);\Sd(S^n;\pi_n))$ coincides
 with $S^n\setminus\{v(e^0)\}$ and we recover the original cellular 
 stratification on $S^n$ as
 \[
  \Sd(S^n) = v(e^0) \cup \St(v(e^n);\Sd(S^n;\pi_n)).
 \]
 Note that we also have another cellular stratification
 \[
  \Sd(S^n) = v(e^n) \cup \St(v(e^0);\Sd(S^n;\pi_n)),
 \]
 which can be regarded as a dual to the original cellular
 stratification. 
\end{example}

We introduce analogous constructions for acyclic categories.

\begin{definition}
 \label{upper_and_lower_star}
 Let $C$ be an acyclic topological category and $x$ an object of
 $C$. The nondegenerate nerve $\overline{N}(x\downarrow C)$ of the
 comma category\footnote{$x\downarrow C$ is the category whose objects
 are morphisms $u$ in $C$ with $s(u)=x$. Morphisms are morphisms in $C$
 making the obvious triangle commutative.}
 $x\downarrow C$ is denoted by $\St_{\ge x}(C)$ and is called the
 \emph{upper star} of $x$ in $C$.

 The full subcategory of $x\downarrow C$ consisting of
 $(x\downarrow C)_{0}\setminus\{1_x\}$ is denoted by $C_{>x}$. 
 The nondegenerate nerve $\overline{N}(C_{>x})$ is 
 denoted by $\Lk_{>x}(C)$ and is called the \emph{upper link} of $x$
 in $C$. 

 The functor induced by the target map in $C$ is denoted by
 \[
  t_x : C_{>x}\subset x\downarrow C \longrightarrow C.
 \]
 The induced map of $\Delta$-spaces is also denoted by
 \[
 t_x : \Lk_{>x}(C) \subset \St_{\ge x}(C) \longrightarrow
 \overline{N}(C). 
 \]

 Dually the nondegenerate nerves of the comma category $C\downarrow x$
 and of its full subcategory $C_{<x}$ consisting of
 $(C\downarrow x)_0\setminus\{1_x\}$ are denoted by 
 $\St_{\le x}(C)$ and $\Lk_{<x}(C)$ and called the \emph{lower star} and
 \emph{lower link} of $x$ in $C$, respectively. We also have a functor
 and a map
 \begin{eqnarray*}
  s_x & : & C_{<x} \longrightarrow C \\
  s_x & : & \Lk_{<x}(C) \longrightarrow \overline{N}(C)
 \end{eqnarray*}
 induced by the source map.
\end{definition}

\begin{remark}
 The notation $C_{>x}$ and its definition is borrowed from Kozlov's book
 \cite{KozlovCombinatorialAlgebraicTopology}. Note that $\Lk_{>x}(C)$ is
 different from the usual link of $x$ in $\overline{N}(C)$ in general.
\end{remark}

We have the following description.

\begin{lemma}
 For an acyclic topological category $C$ and an object $x\in C_0$,
 we have the following identification
 \[
  \Lk_{>x}(C)_k \cong 
 \begin{cases}
  \coprod_{x\neq y} C(x,y), & k=0 \\
  \set{\bm{u}\in \overline{N}_{k+1}(C)}{s(u)=x}, & k>0,
 \end{cases} 
 \]
 under which the face operators
 \[
  d^{\Lk}_i : \Lk_{>x}(C)_k \longrightarrow \Lk_{>x}(C)_{k-1}
 \]
 are identified as $d^{\Lk}_{i} = d_{i+1}$, where $d_{i+1}$ is the
 $(i+1)$-st face operator in $\overline{N}(C)$. 
\end{lemma}

\begin{example}
 Consider the poset $[2] = \{0<1<2\}$ regarded as a category
 $0\to 1\to 2$.
 The category $[2]_{>1}$ has a unique object $1\to 2$ and no nontrivial
 morphism. Thus $\Lk_{>1}([2])$ is a single point. Under the map
 $t_1 : \Lk_{>1}([2]) \to \overline{N}([2])\cong \Delta^2$,
 $\Lk_{>1}([2])$ can be identified with the vertex $2$ in $\Delta^2$. On
 the other hand, the usual link of $1$ in $\Delta^2$ is the $1$-simplex
 spanned by vertices $0$ and $2$.
\end{example}

\begin{example}
 Consider the face category $C(S^1;\pi_1)$ of the minimal cell
 decomposition of $S^1$. $C(S^1;\pi_1)_{>e^0}$ consists of two objects
 $(C(S^1;\pi_1)_{>e^0})_0 = C(S^1;\pi_1)(e^0,e^1)=\{b_{-},b_{+}\}$ and
 no nontrivial morphisms. Thus 
 $\St_{\ge e^0}(C(S^1;\pi_1))$ is the cell complex
 $[-1,1]=\{-1\}\cup(-1,0)\cup\{0\}\cup(0,1)\cup\{1\}$ and
 $\Lk_{>e^0}(C(S^1;\pi_1))$ is $S^0$. The map $t_{e^0}$ maps the
 boundary $\partial [-1,1]=S^0$ to $v(e^2)$ in $\Sd(S^1)$ and defines a
 $1$-cell structure.
\end{example}

Note that, the comma category $x\downarrow C$ has an
initial object $1_{x}$.

\begin{lemma}
 \label{upper_star_as_cone}
 For an acyclic topological category $C$ and an object $x\in C_0$, we
 have a homeomorphism 
 \[
  \|\St_{\ge x}(C)\| \cong \{1_{x}\}\ast \|\Lk_{>x}(C)\|. 
 \]
\end{lemma}

\begin{proof}
 Define a map
 $h_x : \|\St_{\ge x}(C)\| \to \{1_{x}\}\ast \|\Lk_{>x}(C)\|$
 as follows. For
 $[(\bm{u},\bm{t})]\in \|\St_{\ge x}(C)\|=\|\overline{N}(x\downarrow C)\|$,
 choose a representative
 $(\bm{u},\bm{t})\in \overline{N}_k(x\downarrow C)\times\Delta^k$. 
 Here we regard $\bm{u}$ as a sequence of
 composable $k+1$ morphisms in $C$ starting from $x$;
 \[
  \bm{u} : x\rarrow{u_0} x_0\rarrow{u_1} x_1\rarrow{u_2} \cdots
 \rarrow{u_{k}} x_k 
 \]
 with $u_1$, $\ldots$, $u_k$ non-identity morphisms.
 When $u_0$ is not the identity morphism, 
 $\bm{u}$ belongs to $\Lk_{>x}(C)_k$ and $[(\bm{u},\bm{t})]$ can be
 regarded as an 
 element of $\|\Lk_{>x}(C)\|$. Define 
 \[
  h_x([\bm{u},\bm{t}]) = 01_x+1[(\bm{u},\bm{t})].
 \]
 When $u_0=1_x$, write $\bm{t}=t_00+(1-t_0)\bm{t}'$ under the
 identification $\Delta^k\cong \{0\}\ast \Delta^{k-1}$ and define
 \[
  h_x([\bm{u},\bm{t}]) = t_01_x + (1-t_0)[(\bm{u}',\bm{t}')],
 \]
 where 
 \[
  \bm{u}' : x\rarrow{u_1} x_1 \rarrow{u_2} \cdots \rarrow{u_k} x_k
 \]
 is the $(k-1)$-chain obtained from $\bm{u}$ by removing $u_0$. Since
 $\bm{u}$ is a nondegenerate chain, $u_1$ is not the identity morphism
 and $\bm{u}'\in \Lk_{>x}(C)_{k-1}$.

 Since the set of objects $C_0$ has the discrete topology, the
 decomposition
 \begin{eqnarray*}
   & & \overline{N}_k(x\downarrow C) \\
  & = & \left(\coprod_{x<x_1<\cdots<x_k} C(x,x)\times
 C(x,x_1)\times\cdots\times C(x_{k-1},x_k)\right) \\
  & & \amalg
 \left(\coprod_{x<x_0<\cdots<x_k} C(x,x_0)\times\cdots\times
  C(x_{k-1},x_k)\right) \\
  & = & \left(\{1_x\}\times\Lk_{>x}(x\downarrow C)\right) \amalg
   \Lk_{>x}(x\downarrow C)
 \end{eqnarray*}
 is a decomposition of topological spaces. The map $h_x$ is continuous
 on each component of $\overline{N}_k(x\downarrow C)\times\Delta^k$ and
 defines a continuous map
 \[
 h_x : \|\St_{\ge x}(C)\| \rarrow{} \{1_{x}\}\ast \|\Lk_{>x}(C)\|.
 \]
 It is easy to define an inverse to $h_x$, and thus $h_x$ is a
 homeomorphism. 
\end{proof}

%% file: duality_of_totally_normal_cellular_stratified_space.tex
\subsection{Canonical Stellar Stratifications on the Barycentric
  Subdivision} 
\label{duality_of_totally_normal_cellular_stratified_space}

Let us now go back to the discussion of stellar stratifications on
$\Sd(X)$. 

\begin{definition}
 Let $X$ be a cylindrically normal stellar stratified
 space. Define a map 
 \[
  \pi_{X^{\op}} : \Sd(X) \longrightarrow P(X) 
 \]
 by the composition
 \[
 \Sd(X) \rarrow{\pi_{\Sd(X)}} \coprod_{k}\overline{N}_k(C(X)) \rarrow{s} 
 C(X)_0 = P(X),
 \]
 where $s$ is the map induced by the source map in the category
 $C(X)$ and is also called the \emph{source map}.  
\end{definition}

By definition\footnote{Definition \ref{morphism_of_stratified_spaces}},
$\pi_{\Sd(X)}$ is a subdivision of $\pi_{X^{\op}}$ if 
$\pi_{X^{\op}}$ defines a stratification on $\Sd(X)$. Each inverse image
$\pi_{X^{\op}}^{-1}(\lambda)$ can be described by using
the upper stars.

\begin{definition}
 For a cell $e_{\lambda}$ in a cylindrically normal stellar stratified
 space $X$, define
 \[
 D_{\lambda}^{\op} = \left\|\St_{\ge e_{\lambda}}(C(X))\right\|
 \]
 We also denote
 \[
 D_{\lambda}^{\op,\circ} = \left\|\St_{\ge
 e_{\lambda}}(C(X))\right\| \setminus
 \left\|\Lk_{>e_{\lambda}}(C(X))\right\|,
 \]
 where $\|\Lk_{>e_{\lambda}}(C(X))\|$ is regarded as the bottom subspace
 of the cone under the identification
 \[
 \left\|\St_{\ge e_{\lambda}}(C(X))\right\| \cong
 \{1_{e_{\lambda}}\}\ast\|\Lk_{>e_{\lambda}}(C(X))\| 
 \]
 in Lemma \ref{upper_star_as_cone}.
\end{definition}

\begin{lemma}
 \label{stratum_in_D(X)}
 Let $X$ be a cylindrically normal stellar stratified space. In the
 stratification $\pi_{X^{\op}}$, each stratum is given by
 the image of $D_{\lambda}^{\op,\circ}$ under the map
 $t_{\lambda}=t_{e_{\lambda}}:\|\St_{\ge e_{\lambda}}(C(X))\|\to\Sd(X)$
 defined in Definition \ref{upper_and_lower_star}.  Namely
 \[
 \pi_{X^{\op}}^{-1}(\lambda) =
 t_{\lambda}\left(D_{\lambda}^{\op,\circ}\right).  
 \]  
 Hence the closure of each stratum is given by 
 \[
 \overline{\pi_{X^{\op}}^{-1}(\lambda)} =
 t_{\lambda}\left(D_{\lambda}^{\op}\right). 
 \]
\end{lemma}

\begin{proof}
 Elements in $D_{\lambda}^{\op,\circ}$ are those which are represented
 by $(\bm{u},t_00+(1-t_0)\bm{t}')$, where $t_0<1$ and $\bm{u}$ begins
 with the identity morphism $1_{e_{\lambda}}$. Therefore
 $s(t_{\lambda}(\bm{u}))=e_{\lambda}$ and
 $t_{\lambda}\left(D_{\lambda}^{\op,\circ}\right)\subset\pi_{X^{\op}}^{-1}(\lambda)$.
 Conversely, by choosing a representative $(\bm{u},\bm{t})$ of
 $[(\bm{u},\bm{t})]\in \pi_{X^{\op}}^{-1}(\lambda)$ with
 $\bm{t}\in\Int(\Delta^k)$, we see
 $\pi_{X^{\op}}^{-1}(\lambda)\subset t_{\lambda}\left(D_{\lambda}^{\op,\circ}\right)$. 

 Let $p : \coprod_k \overline{N}_k(C(X))\times \Delta^k \to \Sd(X)$ be
 the projection. Then the topology on $\Sd(X)$ is the weak topology
 defined by the covering
 \[
 \left\{p(C(X)(e_{\lambda_{k-1}},e_{\lambda_k})\times\cdots\times
 C(X)(e_{\lambda_0},e_{\lambda_1})\times\Delta^k)\right\}   
 \]
 Thus the closure of
 \begin{multline*}
 t_{\lambda}\left(D_{\lambda}^{\op,\circ}\right) = \\
 p\left(\coprod_k\coprod_{\lambda<\lambda_0<\cdots<\lambda_k} 
 C(X)(e_{\lambda_{k-1}},e_{\lambda_{k}})\times\cdots\times
 C(X)(e_{\lambda},e_{\lambda_{0}})\times
 (\Delta^k \setminus d^0(\Delta^{k-1}))\right)
 \end{multline*}
 is given by adding
 $C(X)(e_{\lambda},e_{\lambda_1})\times\cdots\times
 C(X)(e_{\lambda_{k-1}},e_{\lambda_k})\times d^0(\Delta^{k-1})$. And we
 have 
 \[
  \overline{\pi_{X^{\op}}^{-1}(\lambda)} =
 \overline{t_{\lambda}\left(D_{\lambda}^{\op,\circ}\right)} =
 t_{\lambda}\left(D_{\lambda}^{\op}\right). 
 \]
\end{proof}

\begin{corollary}
 For a cylindrically normal stellar stratified space $X$,
 $\pi_{X^{\op}}$ is a stratification whose face poset is the opposite
 $P(X)^{\op}$ of $P(X)$.  
\end{corollary}

\begin{proof}
 The fact that $\pi_{X^{\op}}^{-1}(\lambda)$ is 
 locally closed for each $\lambda\in P(X)$ follows from the
 description in Lemma \ref{stratum_in_D(X)}.
 It also says that $\pi_{X^{\op}}^{-1}(\lambda)$ is connected, since it 
 contains the barycenter $v(e_{\lambda})$ of $e_{\lambda}$.

 The compatibility with the partial order in
 $P(X)^{\op}$ also follows from the description of the boundary in Lemma
 \ref{stratum_in_D(X)}. 
\end{proof}

\begin{definition}
 The stratification $\pi_{X^{\op}}$ is called the \emph{stellar dual} of
 $\pi_{X}$. 
\end{definition}

Thus when $\|\Lk_{>e_{\lambda}}(C(X))\|$ can be embedded in a sphere
$S^{N-1}$ in such a way that the closure
$\overline{\|\Lk_{>e_{\lambda}}(C(X))\|}$ is a finite cell complex
containing $\|\Lk_{>e_{\lambda}}(C(X))\|$ as a strict cellular
stratified subspace,
$\|\St_{\ge e_{\lambda}}(C(X))\|$ can be
regarded as an aster in $D^N$. 

\begin{proposition}
 \label{stellar_structure_on_Sd}
 Let $X$ be a finite polyhedral relatively compact cellular
 stratified space. Then
 $\Sd(X)$ has a structure of stellar stratified space whose underlying
 stratification is $\pi_{X^{\op}}$ and the face poset is
 $P(\Sd(X),\pi_{X^{\op}}) = P(X,\pi_{X})^{\op}$.
\end{proposition}

\begin{proof}
 For $\lambda\in P(X)$, consider the upper star
 $\St_{\ge e_{\lambda}}(C(X))$ and the upper link
 $\Lk_{>e_{\lambda}}(C(X))$ of $\lambda$ in $C(X)$. 
 Since compact locally cone-like spaces can be expressed as a union of
 a finite number of simplices\footnote{Theorem 2.11 in
 \cite{Rourke-SandersonBook}}, each parameter space has a structure of a 
 finite polyhedral complex. And the comma category
 $e_{\lambda}\downarrow C(X)$  is a cellular category whose morphism 
 spaces are finite cell complexes. By Lemma
 \ref{cellular_stratification_on_BC} and the finiteness assumption,  
 $\|\St_{\ge e_{\lambda}}(C(X))\|=B(e_{\lambda}\downarrow C(X))$ is a
 finite cell complex and 
 $\|\Lk_{>e_{\lambda}}(C(X))\|$ is a subcomplex. Choose an 
 embedding 
 $\|\Lk_{>e_{\lambda}}(C(X))\|\hookrightarrow S^{N-1}$. Then
 $D_{\lambda}^{\op}=\{1_{e_{\lambda}}\}\ast\|\Lk_{>e_{\lambda}}(C(X))\|$
 is embedded in $D^N$.

 By definition,
 \[
  t_{e_{\lambda}} : D_{\lambda}^{\op} \longrightarrow
 \overline{\pi_{X^{\op}}^{-1}(\lambda)} \subset \Sd(X)
 \]
 is a quotient map. The fact that $t_{e_{\lambda}}$ is a homeomorphism
 onto $\pi_{X^{\op}}^{-1}(\lambda)$ when restricted to
 $\Int\left(D_{\lambda}^{\op}\right)=D_{\lambda}^{\op,\circ}$ follows
 easily from the description of elements in $D_{\lambda}^{\op,\circ}$
 in the proof of Lemma \ref{stratum_in_D(X)}.
\end{proof}

\begin{remark}
 The three assumptions, i.e.\ local-polyhedrality, finiteness, and
 relative-compactness, are imposed only for the purpose of
 the existence of an embedding of $\|\Lk_{>e_{\lambda}}(C(X))\|$ in a
 sphere. If we relax the definition of a stellar cell
 $\varphi_{\lambda} : D_{\lambda}\to \overline{e_{\lambda}}$ by dropping
 the embeddability of the domain $D_{\lambda}$ in a disk, we do not need
 to require these conditions.
\end{remark}

The next problem is to define a cylindrical structure for the stellar
dual $(\Sd(X),\pi_{X^{\op}})$.

\begin{theorem}
 \label{stellar_stratification_on_D(X)}
 Let $X$ be a finite polyhedral stellar stratified
 space. Suppose all parameter spaces $P_{\mu,\lambda}$ are 
 compact. For
 $\lambda \le^{\op}\mu$ in $P(\Sd(X),\pi_{X^{\op}})= P(X,\pi_X)^{\op}$,
 define $P^{\op}_{\lambda,\mu}= P_{\mu,\lambda}$.
 Then the stellar structure in Proposition \ref{stellar_structure_on_Sd}
 and parameter spaces 
 $\{P_{\lambda,\mu}^{\op}\}$ make $(\Sd(X),\pi_{X^{\op}})$ 
 into a polyhedral stellar stratified space.
\end{theorem}

\begin{proof}
 It remain to construct PL structure maps
 \begin{eqnarray*}
  b^{\op}_{\lambda,\mu} & : & P^{\op}_{\lambda,\mu}\times
   D^{\op}_{\lambda} \longrightarrow D^{\op}_{\mu} \\ 
  \circ^{\op} & : & P^{\op}_{\lambda_1,\lambda_0}\times
   P^{\op}_{\lambda_2,\lambda_1} \longrightarrow
   P^{\op}_{\lambda_2,\lambda_0}.  
 \end{eqnarray*}
 for $\lambda\le^{\op}\mu$ and
 $\lambda_{2}\le^{\op}\lambda_{1}\le^{\op}\lambda_{0}^{\op}$.

 The composition map $\circ^{\op}$ is obviously given by the composition
 in $X$ under the identification 
 \[
 P^{\op}_{\lambda_1,\lambda_0}\times P^{\op}_{\lambda_2,\lambda_1} =
 P_{\lambda_0,\lambda_1}\times P_{\lambda_1,\lambda_2} \cong  
 P_{\lambda_1,\lambda_2}\times P_{\lambda_0,\lambda_1}.
 \]
 The action $b^{\op}_{\lambda,\mu}$ of the parameter space is given by
 the following composition
 \begin{eqnarray*}
  P^{\op}_{\lambda,\mu}\times D_{\lambda}^{\op} & = & P_{\mu,\lambda}\times
   B(e_{\lambda}\downarrow C(X)) \\
  & = & C(X)(e_{\mu},e_{\lambda}) \times B(e_{\lambda}\downarrow
   C(X)) \\
  & \cong & B(e_{\lambda}\downarrow C(X))\times
   C(X)(e_{\mu},e_{\lambda}) \\ 
  & \longrightarrow & B(e_{\mu}\downarrow C(X)) = D_{\mu}^{\op},
 \end{eqnarray*}
 where the last arrow is given by the composition in $C(X)$.

 The compatibilities of $b_{\mu,\lambda}$ and $\circ$ in $X$ implies
 that $b_{\lambda,\mu}^{\op}$ and $\circ^{\op}$ satisfy the conditions
 for cylindrical structure. Obviously this is polyhedral.
\end{proof}

\begin{definition}
 The stratified space $(\Sd(X),\pi_{X^{\op}})$ equipped with the stellar
 and polyhedral structures defined in Theorem
 \ref{stellar_stratification_on_D(X)} is
 called the \emph{stellar dual of $X$} and is denoted by $X^{\op}$.
\end{definition}

The following fundamental example shows that cells in $X^{\op}$ are usually
stellar and not cellular.

\begin{example}
 Consider the standard regular cellular stratification on $\Delta^2$ as
 a simplicial complex. The barycentric subdivision $\Sd(\Delta^2)$ is a
 simplicial complex depicted in Figure \ref{Sd(Delta^2)}
 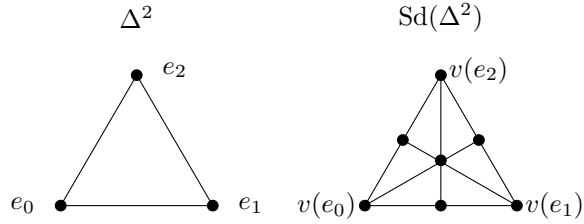
\begin{figure}[ht]
 \begin{center}
  \begin{tikzpicture}
   \draw (1,2.5) node {$\Delta^2$};
   \draw (0,0) -- (1,1.73) -- (2,0) -- (0,0);
   \draw [fill] (0,0) circle (2pt);
   \draw (-0.5,0) node {$e_0$};
   \draw [fill] (2,0) circle (2pt);
   \draw (2.5,0) node {$e_1$};
   \draw [fill] (1,1.73) circle (2pt);
   \draw (1.5,1.8) node {$e_2$};

   \draw (5,2.5) node {$\Sd(\Delta^2)$};
   \draw (4,0) -- (5,1.73) -- (6,0) -- (4,0);
   \draw [fill] (4,0) circle (2pt);
   \draw (3.5,0) node {$v(e_0)$};
   \draw [fill] (6,0) circle (2pt);
   \draw (6.5,0) node {$v(e_1)$};
   \draw [fill] (5,1.73) circle (2pt);
   \draw (5.5,1.8) node {$v(e_2)$};
   \draw (4.5,0.87) -- (6,0);
   \draw (5.5,0.87) -- (4,0);
   \draw (5,1.73) -- (5,0);
   \draw [fill] (4.5,0.87) circle (2pt);
   \draw [fill] (5.5,0.87) circle (2pt);
   \draw [fill] (5,0) circle (2pt);
   \draw [fill] (5,0.6) circle (2pt);   
  \end{tikzpicture}
 \end{center}
  \caption{$\Delta^2$ and its barycentric subdivision.}
  \label{Sd(Delta^2)}
 \end{figure}

 The stellar stratification on $\Sd(\Delta^2)$, i.e.\ $(\Delta^2)^{\op}$
 is given by Figure \ref{stellar_structure_on_Delta^op}.
 \begin{figure}[ht]
 \begin{center}
  \begin{tikzpicture}
   \draw (0,2.5) node {$D^{\op}_{0}$};
   \draw [dotted] (0,1) circle (1);
   \draw (1,1) arc (0:90:1);
   \draw (0,1) -- (1,1);
   \draw (0,1) -- (0,2);
   \draw [fill] (0.71,1.71) circle (2pt);

   \draw (2.9,1.5) node {$t_{0}$};
   \draw [->] (1.5,1.2) -- (3.5,0.7);

   \draw (5,2.5) node {$(\Delta^2)^{\op}$};
   \draw (4,0) -- (5,1.73) -- (6,0) -- (4,0);
   \draw (4.5,0.87) -- (5,0.6);
   \draw (5.5,0.87) -- (5,0.6);
   \draw (5,0.6) -- (5,0);
   \draw [fill] (5,0.6) circle (2pt);   
  \end{tikzpicture}
 \end{center}
  \caption{The stellar structures on $(\Delta^2)^{\op}$.}
  \label{stellar_structure_on_Delta^op}
 \end{figure}
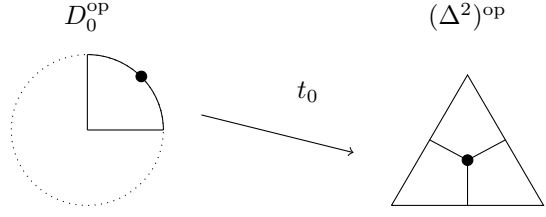
 $(\Delta^2)^{\op}$ consists of one stellar $0$-cell, three stellar
 $1$-cells, and three stellar $2$-cells. The stellar structure for 
 the $2$-cell at the left bottom in $D(\Delta^2)$ is indicated in 
 Figure \ref{stellar_structure_on_Delta^op} as $t_0$.

 Note that $\|\Lk_{>e_0}(C(\Delta^2))\|$ consists of two
 $1$-simplices. It is embedded in $S^1$. The domain for this stellar
 structure map is the circular sector in the left.
 The middle point in the arc of the circular sector in
 $D_0^{\op}$ is mapped to the barycenter of $\Delta^2$ and two
 radii in $D_0^{\op}$ are mapped to the half edges touching the
 lower left vertex of $\Delta^2$.

 The barycentric subdivision of $(\Delta^2)^{\op}$ can be easily seen to
 be isomorphic to $\Sd(\Delta^2)$ as simplicial complexes. And we have
 $((\Delta^2)^{\op})^{\op}\cong \Delta^2$ as simplicial complexes.
 \begin{figure}[ht]
 \begin{center}
  \begin{tikzpicture}
   \draw (1,2.5) node {$\Sd((\Delta^2)^{\op})$};
   \draw (0,0) -- (1,1.73) -- (2,0) -- (0,0);
   \draw [fill] (0,0) circle (2pt);
   \draw [fill] (2,0) circle (2pt);
   \draw [fill] (1,1.73) circle (2pt);
   \draw (0.5,0.87) -- (2,0);
   \draw (1.5,0.87) -- (0,0);
   \draw (1,1.73) -- (1,0);
   \draw [fill] (0.5,0.87) circle (2pt);
   \draw [fill] (1.5,0.87) circle (2pt);
   \draw [fill] (1,0) circle (2pt);
   \draw [fill] (1,0.6) circle (2pt);   

   \draw (5,2.5) node {$((\Delta^2)^{\op})^{\op}$};
   \draw (4,0) -- (5,1.73) -- (6,0) -- (4,0);
   \draw [fill] (4,0) circle (2pt);
   \draw [fill] (6,0) circle (2pt);
   \draw [fill] (5,1.73) circle (2pt);
  \end{tikzpicture}
 \end{center}
  \caption{The dual of the dual of $\Delta^2$.}
 \end{figure}
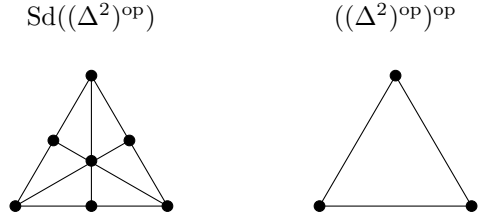
 For example, the $2$-cell corresponding to the unique $0$-cell in
 $(\Delta^2)^{\op}$ is the whole triangle.
\end{example}

The next example shows that, when $X$ contains non-closed cells, the
process of taking the double dual $((X)^{\op})^{\op}$ slims down
$X$ while retaining the stratification. 

\begin{example}
 Consider the $1$-dimensional stellar stratified space $Y$ in Figure
 \ref{circle_with_tail}. 
 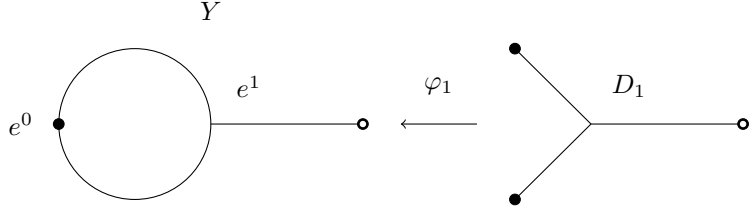
\begin{figure}[ht]
 \begin{center}
  \begin{tikzpicture}
   \draw (-1.5,0) node {$e^0$};
   \draw (0,0) circle (1cm);
   \draw [fill] (-1,0) circle (2pt);
   \draw (1.5,0.5) node {$e^1$};
   \draw (1,0) -- (3,0);
   \draw [fill] (3,0) circle (2pt);
   \draw [fill,white] (3,0) circle (1pt);
   \draw (1,1.5) node {$Y$};

   \draw [<-] (3.5,0) -- (4.5,0);
   \draw (4,0.5) node {$\varphi_1$};

   \draw (5,1) -- (6,0);
   \draw [fill] (5,1) circle (2pt);
   \draw (5,-1) -- (6,0);
   \draw [fill] (5,-1) circle (2pt);
   \draw (6,0) -- (8,0);
   \draw [fill] (8,0) circle (2pt);
   \draw [fill,white] (8,0) circle (1pt);
   \draw (6.5,0.5) node {$D_1$};
  \end{tikzpicture}
 \end{center}
  \caption{A $1$-dimensional stellar stratified space $Y$.}
  \label{circle_with_tail}
 \end{figure}
 It consists of a $0$-cell $e^0$ and a stellar $1$-cell $e^1$ whose
 domain $D_1$ is a graph of the shape of $Y$ with one vertex
 removed. The barycentric subdivision $\Sd(Y)$ is the minimal regular
 cell decomposition of $S^1$, as is shown in Figure
 \ref{dual_of_Y}. Both 
 $Y^{\op}$ and $(Y^{\op})^{\op}$ are the minimal 
 cell decomposition of $S^1$. 
 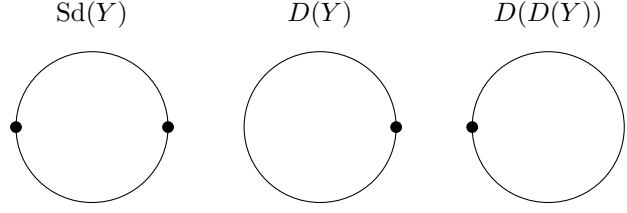
\begin{figure}[ht]
 \begin{center}
 \begin{tikzpicture}
   \draw (1,0) circle (1cm);
   \draw [fill] (0,0) circle (2pt);
   \draw [fill] (2,0) circle (2pt);
   \draw (1,1.5) node {$\Sd(Y)$};

   \draw (4,0) circle (1cm);
   \draw [fill] (5,0) circle (2pt);
   \draw (4,1.5) node {$D(Y)$};
  
   \draw (7,0) circle (1cm);
   \draw [fill] (6,0) circle (2pt);
   \draw (7,1.5) node {$D(D(Y))$};
 \end{tikzpicture}
 \end{center}
  \caption{The stellar dual of $Y$.}
  \label{dual_of_Y}
 \end{figure}
 Note that the embedding $i_{Y}$ in Theorem
 \ref{embedding_Sd} embeds $(Y^{\op})^{\op}$ in $Y$ as a stellar
 stratified space. 
\end{example}

As the above example suggests, the embedding $i_{X}$ in Theorem
\ref{embedding_Sd} is an embedding of stellar stratified
spaces if the domain is regarded as $(X^{\op})^{\op}$. Furthermore, when
all cells are closed, we can always recover $X$ from this stellar 
structure.

Note that we may define this stellar stratification on
$\Sd(X)$ directly without using $(-)^{\op}$.

\begin{definition}
 Let $X$ be a cylindrically normal stellar stratified space. For
 $\lambda\in P(X)$, define 
 \begin{eqnarray*}
  D_{\lambda}^{\Sal(X)} & = & \|\St_{\le e_{\lambda}}(C(X))\| \\  
  D_{\lambda}^{\Sal(X),\circ} & = & \|\St_{\le e_{\lambda}}(C(X))\|
   \setminus \|\Lk_{<e_{\lambda}}(C(X))\|. 
 \end{eqnarray*}
\end{definition}

The following fact is dual to Lemma \ref{stratum_in_D(X)}. The proof is
omitted. 

\begin{lemma}
 \label{stellar_structure_for_Sal(X)}
 Let $X$ be a cylindrically normal stellar stratified space. Define a map
 \[
  \pi_{\Sal(X)} : \Sd(X) \longrightarrow P(X)
 \]
 by the composition 
 \[
  \Sd(X) \rarrow{\pi_{\Sd(X)}} \coprod_{k} \overline{N}_k(C(X))
 \rarrow{t} C(X)_0 = P(X),
 \]
 where $t$ is the target map. Then $\pi_{\Sal(X)}$ is a stratification
 whose strata and their closures are given by
 \begin{eqnarray*}
  \pi_{\Sal(X)}^{-1}(\lambda) & = &
   s_{\lambda}\left(D_{\lambda}^{\Sal(X),\circ}\right) \\ 
  \overline{\pi_{\Sal(X)}^{-1}(\lambda)} & = &  
   s_{\lambda}\left(D_{\lambda}^{\Sal(X)}\right),
 \end{eqnarray*}
 where $s_{\lambda}=s_{e_{\lambda}}$ is the map defined in Definition
 \ref{upper_and_lower_star}. 
\end{lemma}

\begin{definition}
 The stellar stratified space $(\Sd(X),\pi_{\Sal})$ defined above is
 called the \emph{Salvetti complex} of $X$ and is denoted by $\Sal(X)$.
\end{definition}

\begin{remark}
 When the three assumptions in Proposition \ref{stellar_structure_on_Sd}
 are satisfied, we have 
 ${\Sal(X)}=(X^{\op})^{\op}$ as stellar stratified spaces.
\end{remark}

\begin{theorem}
 \label{Salvetti}
 Let $(X,\pi_X)$ be a cylindrically normal stellar
 stratified space. Then $\Sal(X)$ has a structure of cylindrically
 normal stellar stratified space.
 When $X$ is relatively compact, the embedding
 $i_{X} : \Sal(X)\hookrightarrow X$ is an embedding of 
 cylindrically normal stellar stratified spaces. 
 When all cells in $X$ are closed, $i_{X}$ is an isomorphism
 of cylindrically normal stellar stratified spaces. When $X$ is a finite
 polyhedral stellar stratified space satisfying the assumptions
 of Proposition \ref{stellar_structure_on_Sd}, we have
 $\Sal(X)=(X^{\op})^{\op}$ 
 as stellar stratified spaces. 
\end{theorem}

\begin{definition}
 When we regard $\Sd(X)$ as $\Sal(X)$, the embedding
 $i_{X}$ is denoted by $i_{\Sal(X)} : \Sal(X)\hookrightarrow X$.
\end{definition}

In order to prove Theorem \ref{Salvetti}, we use the following
reformulation of the construction of $i_{X}$.

\begin{lemma}
 \label{embedding_stellar_cells}
 Let $X$ be a CW cylindrically normal stellar stratified space $X$. For each
 stellar cell $e_{\lambda}$, there exists an embedding
 \[
  z_{\lambda} : D_{\lambda}^{\Sal(X)} \rarrow{} D_{\lambda}
 \]
 making the diagram
 \begin{equation}
  \begin{diagram}
   \node{P_{\mu,\lambda}\times D_{\mu}^{\Sal(X)}}
   \arrow{s,l}{1\times z_{\mu}} \arrow{e,t}{b^{\Sal(X)}_{\mu,\lambda}}   
   \node{D_{\lambda}^{\Sal(X)}} \arrow{s,r}{z_{\lambda}} \\
   \node{P_{\mu,\lambda}\times D_{\mu}} \arrow{e,b}{b_{\mu,\lambda}}
   \node{D_{\lambda}} 
  \end{diagram}
  \label{z_respects_cylindrical_structure}
 \end{equation}
 commutative, where $b^{\Sal(X)}_{\mu,\lambda}$ is defined by the
 composition 
 \begin{eqnarray*}
  P_{\mu,\lambda}\times D_{\mu}^{\Sal(X)} & = & P_{\mu,\lambda}\times
   B(C(X)\downarrow e_{\mu}) \\
  & = & C(X)(e_{\mu},e_{\lambda})\times B(C(X)\downarrow e_{\mu}) \\
  & \longrightarrow & B(C(X)\downarrow e_{\lambda}).
 \end{eqnarray*}
 Here the last map is induced by the composition in $C(X)$.

 These maps $\{z_{\lambda}\}$ can be glued together to define an
 embedding $\Sal(X)\hookrightarrow X$, which can be identified with
 $i_{X}$ in Theorem \ref{embedding_Sd}
\end{lemma}

\begin{proof}
 Note that we have the following decomposition
 \[
  \overline{N}_{k}(C(X)) = \coprod_{\lambda\in P(X)}
 \overline{N}_{k}(C(X)\downarrow e_{\lambda}).
 \]
 The restrictions of the maps
 $z_{k}:\overline{N}_{k}(C(X))\times\Delta^k\to D(X)$ to
 $\overline{N}_{k}(C(X)\downarrow e_{\lambda})$ define a map
 \[
 z_{\lambda}: \|\overline{N}(C(X)\downarrow e_{\lambda})\| =
 D_{\lambda}^{\op} \rarrow{} D_{\lambda}.
 \] 
 
 The fact that these maps make the diagram
 (\ref{z_respects_cylindrical_structure}) commutative can be verified by
 investigating the construction of $z_{k}$ and is omitted. 
\end{proof}

\begin{proof}[Proof of Theorem \ref{Salvetti}]
 Lemma \ref{embedding_stellar_cells} says that $i_{X}$ is an embeddding
 of stellar stratified spaces.
 
 When all cells in $X$ are closed, $i_{X} : \Sd(X)\to X$ is a
 homeomorphism. The above argument 
 implies that this map defines an 
 isomorphism $i_{\Sal(X)} : \Sal(X) \rarrow{\cong} X$ of cylindrically
 normal stellar stratified spaces. 

 Suppose that $X$ satisfies the assumption of Proposition
 \ref{stellar_structure_on_Sd}. By Theorem
 \ref{stellar_stratification_on_D(X)}, we have an isomorphism of
 categories $C(X^{\op})\cong C(X)^{\op}$. Thus 
 \[
 D_{\lambda}^{(X^{\op})^{\op}}=B(e_{\lambda}\downarrow C(X^{\op})) \cong
 B(e_{\lambda}\downarrow C(X)^{\op}) \cong
 B(C(X) \downarrow e_{\lambda}) 
 \]
 and the stellar structure map
 $t_{\lambda} : D_{\lambda}^{(X^{\op})^{\op}} \to \Sd(X^{\op})$
 can be identified with 
 $s_{\lambda} : D_{\lambda}^{(\Sal(X))} \to \Sd(X)$
 under the identification
 $\Sd(D(X)) \cong B(C(X)^{\op})\cong B(C(X))=\Sd(X)$.
 And we have an identification $D(D(X))\cong \Sal(X)$.
\end{proof}

The following example justifies the name for $\Sal(X)$.

\begin{example}
 For a real hyperplane arrangement $\mathcal{A}$, the structure of
 regular cell complex on the Salvetti complex $\Sal(\mathcal{A})$ for
 the complexification of $\mathcal{A}$ defined in
 \cite{Salvetti87} is nothing but $D(D(M(\mathcal{A}\otimes\bbC)))$. For
 example, in the case of the arrangement $\mathcal{A}=\{\{0\}\}$ in
 $\R$, the stratification on $\R$ is
 \[
 \R = \{0\} \cup (-\infty,0)\cup (0,\infty)
 \]
 and the associated stratification on the complexification is 
 \[
 \bbC = \{(0,0)\}\cup (-\infty,0)\times\{0\} \cup (0,\infty)\times\{0\} 
 \cup \set{x+iy\in\bbC}{y>0} \cup \set{x+iy\in\bbC}{y<0}.
 \]
 Then $\Sd(M(\mathcal{A}\otimes\bbC))$,
 $M(\mathcal{A}\otimes\bbC)^{\op}$, and 
 $(M(\mathcal{A}\otimes\bbC)^{\op})^{\op}$ are given in Figure \ref{}.
 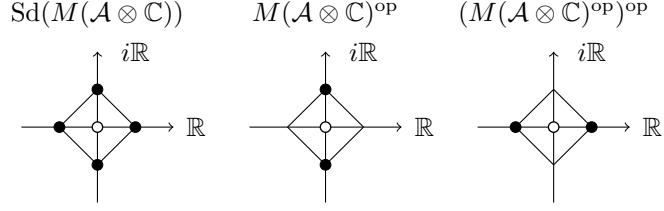
\begin{figure}[ht]
 \begin{center}
  \begin{tikzpicture}
   \draw (0,1.5) node {$\Sd(M(\mathcal{A}\otimes\bbC))$};
   \draw [->] (-1,0) -- (1,0);
   \draw (1.3,0) node {$\R$};
   \draw [->] (0,-1) -- (0,1);
   \draw (0.5,1) node {$i\R$};
   \draw [fill] (0,0) circle (2pt);
   \draw [white,fill] (0,0) circle (1.5pt);
   \draw [fill] (0.5,0) circle (2pt);
   \draw [fill] (0,0.5) circle (2pt);
   \draw [fill] (-0.5,0) circle (2pt);
   \draw [fill] (0,-0.5) circle (2pt);
   \draw (0.5,0) -- (0,0.5);
   \draw (0,0.5) -- (-0.5,0);
   \draw (-0.5,0) -- (0,-0.5);
   \draw (0,-0.5) -- (0.5,0);

   \draw (3,1.5) node {$M(\mathcal{A}\otimes\bbC)^{\op}$};
   \draw [->] (2,0) -- (4,0);
   \draw (4.3,0) node {$\R$};
   \draw [->] (3,-1) -- (3,1);
   \draw (3.5,1) node {$i\R$};
   \draw [fill] (3,0) circle (2pt);
   \draw [white,fill] (3,0) circle (1.5pt);
   \draw [fill] (3,0.5) circle (2pt);
   \draw [fill] (3,-0.5) circle (2pt);
   \draw (3.5,0) -- (3,0.5);
   \draw (3,0.5) -- (2.5,0);
   \draw (2.5,0) -- (3,-0.5);
   \draw (3,-0.5) -- (3.5,0);

   \draw (6,1.5) node {$(M(\mathcal{A}\otimes\bbC)^{\op})^{\op}$};
   \draw [->] (5,0) -- (7,0);
   \draw (7.3,0) node {$\R$};
   \draw [->] (6,-1) -- (6,1);
   \draw (6.5,1) node {$i\R$};
   \draw [fill] (6,0) circle (2pt);
   \draw [white,fill] (6,0) circle (1.5pt);
   \draw [fill] (6.5,0) circle (2pt);
   \draw [fill] (5.5,0) circle (2pt);
   \draw (6.5,0) -- (6,0.5);
   \draw (6,0.5) -- (5.5,0);
   \draw (5.5,0) -- (6,-0.5);
   \draw (6,-0.5) -- (6.5,0);
  \end{tikzpicture}
 \end{center}
  \caption{Stellar duals for $M(\mathcal{A}\otimes\bbC)$}  
 \end{figure}
\end{example}

%% file: subdivision_of_category.tex
\subsection{The Barycentric Subdivision of Face Categories}
\label{subdivision_of_category}

We conclude this paper by proving that the barycentric subdivision
of a totally normal cellular stratified space corresponds to the
barycentric subdivision of the face category.

Let us first recall the definition of the barycentric subdivision of a
small category. We use a definition in Noguchi's papers
\cite{1004.2547,1104.3630}. See also the paper \cite{0707.1718} by del
Hoyo. 

\begin{definition}
 \label{barycentric_subdivision_of_category}
 For a small category $C$, the \emph{barycentric subdivision} $\Sd(C)$
 is a small category defined by
 \begin{eqnarray*}
  \Sd(C)_0 & = & \coprod_{n} \overline{N}_{n}(C), \\
  \Sd(C)(f,g) & = & \set{\varphi : [m]\to [n]}{g\circ\varphi =
   f}/_{\sim} 
 \end{eqnarray*}
 for $f : [m] \to C$ and $g : [n] \to C$, where $\sim$ is the
 equivalence relation generated by the following relation: for functors 
 $\varphi,\psi : [m]\to [n]$ with $g\circ\varphi = f$ and
 $g\circ\psi=f$, $\varphi\sim\psi$ if and only if the morphism
 $g\left(\min\{\varphi(i),\psi(i)\}\le\max\{\varphi(i),\psi(i)\}\right)$ 
 in $C(g(\min\{\varphi(i),\psi(i)\}),g(\max\{\varphi(i),\psi(i)\}))$ 
 is an identity morphism in $C$ for any $i$ in $[m]$. 
\end{definition}

The description can be simplified for 
acyclic categories as follows.

\begin{lemma}
 \label{Sd_of_acyclic_category}
 Let $C$ be an acyclic small category. For $f,g\in \Sd(C)_0$, the
 set of morphisms $\Sd(C)(f,g)$ consists of a single point, if
 there exists $\varphi$ with $f=g\circ\varphi$, and an empty set
 otherwise. 

 Therefore $\Sd(C)$ is a poset.
\end{lemma}

\begin{proof}
 Since $C$ is acyclic, $C(x,x)=\{1_x\}$ for any objects $x\in C_0$. This
 implies that for
 $\varphi,\psi : [m]\to[n]$ with $f=g\circ\varphi=g\circ\psi$,
 $\varphi\sim\psi$ if and only if $g(\varphi(i))=g(\psi(i))$ for all
 $i\in [m]$. In other words, all elements in
 $\set{\varphi:[m]\to[n]}{g\circ\varphi=f}$ are equivalent to each
 other. Hence $\Sd(C)(f,g)$ is a single point if the above set is
 nonempty. 
\end{proof}

In order to compare $C(\Sd(X))$ and $\Sd(C(X))$ for a totally normal
stellar stratified space $X$, we need to understand
the cellular stratification on $\Sd(X)$. By Corollary
\ref{Sd_of_totally_normal}, we know that $\Sd(X)$ is a totally normal
cell complex when $X$ is a totally normal stellar stratified
space. Cells are parametrized by elements in $\overline{N}(C(X))$. Let
us denote the cell corresponding to $\bm{b}\in \overline{N}(C(X))$ by
$e_{\bm{b}}$. Cell structure maps are given as follows.

\begin{lemma}
 \label{characteristic_map_for_e_b}
 For each $k$, fix a homeomorphism $D^k\cong \Delta^k$.
 Let $X$ be a totally normal stellar stratified space. For
 $\bm{b}\in \overline{N}_k(C(X))$, the composition
 \[
  D^k \cong \Delta^k = B[k] \rarrow{B\bm{b}} BC(X) = \Sd(X)
 \]
 defines a cell structure on the cell corresponding to $\bm{b}$, where
 we regard $\bm{b}$ as a functor $\bm{b} : [k] \to C(X)$.
\end{lemma}

\begin{proof}
 The map $B\bm{b} : B[k]\to BC(X)$ is induced by the map
 $\overline{N}\bm{b} : \overline{N}([k])\to \overline{N}(C(X))$.
 As we have seen in the proof of Proposition \ref{stratification_on_Sd},
 a cell structure map on the cell corresponding to $\bm{b}$ is given by
 the composition 
 \[
  \Delta^k \cong \{\bm{b}\}\times\Delta^k \hookrightarrow \coprod_{k}
 \overline{N}_k(C(X))\times\Delta^k \rarrow{} \|\overline{N}(C(X))\|
 \cong BC(X).
 \]
 Since $\overline{N}_k([k])$ consists of a single point, the above
 composition can be identified with
 \[
 \Delta^k \cong \overline{N}_k([k])\times\Delta^k
 \rarrow{\overline{N}(\bm{b})\times 1} 
 \coprod_{k}
 \overline{N}_k(C(X))\times\Delta^k \rarrow{} \|\overline{N}(C(X))\|
 \cong BC(X)
 \]
 and the result follows.
\end{proof}

For simplicity, we use the standard simplices $\Delta^k$ as the domains
of cells in $\Sd(X)$. The cell structure map for $e_{\bm{b}}$ is
identified with $B\bm{b}$ by Lemma \ref{characteristic_map_for_e_b}.

\begin{theorem}
 \label{subdivision_is_subdivision}
 For any totally normal stellar stratified space $X$, we have an
 isomorphism of categories
 \[
  \Sd(C(X)) \cong C(\Sd(X)).
 \]
\end{theorem}

\begin{proof}
 By definition, objects in $\Sd(C(X))$ are elements of the nondegenerate 
 nerve of $C(X)$. On the other hand, objects in $C(\Sd(X))$ are in
 one-to-one correspondence with cells 
 in $\Sd(X)=BC(X)$. Under the stratification in Proposition
 \ref{stratification_on_Sd}, we obtain a bijection
 $C(\Sd(X))_0\cong\Sd(C(X))_0$. 

 For $\bm{b}\in \overline{N}_k(C(X))$ and
 $\bm{b}'\in\overline{N}_{m}(C(X))$, we have 
 \[
  C\left(\Sd(X);\pi_{\Sd(X)}\right)(e_{\bm{b}'},e_{\bm{b}}) =
 \set{f : \Delta^m\to\Delta^k}{B\bm{b}'=B\bm{b}\circ f}. 
 \]
 Since $B\bm{b}'|_{\Int(\Delta^m)}$ is injective, $f|_{\Int(\Delta^m)}$
 is also injective. The condition $B\bm{b}'=B\bm{b}\circ f$ implies that
 $f|_{\Int(\Delta^m)}$ is a PL map and hence $f$ is a PL map.
 Since $B\bm{b}|_{\Int(\Delta^k)}$ is injective, such a PL map is unique
 if it were to exist.
 It is given by $f=B\varphi$ for some poset map $\varphi : [m]\to[k]$.

 On the other hand, by Lemma \ref{Sd_of_acyclic_category},
 $\Sd(C(X))(\bm{b}',\bm{b})$ is
 nonempty (and a single point set) if and only if there exists a poset
 map $\varphi : [m]\to[k]$ with 
 $\bm{b}'=\bm{b}\circ\varphi$. Thus the classifying space functor $B(-)$
 induces an isomorphism of categories
 \[
  B : \Sd(C(X)) \longrightarrow C(\Sd(X)).
 \]
\end{proof}

\begin{remark}
 Note that we obtained an isomorphism of categories instead of an
 equivalence. Since $\Sd(C(X))$ is a poset, it implies that $C(\Sd(X))$
 is also a poset. Thus the barycentric subdivision $\Sd(X)$ of a totally
 normal stellar stratified space is a regular cell complex.
\end{remark}

%% file: quotient_map.tex
\section{Generalities on Quotient Maps}
\label{quotient_map}

In our definition of cell structures, we required the cell structure map
$\varphi : D \to \overline{e}$ of a cell $e$ to be a quotient map.
In order to perform operations on cellular stratified
spaces, such as taking products and subspaces, we need to understand
basic properties of quotient maps. 

\subsection{Definitions}

It is well-known that the quotient topology does not behave well with
respect to certain operations of topological spaces, such as taking
products and subspaces. We need to impose stronger conditions.

\begin{definition}
 \label{bi-quotient_definition}
 A surjective continuous map $f : X\to Y$ is called \emph{bi-quotient},
 if, for any $y\in Y$ and any open covering $\mathcal{U}$ of
 $f^{-1}(y)$, there exists finitely many $U_1,\ldots, U_k\in\mathcal{U}$
 such that $\bigcup_{i=1}^k f(U_i)$ contains a neighborhood of $y$ in
 $Y$. 
\end{definition}

Another important class of maps are hereditarily quotient maps.

\begin{definition}
 \label{hereditarily_quotient_definition}
 A surjective continuous map $f : X\to Y$ is called \emph{hereditarily
 quotient} if, for any $y\in Y$ and any neighborhood $U$ of $f^{-1}(y)$,
 $f(U)$ is a neighborhood of $y$.
\end{definition}

\subsection{Properties}

This quotient topology condition imposes some restrictions on the
topology of $\overline{e}$, especially when $e$ is closed.  
For example, $\overline{e}$ is metrizable for any closed cell $e$. A
proof can be found in a book \cite{Lundell-Weingram} by
Lundell and Weingram. Their proof can be modified to obtain the following
extension of this fact. 

\begin{lemma}
 \label{overlinee_is_metrizable}
 Suppose $\varphi : D\to \overline{e}\subset X$ is an $n$-cell structure
 with $\varphi^{-1}(y)$ compact for each $y\in \overline{e}$.
 Then $\overline{e}$ is metrizable. In particular,
 it is Hausdorff and paracompact\footnote{Theorem 41.4 in
 \cite{MunkresTopology2nd}}. 
\end{lemma}

\begin{proof}
 For $y,y'\in \overline{e}$, define 
 \[
 \overline{d}(y,y') = \min\lset{d(x,x')}{x\in \varphi^{-1}(y),
 x'\in\varphi^{-1}(y')}, 
 \]
 where $d$ is the metric on $D^n$.
 By assumption, $\varphi^{-1}(y)$ and $\varphi^{-1}(y')$ are compact and
 $\overline{d}(y,y')$ is defined. The compactness of $\varphi^{-1}(y)$
 and $\varphi^{-1}(y')$ also implies that $\overline{d}$ is a metric on
 $\overline{e}$. 

 Let us verify that the topology defined by $\overline{d}$ coincides
 with the quotient topology by $\varphi$. The continuity of $\varphi$
 with respect to the metric topologies on $D$ and $\overline{e}$ implies
 that open subsets in the $\overline{d}$-metric topology are open in the
 quotient topology. Conversely let $U$ be an open subset of
 $\overline{e}$ with respect to the quotient topology. We would like to
 show that, for each $y\in U$, there exists $\delta>0$ such that the
 open disk $U_{\delta}(y;\overline{d})$ around $y$ with radius $\delta$
 with respect to the metric $\overline{d}$ is contained in $U$. Let
 $\delta$ be a Lebesgue number of the open covering
 $\{\varphi^{-1}(U)\}$ of the 
 compact metric space $\varphi^{-1}(y)$. For
 $y'\in U_{\delta}(y;\overline{d})$, 
 there exist $x\in\varphi^{-1}(y)$ and $x'\in\varphi^{-1}(y')$ such that
 $d(x,x')<\delta$. Thus
 $x'\in U_{\delta}(x;d)\subset\varphi^{-1}(U)$ by the definition of
 Lebesgue number. Or $y'=\varphi(x')\in U$. And we have
 $U_{\delta}(y';\overline{d})\subset U$.
\end{proof}

\begin{definition}
 \label{relatively_compact_cell_definition}
 We say a cell structure $\varphi : D \to \overline{e}$ is
 \emph{relatively compact}
 if $\varphi^{-1}(y)$ is compact for each $y\in \overline{e}$. We also
 say that the cell $e$ is 
 \emph{relatively compact}. 
\end{definition}

In particular, when $\varphi : D\to\overline{e}$ is proper (i.e.\ closed
and each $\varphi^{-1}(y)$ is compact),
$\overline{e}$ is metrizable. On the other hand, the properness of
$\varphi$ implies that $\varphi$ is a bi-quotient map. 

It is straight forward to verify that a hereditarily quotient map can be
restricted freely.

\begin{lemma}
 \label{hereditarily_quotient_map_can_be_restricted}
 Any hereditarily quotient map $f : X\to Y$ is a quotient map. More
 generally, for any subspace $A\subset Y$, the restriction
 $f|_{f^{-1}(A)} : f^{-1}(A)\to A$ is hereditarily quotient, hence a
 quotient map.  
\end{lemma}

\begin{proof}
 Suppose $f$ is hereditarily quotient. For a
 subset $U\subset Y$, suppose that $f^{-1}(U)$ is open in
 $X$. For a point $y\in U$, $f^{-1}(U)$ is an neighborhood of
 $f^{-1}(y)$. Since $f$ is hereditarily quotient, $f(f^{-1}(U))=U$ is a
 neighborhood of $y$ in $Y$. Thus $y$ is an interior point of $U$ and it
 follows that $U$ is an open subset of $Y$.

 Since the definition of hereditarily quotient map is local,
 $f|_{f^{-1}(A)}$ is hereditarily quotient for any $A\subset Y$.
\end{proof}

\begin{remark}
 See also Arhangel$'$skii's paper \cite{Arhangelskii1963} for
 hereditarily quotient maps.
\end{remark}

\begin{lemma}
 \label{biquotient_is_quotient}
 Any bi-quotient map is hereditarily quotient. In particular, it is a
 quotient map. 
\end{lemma}

\begin{proof}
 By definition.
\end{proof}

Michael \cite{Michael1968-2} proved that bi-quotient maps are abundant.

\begin{lemma}
 \label{sufficient_conditions_for_bi-quotient}
 Any one of the following conditions implies that a map $f : X\to Y$ is
 bi-quotient:
 \begin{enumerate}
  \item $f$ is open.
  \item $f$ is proper.
  \item $f$ is hereditarily quotient and the boundary
	$\partial f^{-1}(y)$ of each fiber is compact.
 \end{enumerate}
\end{lemma}

\begin{proof}
 Proposition 3.2 in \cite{Michael1968-2}.
\end{proof}

%

Recall that a product of quotient maps may not be a quotient map.
There exist a space $X$ and a quotient map $f : Y\to Z$ such that
the product $1_X\times f : X\times Y\to X\times Z$ is 
not a quotient map. The following fact is a well-known result of
J.H.C.~Whitehead \cite{J.H.C.Whitehead1948}. 

\begin{lemma}
 For a locally compact Hausdorff space $X$, $1_X\times f$ is a quotient
 map for any quotient map $f$.
\end{lemma}

Unfortunately the domain of cell structures may not be locally compact.

\begin{example}
 \label{nonlocally_compact}
 $D^2-\{(1,0)\}$ is locally compact, while $\Int(D^2)\cup\{(1,0)\}$ is
 not locally compact.
\end{example}

The domain $D$ of an $n$-cell structure $\varphi : D\to \overline{e}$ is
often a stratified subspace of $D^n$ under a normal cell decomposition
of $D^n$. In other words, $D$ is obtained from $D^n$ by removing cells.
In the Example \ref{nonlocally_compact}, $D^2$ is regarded as a cell
complex $D^2=e^0\cup e^1\cup e^2$. $D^2\setminus e^0$ is locally
compact, while $D^2\setminus e^1$ is not. 
More generally we have the following criterion of locally compact
subspaces in a CW complex.
 
\begin{proposition}
 \label{complement_of_subcomplex}
 Let $X$ be a locally finite CW complex and $A$ be a subcomplex, then
 $X\setminus A$ is locally compact.
\end{proposition}

This is an immediate corollary to the following fact, which can be
found, for example, in Chapter XI of Dugundji's book
\cite{DugundjiTopology} as Theorem 6.5. 

\begin{lemma}
 Let $X$ be a locally compact Hausdorff space. A subspace $A\subset X$
 is locally compact if and only if there exist closed subsets
 $F_1,F_2\subset X$ with $A=F_2\setminus F_1$.
\end{lemma}

\begin{proof}[Proof of Proposition \ref{complement_of_subcomplex}]
 Since $X$ is locally finite, it is locally compact. The CW condition
 implies that $A$ is closed in $X$.
\end{proof}

Let us go back to the discussion on products of quotient maps. The main
motivation of Michael for introducing bi-quotient maps is that 
they behave well with respect to products.

\begin{proposition}
 \label{biquotients_are_closed_under_product}
 For any family of bi-quotient maps $\{f_i : X_i\to Y_i\}_{i\in I}$, the
 product
 \[
  \prod_{i\in I} f_i : \prod_{i\in I} X_i \longrightarrow \prod_{i\in I}
 Y_i 
 \]
 is a bi-quotient map.
\end{proposition}

\begin{proof}
 Theorem 1.2 in \cite{Michael1968-2}.
\end{proof}

The following property is also useful when we study cell structures.

\begin{lemma}
 \label{Lindeloef_fiber}
 Let $f : X \to Y$ be a quotient map. Suppose that $Y$ is first
 countable and Hausdorff and that, for each $y\in Y$,
 $\partial f^{-1}(y)$ is Lindel{\"o}f. Then $f$ is bi-quotient.
\end{lemma}

\begin{proof}
 Proposition 3.3(d) in \cite{Michael1968-2}.
\end{proof}


\begin{corollary}
 \label{compact_fiber}
 Let $\varphi : D \to \overline{e}$ be a relatively compact cell.
 Then $\varphi$ is bi-quotient.
\end{corollary}

\begin{proof}
 By Lemma \ref{overlinee_is_metrizable}, $\overline{e}$ is first
 countable and Hausdorff. By assumption each fiber $\varphi^{-1}(y)$ is
 compact and so is the boundary $\partial \varphi^{-1}(y)$.
 The result follows from Lemma \ref{Lindeloef_fiber}.
\end{proof}

%% file: PL.tex
\section{Simplicial Topology}
\label{PL}

In this second appendix, we recall basic definitions and theorems in PL
(Piecewise Linear)
topology and simplicial homotopy theory used in this paper. Our
references are
\begin{itemize}
 \item the book \cite{Rourke-SandersonBook} of Rourke and
       Sanderson for PL topology, and
 \item the book \cite{Goerss-Jardine} by Goerss and Jardine for
       simplicial homotopy theory.
\end{itemize}

\input{simplicial_complex}

\input{locally_cone-like_space}

\input{PL_map}

%% file: simplicial_complex.tex
\subsection{Simplicial Complexes, Simplicial Sets, and Simplicial
  Spaces} 
\label{simplicial_complex}

Let us fix notation and terminology for simplicial complexes first. 
Good references are Dwyer's monograph \cite{Dwyer-Henn} and Friedman's
survey article \cite{0809.4221}.

\begin{definition}
 For a set $V$, the power set of $V$ is denoted by $2^V$.
\end{definition}

\begin{definition}
 Let $V$ be a set. An \emph{abstract simplicial complex} on $V$ is a
 family of subsets $K \subset 2^V$ satisfying the following condition:
 \begin{itemize}
  \item $\sigma\in K$ and $\tau\subset \sigma$ imply $\tau\in K$.
 \end{itemize}

 $K$ is called \emph{finite} if $V$ is a finite set.
\end{definition}

\begin{definition}
 An \emph{ordered simplicial complex} $K$ is an abstract simplicial
 complex whose vertex set $P$ is partially ordered in such a way that 
 the induced ordering on each simplex is a total order.

 An $n$-simplex $\sigma\in K$ with vertices $v_0<\cdots< v_n$ is denoted
 by $\sigma = [v_0,\ldots,v_n]$.
\end{definition}

There are several ways to define the geometric realization of an
abstract simplicial complex.

\begin{definition}
 For an abstract simplicial complex $K$ with vertex set $V$, define a
 space $\|K\|$ by
 \[
  \|K\| = \set{f\in \Map^f(V,\R)}{\sum_{v\in \sigma}f(v)=1, f(v)\ge 0,
 \sigma\in K},
 \]
 where $\Map^f(V,\R)$ is the set of maps from $V$ to $\R$ whose values
 are $0$ except for a finite number of elements. It is equipped with the
 compact-open topology.
 The space $\|K\|$ is called the \emph{geometric realization} of $K$.
\end{definition}

\begin{lemma}
 Suppose the vertex set $V$ of an abstract simplicial complex $K$ is
 finite. Choose an embedding 
 \[
  i : V \hookrightarrow \R^{N}
 \]
 for a sufficiently large $N$ so that the $i(V)$ is affinely
 independent. Then we have a homeomorphism
 \[
 \|K\| \cong \set{\sum_{v\in V} a_v i(v)}{\sum_{v\in \sigma} a_v=1, a_v\ge
 0, \sigma\in K}. 
 \]
\end{lemma}

\begin{example}
 Consider $2^{V}\setminus\{\emptyset\}$ for
 $V=\{0,\ldots,n\}$. This is an abstract simplicial complex.
 Then we have a homeomorphism 
 \[
 \left\|2^{V} \setminus \{\emptyset\}\right\| \cong
 \set{(t_0,\ldots,t_n)\in\R^{n+1}}{\sum_{i}t_i=1,t_i\ge 0}
 = \Delta^n.
 \]
 $\Delta^n$ is a convex polytope having $(n+1)$ codimension $1$
 faces. Each codimension $1$ face can be realized as the image of the map
 \[
  d^i : \Delta^{n-1} \longrightarrow \Delta^{n}
 \]
 defined by
 \[
  d^i(t_0,\ldots,t_{n-1}) = (t_0,\ldots, t_{i-1},0,t_{i},\ldots,t_{n-1}).
 \]
 We also have maps
 \[
  s^i : \Delta^{n} \longrightarrow \Delta^{n-1}
 \]
 defined by
 \[
  s^i(t_0,\ldots,t_n) = (t_0,\ldots, t_{i}+t_{i+1},t_{i+2},\ldots, t_n).
 \]
\end{example}

For an ordered simplicial complex $K$, we may forget the ordering and
apply the above construction. However, there is another construction.

\begin{definition}
 For an ordered simplicial complex $K$ with vertex set $V$ Let $K_n$ be
 the set of $n$-simplices in $K$. Each element $\sigma$ in $K_n$ can be
 written as $\sigma=(v_0,\ldots,v_n)$ with $v_0<\cdots<v_n$. Under such
 an expression, define
 \[
  d_i(\sigma) = (v_0,\ldots, v_{i-1},v_{i+1},\ldots,v_n).
 \]
 Define
 \[
  \|K\| = \quotient{\coprod_n K_n\times \Delta^n}{\sim},
 \]
 where the relation $\sim$ is generated by
 \[
  (\sigma,d^i(\bm{t})) \sim (d_i(\sigma),\bm{t}).
 \]
 This is called the \emph{geometric realization} of $K$.
\end{definition}

\begin{lemma}
 For a finite ordered simplicial complex, the above two constructions of 
 the geometric realization coincide.
\end{lemma}

The above construction can be extended to simplicial sets and simplicial
spaces.

\begin{definition}
 A \emph{simplicial set} $X$ consists of
 \begin{itemize}
  \item a sequence of sets $X_0,X_1,\ldots$, 
  \item a family of maps $d_i : X_n \to X_{n-1}$ for $0\le i\le n$,
  \item a family of maps $s_i : X_n \to X_{n+1}$ for $0\le i\le n$
 \end{itemize}
 satisfying the following relations
  \[
  \begin{cases}
   d_i\circ d_j = d_{j-1}\circ d_i, & i<j \\
   d_i\circ s_j = s_{j-1}\circ d_i, & i<j \\
   d_j\circ s_j  = 1 = d_{j+1}\circ s_j \\
   d_i\circ s_j  = s_j\circ d_{i-1}, & i>j+1 \\
   s_i\circ s_j = s_{j+1}\circ s_i, & i\le j.
  \end{cases}
 \]
 The maps $d_i$'s and $s_j$'s are called \emph{face operators} and
 \emph{degeneracy operators}, respectively.

 When each $X_n$ is a topological space and maps $d_i, s_j$ are
 continuous, $X$ is called a \emph{simplicial space}.
\end{definition}

\begin{remark}
 It is well known that defining a simplicial set $X$ is equivalent to
 defining a functor
 \[
  X : \Delta^{\op} \longrightarrow \Sets,
 \]
 where $\Delta$ is the full subcategory of the category of posets
 consisting of $[n]=\{0<1<\cdots<n\}$ for $n=0,1,2,\ldots$.
\end{remark}

\begin{example}
 \label{singular_simplicial_space}
 For a topological space $X$, define
 \[
  S_n(X) = \Map(\Delta^n,X).
 \]
 The operators $d^i$ and $s^i$ on $\Delta^n$ induce
 \begin{eqnarray*}
  d_i & : & S_n(X) \longrightarrow S_{n-1}(X) \\
  s_i & : & S_{n-1}(X) \longrightarrow S_n(X).
 \end{eqnarray*}
 When each $S_n(X)$ is equipped with the compact-open topology, these
 maps are continuous and we obtain a simplicial space $S(X)$. This is
 called the \emph{singular simplicial space}.

 Usually $S_n(X)$'s are merely regarded as sets and $S(X)$ is regarded
 as a simplicial set, in which case we denote it by $S^{\delta}_n(X)$.
\end{example}

\begin{example}
 \label{ordered_simplicial_complex_as_simplicial_set}
 Let $K$ be an ordered simplicial complex on the vertex set $V$. Define 
 \[
 s(K)_n = \lset{[\underbrace{v_0,\ldots,v_0}_{i_0},
 \underbrace{v_1,\ldots,v_1}_{i_1}, \ldots,
 \underbrace{v_k,\ldots,v_k}_{i_k}]}{\sigma=[v_0,\ldots,v_k]\in K,
 \sum_{j=0}^k i_j=n}. 
 \]
 Then the collection $s(K)=\{s(K)_n\}$ becomes a simplicial set. This is
 called the \emph{simplicial set generated} by $K$.
\end{example}

\begin{definition}
 The \emph{geometric realization} of a simplicial space $X$ is defined by
 \[
  |X| = \quotient{\coprod_{n} X_n\times \Delta^n}{\sim}
 \]
 where the relation $\sim$ is generated by
 \begin{eqnarray*}
  (x,d^i(\bm{t})) & \sim & (d_i(x),\bm{t}), \\
  (x,s^i(\bm{t})) & \sim & (s_i(x),\bm{t}).
 \end{eqnarray*}
 The map induced by the evaluation maps
 \[
  \Delta^n\times S_n(X) \longrightarrow X
 \]
 is denoted by
 \[
  \ev : |S(X)| \longrightarrow X.
 \]
\end{definition}

Note that the geometric realization of an ordered simplicial complex is
defined only by using face operators.

\begin{definition}
 \label{Delta-set_definition}
 A \emph{$\Delta$-set} $X$ consists of
 \begin{itemize}
  \item a sequence of sets $X_0,X_1,\ldots$, and
  \item a family of maps $d_i : X_n\to X_{n-1}$ for $0\le i\le n$,
 \end{itemize}
 satisfying the following relations
 \[
 d_i\circ d_j = d_{j-1}\circ d_i,
 \]
 for $i<j$. 

 When each $X_n$ is equipped with a topology under which $d_i$'s are
 continuous, $X$ is called a \emph{$\Delta$-space}.
\end{definition}

\begin{remark}
 A $\Delta$-set $X$ can be identified with a functor
 \[
  X : \Delta_{\inj}^{\op} \longrightarrow \Sets,
 \]
 where $\Delta_{\inj}$ is the subcategory of $\Delta$ consisting of
 injective maps. In particular, any simplicial set can be regarded as a
 $\Delta$-set. 
\end{remark}

\begin{definition}
 The \emph{geometric realization} of a $\Delta$-space $X$ is defined by
 \[
  \|X\| = \quotient{\coprod_n X_n\times \Delta^n}{\sim},
 \]
 where the relation $\sim$ is generated by
 \[
  (x,d^i(\bm{t})) \sim (d_i(x),\bm{t}).
 \]
\end{definition}

\begin{remark}
 Note that any simplicial space $X$ can be regarded as a
 $\Delta$-space. However the geometric realization of $X$ as a
 $\Delta$-space, $\|X\|$, is much larger than that of $X$ as a
 simplicial space. $\|X\|$ is often called the fat realization.
\end{remark}

In order to study the homotopy type of simplicial complexes, and more
generally, regular cell complexes, the notion of regular neighborhood is
useful. Let us recall the definition.

\begin{definition}
 \label{star_definition}
 Let $K$ be a cell complex. For $x\in K$, define
 \[
  \St(x;K) = \bigcup_{x\in\overline{e}} e.
 \]
 This is called the \emph{open star} around $x$ in $K$. For a subset
 $A \subset K$, define 
 \[
  \St(A;K) = \bigcup_{x\in A} \St(x;K).
 \]
 When $K$ is a simplicial complex and $A$ is a subcomplex,
 $\St(A;K)$ is called the \emph{regular neighborhood} of $A$ in $K$. 
\end{definition}

The regular neighborhood of a subcomplex is often defined in terms of
vertices.

\begin{lemma}
 Let $A$ be a subcomplex of a simplicial complex $K$. Then
 \[
  \St(A;K) = \bigcup_{v\in \sk_0(A)} \St(v;K).
 \]
\end{lemma}

%% file: locally_cone-like_space.tex
\subsection{Locally Cone-like Spaces}
\label{locally_cone-like_space}

Let us first define the operation $\ast$ on subsets of a Euclidean space.

\begin{definition}
 \label{join_definition}
 For subspaces $P, Q \subset \R^n$, define
 \[
  P\ast Q = \lset{(1-t)p+tq}{p\in P, q\in Q, 0\le t\le 1}.
 \]
 This is called the \emph{convex sum} or \emph{join} of $P$ and $Q$

 When $P$ is a single point $v$, $v\ast Q$ is called the \emph{cone} on
 $Q$ with vertex $v$.
\end{definition}

\begin{remark}
 When $P$ and $Q$ are closed and ``in general position'', $P\ast Q$
 agrees with the join operation.
\end{remark}

\begin{definition}
 \label{locally_cone-like_space_definition}
 We say a subspace $P \subset \R^n$ is \emph{locally cone-like}, if, for
 any $a\in P$, there exists a compact subset $L\subset P$ such that the
 cone $a\ast L$ is a neighborhood of $a$.
\end{definition}

\begin{remark}
 Locally cone-like spaces are called polyhedra in the
 Rourke-Sanderson book \cite{Rourke-SandersonBook}. 
\end{remark}

\begin{example}
 The half space $\R^n_{+}=\lset{\bm{x}\in\R^n}{x_n\ge 0}$ is locally
 cone-like. 
\end{example}

\begin{example}
 A vector (or affine) subspace of $\R^n$ is locally cone-like.
\end{example}

\begin{example}
 A convex polytope in $\R^n$ is locally cone-like.
\end{example}

\begin{lemma}
 The class of locally cone-like spaces is closed under the following
 operations:
 \begin{itemize}
  \item finite intersections, 
  \item finite products, and
  \item locally finite unions.
 \end{itemize}
\end{lemma}

\begin{corollary}
 Euclidean polyhedral complexes are locally cone-like.
\end{corollary}

The following theorem characterizes compact locally cone-like spaces.

\begin{theorem}
 Any compact locally cone-like space can be expressed as a union of
 finite number of simplices.
\end{theorem}

%% file: PL_map.tex
\subsection{PL Maps Between Polyhedral Complexes}
\label{PL_map}

In PL topology, we study triangulated spaces. Following Rourke and
Sanderson, let us consider locally cone-like spaces. We also consider
polyhedral complexes in a Euclidean 
space.

\begin{definition}
 \label{PL_map_definition}
 Let $P$ and $Q$ be locally cone-like spaces. A map
 \[
  f : P \longrightarrow Q
 \]
 is said to be \emph{piecewise-linear (PL)} if, for each $a\in P$, there
 exists a cone neighborhood $N=a\ast L$ such that
 \[
  f(\lambda a+\mu x) = \lambda f(a)+\mu f(x)
 \]
 for $x\in L$ and $\lambda,\mu\ge 0$ and $\lambda+\mu=1$.
\end{definition}

\begin{lemma}
 PL maps are closed under the following operations:
 \begin{itemize}
  \item products,
  \item compositions,
  \item the cone construction.
 \end{itemize}
\end{lemma}

Another important property of PL maps is the extendability.

\begin{lemma}
 \label{PL_map_is_extendalbe}
 Let $P$ be a convex polytope and 
 \[
  f : \Int P \longrightarrow \R^n
 \]
 be a PL map. Then it has a PL extension
 \[
  \tilde{f} : P \longrightarrow \R^n.
 \]
\end{lemma}


\begin{theorem}
 \label{simplicial_subdivision_of_PL_map}
 Let $K$ and $L$ be Euclidean polyhedral complexes. For any PL map 
 \[
  f : K \longrightarrow L,
 \]
 there exist simplicial subdivisions $K'$ and $L'$ of $K$ and $L$,
 respectively, such that the induced map
 \[
  f : K' \longrightarrow L'
 \]
 is simplicial.
\end{theorem}

\begin{proof}
 By Theorem 2.14 in \cite{Rourke-SandersonBook}.
\end{proof}

%% file: topological_category.tex
\section{Topological Categories}
\label{topological_category}

In this third appendix, we recall basics of topological categories. Our
references are Segal's paper \cite{SegalBG} and the article by Dwyer in
\cite{Dwyer-Henn}. 

\input{topological_acyclic_category}
\input{nerve}

%% file: topological_acyclic_category.tex
\subsection{Topological Acyclic Categories}
\label{topological_acyclic_category}

\begin{definition}
 \label{quiver_definition}
 A \emph{topological quiver} $Q$ is a diagram of spaces of the form
 \[
  s,t : Q_1 \longrightarrow Q_0.
 \]
 For a topological quiver $Q$, define
 \[
 N_n(Q)= \lset{(u_n,\ldots,u_1)}{s(u_n)=t(u_{n-1}), \ldots,
 s(u_{2})=t(u_{1})}. 
 \]
 An element of $N_n(Q)$ is called an \emph{$n$-chain} of $Q$.
\end{definition}

\begin{definition}
 \label{topological_category_definition}
 A \emph{topological category} $C$ is a topological quiver equipped with
 two more maps
 \begin{eqnarray*}
  i & : & C_0 \longrightarrow C_1 \\
  \circ & : & N_2(C) \longrightarrow C_1
 \end{eqnarray*}
 making the following diagrams commutative
 \begin{enumerate}
  \item 
	\[
	 \begin{diagram}
	  \node{N_3(C)} \arrow{e,t}{\circ\times 1}
	  \arrow{s,l}{1\times\circ} 
	  \node{N_2(C)} \arrow{s,r}{\circ} \\ 
	  \node{N_2(C)} \arrow{e,b}{\circ} \node{C_1,}
	 \end{diagram}
	\]
  \item 
	\[
	 \begin{diagram}
	  \node{C_1} \arrow{se,=} \arrow{e,t}{i\times 1}
	  \node{N_2(C)} \arrow{s,r}{\circ} 
	  \node{C_1} \arrow{w,t}{1\times i} \arrow{sw,=} \\
	  \node{} \node{C_1}
	 \end{diagram}
	\]
 \end{enumerate}
 Elements of $C_0$ are called \emph{objects}. An element $u\in C_1$ with 
 $s(u)=x$ and $t(u)=y$ is called a \emph{morphism} from $x$ to $y$ and
 is denoted by $u : x \to y$. The subspace of morphisms from $x$ to $y$
 is denoted by $C(x,y)$, i.e.\ $C(x,y)= s^{-1}(x)\cap t^{-1}(y)$. For
 $x\in C_0$, $i(x)$ is called the \emph{identity morphism} on $x$ and is
 denoted by $1_x : x \to x$.

 When $C_0$ has the discrete topology, $C$ is called a
 \emph{top-enriched category}.
\end{definition}

\begin{definition}
 \label{topological_acyclic_category_definition}
 An \emph{acyclic topological category} is a top-enriched category $C$
 in which, for any pair of
 distinct objects $x,y\in C_0$, either $C(x,y)$ or $C(y,x)$ is empty and,
 for any object $x\in C_0$, $C(x,x)$ consists of the identity morphism.
\end{definition}

\begin{remark}
 When the topology of $C_{0}$ is not discrete, we need to assume that
 $C_1$ decomposes into a disjoint union of $i(C_0)$ and its complement.
\end{remark}

The following fact is well known.

\begin{lemma}
 For an acyclic topoloical category $C$, define a relation $\le$ on
 $C_0$ as follows: 
 \begin{quote}
  $x\le y$ if and only if $C(x,y)\neq\emptyset$.
 \end{quote}
 Then the relation $\le$ is a partial order on $C_0$.
\end{lemma}

\begin{definition}
 \label{underlying_poset}
 For an acyclic topological category $C$, the poset $(C_0,\le)$ is
 called the \emph{underlying poset} of $C$ and is denoted by $P(C)$. The
 canonical functor from $C$ to $P(C)$ is denoted by
 \[
 \pi_{C} : C \longrightarrow P(C).
 \]
\end{definition}

In this paper, we are concerned with cellular structures.

\begin{definition}
 \label{cellular_category_definition}
 A topological category $C$ is said to be \emph{cellular} if both $C_0$
 and $C_1$ have structures of cellular stratified spaces and the
 structure maps $s,t,i,\circ$ are morphisms of cellular stratified
 spaces. 
\end{definition}

\begin{remark}
 In the above definition, we assume that each $C(y,z)\times C(x,y)$ is a
 cellular stratified space under the product stratification. See
 Proposition \ref{relatively_compact_or_locally_polyhedral} and
 discussions in \S\ref{product} for conditions under which the
 bi-quotient assumption in Lemma \ref{cellular_stratification_on_BC} is
 satisfied. 

 When $C_0$ is not discrete it is not easy to define cellularness, since
 the pullback of stratified spaces over a stratified space 
 may not be a stratified space in general.
\end{remark}

Functors between topological categories are always assumed to be
continuous.

\begin{definition}
 A \emph{continuous functor} $f$ from a topological category $C$ to
 another $D$ is a pair 
 $f=(f_0,f_1)$ of continuous map
 \begin{eqnarray*}
  f_0 & : & C_0 \rarrow{} D_0 \\
  f_1 & : & C_1 \rarrow{} D_1
 \end{eqnarray*}
 that are compatible with all structure maps of topological category.
\end{definition}

Topological categories defined in Definition
\ref{topological_category_definition} are usually called small
topological categories, in the sense that collections of objects and
morphisms form sets. In this paper the only category which is not small
is the category $\Spaces$ of topological 
spaces and continuous maps. The continuity of a functor to
$\Spaces$ is defined as follows.

\begin{definition}
 \label{continuous_functor_to_Top}
 Let $C$ be a (small) top-enriched category. A functor
 $f: C\to \Spaces$ is said to be \emph{continuous} if, for each
 pair $x,y\in C_0$ of objects, the adjoint
 \[
  C(x,y)\times f(x) \rarrow{} f(y)
 \]
 to the map $f(x,y): C(x,y)\to \Map(f(x),f(y))$ is continuous.

 For a continuous functor $f: C\to \Spaces$,
 define $\tot(f)=\coprod_{x\in C_0} f(x)$. The canonical projection onto
 $C_0$ is denoted by $\pi_{f}: \tot(f)\to C_0$. Define
 $C_1\Box_{C_0}\tot(f)$ by the following pullback diagram
 \[
  \begin{diagram}
   \node{C_1\Box_{C_0}\tot(f)} \arrow{s} \arrow{e,t}{s\Box 1}
   \node{\tot(f)} 
   \arrow{s,r}{\pi_{f}} \\ 
   \node{C_1} \arrow{e,b}{s} \node{C_0.}
  \end{diagram}
 \]
 Then the \emph{colimit} of $f$, $\colim_{C} f$, is defined by the
 following coequalizer diagram 
 \[
 \xymatrix{
 C_{1}\Box_{C_0}\tot(f)
 \ar@<1ex>[r]^(.55){s\Box 1} \ar@<-1ex>[r]_(.55){\mu_{f}}
 & \tot(f) \ar[r]^{p_f} & \displaystyle\colim_{C} f,
 }
 \]
 where $\mu_{f}$ is defined on each component by the adjoint
 $C(x,y)\times f(x) \rarrow{} f(y)$ to $f(x,y): C(x,y)\to\Map(f(x),f(y))$.
\end{definition}

%% file: nerve.tex
\subsection{Nerves and Classifying Spaces}
\label{nerve}

One of the most fundamental facts is the collection
$N(X)=\{N_n(X)\}_{n\ge 0}$ defined in Definition \ref{quiver_definition}
forms a simplicial space.

\begin{lemma}
 For a topological category $X$, we have
 \[
  N_n(X) = \Funct([n],X),
 \]
 where the poset $[n]=\{0,\ldots,n\}$ is regarded as a topological
 category (with discrete topology). Thus $N(X)$ defines a functor
 \[
  N(X) : \Delta^{\op} \longrightarrow \Spaces.
 \]
 In other words, $N(X)$ is a simplicial space for any topological
 category $X$.
\end{lemma}

\begin{definition}
 \label{nerve_definition}
 The simplicial space $N(X)$ is called the \emph{nerve} of $X$.
 The source and the
 target maps on $X$ can be extended to
 \[
  s,t : N_k(X) \longrightarrow X_0
 \]
 by $s(\bm{f})=f(0)$ and $t(\bm{f})=f(k)$, respectively. These are also
 called the \emph{source and target maps}.
\end{definition}

\begin{definition}
 The geometric realization of $N(X)$ is called the \emph{classifying
 space} of $X$ and is denoted by $BX$.
\end{definition}

When we form the geometric realization of a simplicial space,
nondegenerate chains are essential.

\begin{definition}
 For a topological category $X$, define
 \[
  \overline{N}_n(X) = N_n(X)\setminus \bigcup_i s_i(N_{n-1}(X)).
 \]
 Elements of $\overline{N}_n(X)$ are called \emph{nondegenerate}
 $n$-chains. 
\end{definition}

\begin{lemma}
 When $P$ is a poset regarded as a small category,
 $\overline{N}(P) = \left\{\overline{N}_n(P)\right\}$ is an ordered
 simplicial complex and we have
 \[
  BP = \left\|\overline{N}(P)\right\|.
 \]
\end{lemma}

More generally, we have the following description.

\begin{lemma}
 \label{classifying_space_of_acyclic_category}
 When $X$ is a topological acyclic category, the simplicial structure on
 $N(X)$ can be restricted to give a
 structure of $\Delta$-space on $\overline{N}(X)$. Furthermore the
 composition
 \[
  \left\|\overline{N}(X)\right\| \rarrow{} \|N(X)\| \rarrow{} |N(X)| =
 BX 
 \]
 is a homeomorphism.
\end{lemma}

When an acyclic topological category $C$ is cellular in the sense of
Definition \ref{cellular_category_definition}, the classifying space
$BC$ has a canonical cell decomposition.

\begin{lemma}
 \label{cellular_stratification_on_BC}
 Let $C$ be an acyclic topological category in which each morphism space
 $C(x,y)$ is equipped with a structure of cellular stratified space
 whose cell structures are bi-quotient. Then the classifying space $BC$
 has a structure of cellular stratified space.

 In particular when all morphism spaces are cell complexes, the
 classifying space $BC$ has a structure of cell complex.
\end{lemma}

\begin{proof}
 Under the identification $BC \cong\left\|\overline{N}(C)\right\|$, any
 point in $BC$ can be represented by a pair
 $(x,t)\in \overline{N}_k(C)\times\Int(\Delta^k)$ uniquely. In other
 words, we have a decomposition
 \[
  BC = \coprod_{k=0}^{\infty} \overline{N}_k(C)\times \Int(\Delta^k)
 \]
 as sets.
 The bi-quotient assumption on cell structures on $C(x,y)$ implies that
 $\overline{N}_k(C)$ has a structure of cellular stratified
 space\footnote{See Definition \ref{product_cellular_stratification} for
 product cellular stratifications}. Thus the above decomposition and the
 cellular stratification on 
 each $\overline{N}_k(C)$ define a stratification on $BC$.

 For each cell
 $\varphi:D\to\overline{e}\subset\overline{N}_k(C)$, let us denote the
 cell in $BC$ corresponding to $e\times\Int(\Delta^k)$ by $e\times e^k$.
 The composition 
 \[
  D\times \Delta^k \rarrow{\varphi\times 1_{\Delta^k}}
 \overline{N}_k(C)\times\Delta^k \rarrow{}
 \left\|\overline{N}(C)\right\|=BC 
 \]
 defines a cell structure on $e\times e^k$.
\end{proof}

%% file: biblist.tex
\bibliography{%
preamble,%
mathA,%
preprintA,%
mathB,%
preprintB,%
mathC,%
preprintC,%
mathD,%
preprintD,%
mathE,%
preprintE,%
mathF,%
preprintF,%
mathG,%
preprintG,%
mathH,%
preprintH,%
mathI,%
preprintI,%
mathJ,%
preprintJ,%
mathK,%
preprintK,%
mathL,%
preprintL,%
mathM,%
preprintM,%
mathN,%
preprintN,%
mathO,%
preprintO,%
mathP,%
preprintP,%
mathQ,%
preprintQ,%
mathR,%
preprintR,%
mathS,%
preprintS,%
mathT,%
preprintT,%
mathU,%
preprintU,%
mathV,%
preprintV,%
mathW,%
preprintW,%
mathX,%
preprintX,%
mathY,%
preprintY,%
mathZ,%
preprintZ}